\documentclass{article}
\usepackage[utf8]{inputenc}

\usepackage{enumerate}
\usepackage{dsfont}
\usepackage{amssymb}
\usepackage{amsmath}
\usepackage{mathtools}
\usepackage{mathrsfs}
\usepackage{amsthm}
\usepackage{graphicx}
\usepackage{wrapfig}
\usepackage{lipsum}
\usepackage{float}
\usepackage[dvipsnames]{xcolor}
\usepackage{autobreak}
\usepackage{hyperref}
\usepackage{color}
\usepackage{todonotes}
\usepackage{soul}
\usepackage{authblk}
\usepackage{geometry}
\usepackage{chngcntr}
\usepackage{thmtools}
\usepackage{thm-restate}
\usepackage{bm}
\usepackage[maxnames=50]{biblatex}
\newtheorem{theorem}{Theorem}[section]
\newtheorem{definition}[theorem]{Definition}
\newtheorem{lemma}[theorem]{Lemma}
\newtheorem{corollary}[theorem]{Corollary}
\newtheorem{remark}[theorem]{Remark}

\newtheorem{proposition}[theorem]{Proposition}

\newtheorem{condition}[theorem]{Condition}

\counterwithin{theorem}{section}
\counterwithin{equation}{section}

\newcommand{\niton}{\not\owns}

\newcommand*\bgrg{\mathrm{BGRG}_n(\bm{w})}
\newcommand*\drig{\mathrm{DRIG}_n(\bm{w})}

\newcommand*\bgrgs{\mathrm{BGRG}_n^{s}(\bm{w})}
\newcommand*\drigs{\mathrm{DRIG}_n^{s}(\bm{w})}

\newcommand*\bgrgunion{\mathrm{BGRG}_n^{[0,t]}(\bm{w})}
\newcommand*\drigunion{\mathrm{DRIG}_n^{[0,t]}(\bm{w})}

\newcommand*\bcm{\mathrm{BCM}_n(\bm{d})}

\newcommand*\bgrgrescaled{\mathrm{BGRG}_n^{(t)}(\bm{w})}

\newcommand*\piaon{\pi^a_{\text{\rm{ON}}}}
\newcommand*\piaoff{\pi^a_{\text{\rm{OFF}}}}

\newcommand*\piaonunion{\pi^{a,[0,t]}_{\text{\rm{ON}}}}

\newcommand*\piaonrescaled{\pi^{a,(t)}_{\text{\rm{ON}}}}

\newcommand*\laon{\lambda^a_{\text{\rm{ON}}}}
\newcommand*\laoff{\lambda^a_{\text{\rm{OFF}}}}

\newcommand*\Var{\mathrm{Var}}

\newcommand{\longversion}[1]{#1}
\newcommand{\shortversion}[1]{}

\title{Dynamic random intersection graph:\\
Dynamic local convergence and giant structure}

\author[1]{Marta Milewska}
\author[2]{Remco van der Hofstad}
\author[1,2]{Bert Zwart}

\date{}

\affil[1]{Centrum Wiskunde \& Informatica (CWI), P.O. Box 94079, 1090 GB Amsterdam, The Netherlands}
\affil[2]{Department of Mathematics and Computer Science, Eindhoven University of Technology, 5600 MB Eindhoven, The Netherlands}

\begin{document}
\maketitle
\begin{abstract}
Random intersection graphs containing an underlying community structure are a popular choice for modelling real-world networks. Given the group memberships, the classical random intersection graph is obtained by connecting individuals when they share at least one group.\\
We extend this approach and make the communities \emph{dynamic} by letting them alternate between an active and inactive phase. We analyse the new model, delivering results on degree distribution, local convergence, giant component, and maximum group size, paying particular attention to the dynamic description of these properties. We also describe the connection between our model and the bipartite configuration model, which is of independent interest.
\end{abstract}


\section{Introduction and main results}
\subsection{Introduction}
Networks are present in many areas of everyday life. We distinguish for instance social networks, such as acquaintance networks, technological networks, such as the Worldwide Web, or biological networks such as neural networks (see \cite{Networks} for an extensive overview). Networks can be investigated in terms of many aspects, one of them being the existence and shape of local communities.

Local communities are smaller structures that have, on average, more connections than the network as a whole. They are crucial components of many real-life networks, such as social networks or the Internet, and they naturally give rise to high clustering (see \textup{\cite[Chapters 7 and 11]{NetworksBook}}) - one of the most fiercely investigated network's properties, that many systems seem to share.

There are many reasons why communities appear in networks. It might be because of a set of common features shared by a certain amount of individuals (for example the same nationality) or some underlying geometry (for example living in the same city). In our model, we intuitively associate communities with social groups that people can belong to, such as families, groups of friends, commuters on the same bus, etc. However, the model can also be relevant to other types of networks with similar structures.

Due to their complexity, real-world networks are often modelled with the help of \textit{random graphs}. There are many ways of implementing a community structure such as described above. A classic choice is the random intersection graph (RIG), first introduced in \cite{Singer_thesis}. The distinctive feature of intersection graph models is a double layer. We first establish a bipartite graph, with vertices on one side and communities on the other. After drawing edges between the two groups, we obtain a resulting graph by connecting two vertices if and only if they both belong to the same community. This procedure is called the one-node projection. Throughout the years, multiple suggestions on how to generate such bipartite graphs with group memberships appeared \cite{Bloznelis2015}. This includes pre-assigning the number of group memberships to every vertex and then connecting them to groups in a uniform manner (uniform RIG \cite{Blackburn2009, Rybarczyk2011} or generalized RIG \cite{Bloznelis2010, Bloznelis2013, Bloznelis2017, Bloznelis_Jaworski_Kurauskas2013, Godehardt_Jaworski}, generating group membership via the bipartite configuration model, i.e., assigning half-edges to individuals and groups and then connecting them uniformly at random \cite{Lelarge2015, Hofstad2022, Hofstad2018, Newman2003}, or performing independent percolation on the complete bipartite graph (binomial RIG \cite{Fill2000, Karonski1999, Singer_thesis} or inhomogeneous RIG \cite{Bloznelis_Damarackas2013, Deijfen2009}).

Our approach shares some similarities with the mentioned inhomogeneous RIG since we also assign weight $w_i$ to every vertex $i$ and the probability that vertex $i$ belongs to a certain community depends on this weight. However, we add a novel \textit{dynamic} factor by letting all communities go through active and inactive phases. We argue that such a modification is relevant for real-world networks since many of our regular social contacts are temporary. For instance, we usually do not spend all our days with our colleagues, and we only meet our closest friends a few times a week for a couple of hours. Other examples of temporary social interactions are concerts or rides with public transport. Note that our model is also a natural generalization of static graphs with communities, as the scenario in which all of the groups are present all the time is a special case.

The dynamics we implement in our model significantly differ from the well-investigated dynamic of graphs whose size evolves in time, such as preferential attachment models \cite{Barabasi1999, Hofstad2016, Simon1955, Yule1925}. The size of our graph is \emph{static}, but the connections between individuals and communities (and hence, in the resulting graph the connections between individuals themselves) keep evolving. This makes the model more similar to the graphs with dynamic bond percolation \cite{DynamicPercol, DynamicPercol2} or evolving configuration models \cite{AvenadHvdH2018,Avena2022}.\\

\noindent \textbf{Contribution of the paper.} In this paper, we investigate degree distribution, local convergence, and behavior of the giant connected component. A main innovation of our work is the methodology required to give a dynamic description of a bipartite generalized random graph. We investigate its dynamic local convergence and in particular, we introduce the concept of a dynamic giant component from the perspective of a uniformly chosen vertex, looking at it as a process in time. We then show in which way it is related to the dynamic local limit. To our best knowledge, the concept of dynamic local limit has only been discussed in the recent paper \cite{dynweaklimit2023} where the authors adopt a different approach and treat a different model. We also develop an auxiliary result on the equivalence between our model and the bipartite configuration model described in \cite{Hofstad2018}, that can be of independent interest \longversion{(see Appendix)}\shortversion{(see extended version \textup{\cite[Appendix]{Milewska2023}})}. This result can be thought of as a bipartite/community version of the equivalence between the configuration model and generalized random graph (see \textup{\cite[Theorem 7.18]{Hofstad2016}} and also Section \ref{sec_overview} in this paper).\\ 

\noindent \textbf{Outline of the paper.} We introduce our model and all necessary assumptions in Section \ref{sec_intro}. We state our main results in Section \ref{sec__main_results} and provide a discussion in Section \ref{sec_discussion}. We describe the overview of the proofs and state some secondary results in Section \ref{sec_overview}. In Section \ref{main_proofs} we prove the main results presented in Section \ref{sec__main_results}. For the conceptually straightforward proofs of the remaining results, we refer the reader \longversion{to the Appendix}\shortversion{to the extended version, see \textup{\cite[Appendix]{Milewska2023}}}.

\subsection{Model and notation} \label{sec_intro}
In this section, we intuitively and formally introduce the model and list some necessary assumptions that will hold throughout the paper. 

\noindent \textbf{Bipartite structure.} As mentioned in the introduction, there are two layers to our model: the underlying bipartite generalized random graph that we will call $\bgrg$, and the resulting dynamic random intersection graph that we will call $\drig$. Let $[n]_k$ denote the set of subsets of size $k$ of $[n]=\{1,...,n\}$. $\bgrg$ consists of the set of vertices $[n]$ on the one hand and the set of groups $\cup_{k\geq 2} [n]_{k}$ - a union of all $k$-element subsets of $[n]$ - with $k\geq2, k\in\mathbf{N}$ on the other. We intuitively think of them as left- and right-vertices respectively. To differentiate degrees in the underlying and the resulting graph, we write $d^{(l)}_i$ for the degree of vertex $i\in[n]$ in $\bgrg$ and $d^{(r)}_a$ for the degree of a group $a\in \cup_{k\geq 2}[n]_k$ in $\bgrg$. The connections are then formed by making the left-vertices connect to active right-vertices (groups).\\

In the following, we explain in detail how these connections are formed and with what probabilities, and what it means for a group $a\in \cup_{k\geq 2}[n]_k$ to be active. The left-degrees here denote the (potential) number of groups that vertices belong to and the right-degrees are the number of left-vertices connected to a particular right-vertex, i.e., the group sizes. Hence,
\begin{align} \label{bipartite_l_deg}
    d^{(l)}_i = \sum_{k=2}^{\infty} \sum_{a\in[n]_k: a \ni i} \mathds{1}_{\{a \hspace{0.1cm} \text{is ON}\}},
\end{align}
and, with $|a|$ denoting the size of group $a\in \cup_{k\geq 2}[n]_k$,
\begin{align} \label{bipartite_r_deg}
    d^{(r)}_a = |a|\cdot\mathds{1}_{\text{\{$a$ is \rm{ON}\}}}.
\end{align}

\noindent \noindent \textbf{Group memberships.} We quantify the number of all active groups as
\begin{align}
     \sum_{k=2}^{\infty} \sum_{a\in[n]_k} \mathds{1}_{\text{\{$a$ is \rm{ON}\}}},
\end{align}
and the number of all active groups a vertex $i\in[n]$ is a part of as
\begin{align}
    \sum_{k=2}^{\infty} \sum_{a\in[n]_k,a\ni i} \mathds{1}_{\text{\{$a$ is \rm{ON}\}}}.
\end{align}
As one can see, we conceptually deviate from the popular strategy of fixing upfront a $p_{ia}$ - a probability that a vertex $i$ is connected to a group $a$. Instead, we say that every combination of $k$ vertices, with $k\in[2,n]$, might potentially form a group and we define the probability that such a group is active, depending on the choice of vertices in question.\\

\noindent \textbf{Group dynamics.} Every group $a\in \cup_{k\geq 2}[n]_k$ will alternate between an ON and OFF state. Each group $a \in [n]_k$ behaves like a continuous-time Markov process with two states, which we refer to as ON and OFF. The holding times, i.e., the time that a group spends in each of them, are exponentially distributed with rates
\begin{align} \label{holding_times}
    \lambda^a_{\text{\rm{ON}}}=1 \hspace{1cm} \text{and} \hspace{1cm} \lambda^a_{\text{\rm{OFF}}}=\frac{f(|a|)\prod_{i\in a}w_i}{\ell_n^{|a|-1}},
\end{align}
respectively, where $\bm{w} = (w_i)_{i \in [n]}$ are the vertex weights, $\ell_n = \sum_{i\in[n]} w_i$ and $|a|$ denotes the size of a group $a\in [n]_k$. Naturally, $f(|a|)$ is a function of a group's size and can be chosen in a flexible way. Hence, the stationary distribution $\pi = [\pi_{\text{\rm{ON}}},\pi_{\text{\rm{OFF}}}]$ of these Markov chains is given by
\begin{align} \label{pi_ON}
    \pi^a_{\text{\rm{ON}}} = \frac{\laoff}{\laon+\laoff} = \frac{f(|a|)\prod_{i \in a}w_i}{\ell_n^{|a|-1}+f(|a|)\prod_{i \in a}w_i} \hspace{0.5cm} \text{and} \hspace{0.5cm} \pi^a_{\text{\rm{OFF}}} = \frac{\laon}{\laon+\laoff} = \frac{\ell_n^{|a|-1}}{\ell_n^{|a|-1}+f(|a|)\prod_{i \in a}w_i}.
\end{align}

\noindent \textbf{Assumptions on weights.} We first define what the empirical distribution of the weights is.
\begin{definition} [Empirical vertex weights distribution] We define the empirical distribution function of the vertex weights as
\begin{align}
    F_n(x) = \frac{1}{n} \sum_{i \in [n]} \mathds{1}_{\{w_i \leq x\}}, \hspace{0.2cm} \text{for} \hspace{0.2cm} x\geq 0.
\end{align}
\end{definition}
\noindent $F_n$ can be interpreted as the distribution of the weight of a uniformly chosen vertex. We denote the weight of such a uniformly chosen vertex $o_n$ in $[n]$ by $W_n=w_{o_n}$.
We impose the following conditions on the vertex weights:
\begin{condition} [Regularity condition for vertex weight] \label{cond_weights}
There exists a distribution function $F$ such that, as $n \to \infty$, the following conditions hold:
\begin{enumerate}[(a)]
    \item Weak convergence of vertex weight:
    \begin{align}
        W_n \stackrel{d}{\longrightarrow} W,
    \end{align}
    where $W_n$ and $W$ have distribution functions $F_n$ and $F$, respectively. Equivalently, for any $x$ for which $x \mapsto F(x)$ is continuous,
    \begin{align}
        \lim_{n\to\infty} F_n(x)=F(x).
    \end{align}
    \item Convergence of average vertex weight:
    \begin{align}
        \lim_{n\to\infty} \mathbf{E}[W_n] = \mathbf{E}[W],
    \end{align}
    where $W_n$ and $W$ have distribution functions $F_n$ and $F$, respectively. Further, we assume that $\mathbf{E}[W]>0$.
    \item Convergence of second moment of vertex weight:
    \begin{align}
        \lim_{n\to\infty} \mathbf{E}[W_n^2] = \mathbf{E}[W^2],
    \end{align}
\end{enumerate}
\end{condition}
\begin{remark}
For the time being, we assume Condition \ref{cond_weights}(c), however, the results treated in this paper are also true without it, which is very convenient for applications. We explain how Condition \ref{cond_weights}(c) can be lifted \longversion{in Remark \ref{appx_eliminating_cond_c}}\shortversion{in the extended version, see \textup{\cite[Remark B.11]{Milewska2023}}}.
\end{remark}

\noindent \textbf{Assumptions on the function of a group's size.}
We take $f(|a|) = |a|!p_{|a|}$, where $(p_{k})_{k\geq2}$ is the probability mass function of the group sizes. A particularly important case is a power-law group-size distribution where $p_k$ is close to $k^{-(\alpha+1)}$. We denote
\begin{align} \label{assump_1_mom}
    \mu = \sum_{k=2}^{\infty} kp_k,
\end{align}
and assume $\mu<\infty$. We also assume that the second moment of the group-size distribution is finite, so that $\alpha>2$, i.e.,
\begin{align} \label{assump_2_mom}
    \mu_{(2)} = \sum_{k=2}^{\infty} k^2p_k < \infty.
\end{align}

\noindent As explained below, these assumptions are necessary for the graph to be sparse, that is, the average degree remains uniformly bounded.\\

\noindent \textbf{Dynamic intersection graph.} The resulting dynamic random intersection graph $\drig$ consists of the set of vertices $[n]=\{1,...,n\}$. It is formed from $\bgrg$ by drawing an edge between two vertices $i,j\in[n]$ at time $s$ if they are in at least one active group together at time $s$. Hence, $\drig$ is a projection of $\bgrg$, the random multi-graph given by the edge multiplicities $\big(X(i,j)\big)_{i,j\in[n]}$ such that
\begin{align} \label{com_proj}
    X(i,j) = \sum_{k=2}^{\infty}\sum_{a\in[n]_k} \mathds{1}_{\{\text{$i$ in $a$}\}\cap\{\text{$j$ in $a$}\}}\mathds{1}_{\{\text{$a$ is \rm{ON}}\}}.
\end{align}
Let $d_i(s)$ denote the degree of a vertex $i\in[n]$ at time $s$ in $\drig$. Then
\begin{align}
    d_i(s) =\sum_{k=2}^{\infty} \sum_{a\in[n]_k: a \ni i} (|a|-1) \mathds{1}_{\{a \hspace{0.1cm} \text{is ON at time $s$}\}}.
\end{align}

\noindent \textbf{Stationary vs dynamic.} Due to the Markovian nature of our groups, $\bgrg$ and $\drig$ can be examined in two scenarios: in the long run, under the stationary distribution, and dynamically, for every time $s\in[0,t]$ with $t$ fixed, incorporating constant switching of the groups between the ON and OFF states. To make it clear which situation we are referring to, we will name the graphs resulting from the stationary distribution $\bgrg$ and $\drig$ and the dynamic graphs by $\big(\bgrgs\big)_{s\geq0}$ and $\big(\drigs\big)_{s\geq0}$.\\

\noindent\textbf{Uniformly chosen vertices.} Throughout the paper we often discuss results with respect to a uniformly chosen vertex. We denote a uniformly chosen vertex in the bipartite graph by $V^b_n$ and in the intersection graph by $o_n$. Note that this notation does not specify whether we refer to the stationary, dynamic, or union graphs. If such an indication is needed, then we will make it clear by stating an appropriate name of the graph, according to the notation introduced in the previous paragraph. We also denote the degree of a uniformly chosen vertex in $\bgrg$ by $D^{(l)}_n$ and the degree of a uniformly chosen group by $D^{(r)}_n$. We denote the degree of a uniformly chosen vertex in $\drig$ by $D_n$.

\subsection{Main results} \label{sec__main_results}
In this section, we state our main results. We first investigate the behavior of our model in stationarity. We start with the description of  \textbf{local convergence}, i.e., the convergence of the neighborhood counts. We explain this notion in more detail later in Section \ref{sec_locconv_theory}. We continue with the description of the giant component. Next, we proceed to analyse the dynamic situation. We state our results on dynamic local convergence and dynamic giant membership process. We close by discussing the dynamics of the largest group that is ON.

\subsubsection{Stationarity}
It turns out that under the stationary distribution and given appropriate conditions, our underlying graph - $\bgrg$ - is equivalent to the bipartite configuration model, $\bcm$, introduced in \cite{Hofstad2018} and investigated further in \cite{Hofstad2022}. Hence, it is possible to transfer the results on the local convergence and the giant component from \cite{Hofstad2018} and \cite{Hofstad2022} to our model. The link between the two models and the transfer of results are explained in detail in further sections. We now state its most important consequences: the results on the static local limit and giant component.\\

\noindent\textbf{Static local limit.}
Similarly, as in $\bcm$, the neighborhood of a uniformly chosen vertex in $\bgrg$ resembles a mixture of two branching processes, each of them corresponding to offspring distributions of left- and right-vertices of the root. Then, the neighborhood of a uniformly chosen vertex in $\drig$ resembles a community projection (see \ref{com_proj}) of this mixture. We summarise these statements in the following theorem, the limiting objects themselves are later explained in detail in Section \ref{sec_overview_loc_lim_static}:

\begin{restatable}[Local convergence of $\bgrg$ and $\drig$]{theorem}{locconvstatic}\label{loc_conv_static} Consider $\bgrg$ under Condition \ref{cond_weights}. As $n\to\infty$, $(\bgrg,V^b_n)$ converges locally in probability to $(\mathrm{BP}_{\gamma},0)$, where $(\mathrm{BP}_{\gamma},0)$ is a mixture of two branching processes.\\
Consequently for $\drig$ under Condition \ref{cond_weights}, as $n\to\infty$, $(\mathrm{DRIG}^{\pi}_n(\bm{w}),o_n) $ converges locally in probability to $(\mathrm{CP},o)$, where (CP, o) is a random rooted graph.
\end{restatable}
\noindent \longversion{We prove Theorem \ref{loc_conv_static} in Appendix \ref{appx_static_sec_loc_conv}.}\shortversion{The proof of Theorem \ref{loc_conv_static} is deferred to the extended version, see \textup{\cite[Appendix B.4]{Milewska2023}}.}\\

\noindent\textbf{Static giant component.}
Denote the giant component in $\drig$ by $\mathscr{C}_1$. The random variables $\Tilde{D}^{(l)}$ and $\Tilde{D}^{(r)}$ are strongly connected to the offspring distributions present in the local limit of the underlying $\bgrg$ and are explained in detail in Section \ref{sec_overview_loc_lim_static}. Denote $M_n = \#\{a\in[n]_k: \text{$a$ is \rm{ON}}\}$ and $A_k = \#\{a\in[n]_k:\text{$a$ is \rm{ON}}\}$. The following theorem gives a precise condition for the existence of the giant component in $\drig$.
\begin{restatable}[Giant component in $\drig$]{theorem}{giantstationary}\label{giant_stationary}
Consider $\drig$. There exists $\eta_l \in [0,1]$, the smallest solution of the fixed point equation
\begin{align} \label{giant_fixed_pt_eq}
    \eta_l = G_{\Tilde{D}^{(r)}}(G_{\Tilde{D}^{(l)}}(\eta_l)),
\end{align}
and $\xi_l= 1-G_{D^{(l)}}(\eta_l) \in[0,1]$ such that
\begin{align} \label{giant_intersect_conv_our}
    \frac{|\mathscr{C}_1|}{n} \stackrel{\mathbf{P}}{\longrightarrow} \xi_l.
\end{align}
Furthermore, $\xi_l>0$ exactly when
\begin{align} \label{giant_cond_our}
    \mathbf{E}[\Tilde{D}^{(l)}]\mathbf{E}[\Tilde{D}^{(r)}] = \frac{\mathbf{E}[W^2](\mu_{(2)}-\mu)}{\mathbf{E}[W]} > 1.
\end{align}
We call the above the \emph{supercritical case}. In this case, $\mathscr{C}_1$ is unique.
\end{restatable}

\noindent \longversion{We prove Theorem \ref{giant_stationary} in Appendix \ref{appx_static_sec_giant}}\shortversion{The proof of Theorem \ref{giant_stationary} is deferred to the extended version, see \textup{\cite[Appendix B.5]{Milewska2023}}}.

\subsubsection{Dynamic local convergence}
We start with the description of the local limit of our graph in the dynamic scenario. This concept is rather new in the study of the local convergence of graphs and to our best knowledge it has not been thoroughly investigated yet (see the discussion later on).\\

\noindent \textbf{The concept behind dynamic local convergence explained.} To examine the dynamic behavior, we first look collectively at everything that happens during the time interval $[0,t]$, for $t$ fixed, and then develop a method to withdraw information on the state of the graph for every $s\in[0,t]$. This yields one of our main results - the convergence of the dynamic intersection graph seen as a stochastic process in time:

\begin{restatable}[Dynamic local limit of $\drigs$]{theorem}{dynloclim}\label{dyn_loc_lim}
Under Condition \ref{cond_weights}, the dynamic intersection graph process $\big(\big(\mathrm{DRIG}^{s}_n(\bm{w}),o_n\big)\big)_{s\in[0,t]}$ converges in a local weak sense to $\big((\mathrm{CP}^{s},o)\big)_{s\in[0,t]}$, where, for every $s\in[0,t]$, $\big((\mathrm{CP}^{s},o)\big)_{s\in[0,t]}$ is a stochastic process of random rooted marked graphs.
\end{restatable}

\noindent We prove Theorem \ref{dyn_loc_lim} in Section \ref{sec_pf_dyn_loclim}.\\

\noindent \textbf{The limiting object.} Due to the one-node projection needed to obtain $\drigs$, the limiting object is also rather involved and we fully explain what it is later (see Section \ref{sec_overview}). Now we only provide an informal description. The limit $(\mathrm{CP}^{s},o)$ depends heavily on the limit of the marked bipartite union graph $\bgrgunion$, which is a mixture of two marked branching processes: one corresponding to the offspring distribution of the left partition and the other one to the offspring distribution of the right partition. The marks $\big(\sigma^a_{\rm{ON}},\sigma^a_{\rm{OFF}}\big)_a$ encoding activity times of the edges (details in Section \ref{sec_edge_marks_intro}) are preserved in the limiting object as the limiting marks $\big(t^a_{\rm{ON}},t^a_{\rm{OFF}}\big)_a$. Because of the way the intersection graph is constructed, the limit $(\mathrm{CP}^{[0,t]},\big((t^a_{\rm{ON}},t^a_{\rm{OFF}})\big),o)$ of the marked union intersection graph $\drigunion$ is then a community projection of the limit of the marked bipartite union graph, again preserving the edge marks. From there the limit of $\drigs$ can be constructed for every $s\in[0,t]$ from $(\mathrm{CP}^{[0,t]},\big((t^a_{\rm{ON}},t^a_{\rm{OFF}})\big),o)$, taking only the groups for which $s\in[t^a_{\rm{ON}},t^a_{\rm{OFF}}]$.\\

\noindent \textbf{Dynamic degree.}
As mentioned earlier, if we know what happens to the graph during the time period  $[0,t]$, we can extract the information on the degree of $i\in[n]$ for every $s\in[0,t]$.\\

\noindent \textbf{Degree in $\drigs$.} For every $i\in[n]$,
\begin{align}
    d_i(s) = \sum_{a\in\cup_{k\geq 2}[n]_k:a\ni i,\text{$a$ ON in $[0,t]$}} (|a|-1) \mathds{1}_{\big\{s\in[\sigma^a_{\rm{ON}},\sigma^a_{\rm{OFF}}]\big\}}.
\end{align}
Denote the number of groups of size $k$ containing a vertex $i$ at time $s\in[0,t]$ by $C_k^s(i)$. Since the degree distribution is a functional of the local limit we deduce the following convergence of the degree process:
\begin{corollary} \label{degdistr_static_cor}
As $n\to\infty$,
\begin{align}
    \big(D^s_n\big)_{s\in[0,t]} \stackrel{\text{d}}{\longrightarrow} \bigg(\sum_{k=2}^{\infty} (k-1)C_k^s(o_n)\bigg)_{s\in[0,t]}.
\end{align}
\end{corollary}

\subsubsection{Dynamic giant component}
Denote the giant component in $\big(\drigs\big)_{s\geq0}$ at time $s$ by $\mathscr{C}^s_1$ and the connected component of the root $o$ at time $s$ in the limiting structure $(\mathrm{CP}^{s},o)_{s\geq0}$ (see Theorem \ref{dyn_loc_lim}) by $\mathscr{C}^s(o)$. We examine the behavior of the process $J_n(s)=\mathds{1}_{\big\{o_n \in \mathscr{C}^s_1\big\}}$. From the behavior of the static giant (see Theorem \ref{giant_stationary}), for all $s$,
\begin{align}
    J_n(s) \stackrel{\text{d}}{\longrightarrow } \mathds{1}_{\big\{|\mathscr{C}^s(o)|=\infty\big\}}.
\end{align}
However, we want to investigate the behavior of $J_n(s)$ as a stochastic process in time. It turns out that this question can be linked to local neighborhoods in our graph and answered thanks to local convergence:

\begin{restatable}[Dynamic giant component]{theorem}{giantdynamic}\label{giant_dynamic}
As $n\to\infty$,
\begin{align}
    \Big(\mathds{1}_{\big\{o_n \in \mathscr{C}^s_1\big\}} \Big)_{s\in[0,t]} \stackrel{\text{d}}{\longrightarrow} \Big(\mathds{1}_{\big\{|\mathscr{C}^s(o)|=\infty\big\}} \Big)_{s\in[0,t]}
\end{align}
in the Skorokhod $J_1$ topology.
\end{restatable}
\noindent We provide the proof of Theorem \ref{giant_dynamic} in Section \ref{sec_pf_dyn_giant}.

\subsubsection{Maximal group size}
Define
\begin{align}
    K_{\max}^{\{0\}} = \max_{a\in\cup_{k\geq 2}[n]_k:\text{$a$ is \rm{ON} at time $0$}} |a|,
\end{align}
and $K^{(0,t]}_{\max}$ is the maximum group size in the set of groups that switch ON in the time interval $(0,t]$, respectively. We define
\begin{align}
    K^{[0,t]}_{\max} = \max\{K^{\{0\}}_{\max},K^{(0,t]}_{\max}\},
\end{align}
with $K^{\{0\}}_{\max}$ and $K^{(0,t]}_{\max}$ independent, as a consequence of the fact that different groups arrive independently. We further define $K^{(s,t]}_{\max}$ for every $s\in[0,t)$ as the maximum group size in the set of groups that switched ON in time interval $(s,t]$. Hence, for every $s\in[0,t)$,
\begin{align}\label{max_gr_size_split}
    K^{(0,t]}_{\max} = \max\{K^{(0,s]}_{\max},K^{(s,t]}_{\max}\},
\end{align}
and $K^{(0,s]}_{\max}, K^{(s,t]}_{\max}$ are independent. Now define $\big( \kappa^{(0,t]}_{\max} \big)_{t\geq0}$ such that, for every $t\geq0$,
\begin{align}
    \mathbf{P}(\kappa^{(0,t]}_{\max}\leq k) = e^{-tk^{-\alpha}\mathbf{E}[W]},
\end{align}
and the evolution of the whole process is such that for every partition $\{0,s_1,s_2,...,s_{t-1},s_t\}$ of the time interval $[0,s_t]$ it holds
\begin{align}
    \kappa^{(0,s_t]}_{\max}=\max\{\kappa^{(0,s_1]}_{\max},\kappa^{(s_1,s_2]}_{\max},...,\kappa^{(s_{t-1},s_t]}_{\max}\},
\end{align}
where, for non-overlapping time intervals $(s_1,t_1]$ and $(s_2,t_2]$, $\kappa_{\max}^{(s_1,t_1]}$, $\kappa_{\max}^{(s_2,t_2]}$ are independent. We show that the largest group size converges in distribution as a stochastic process to a limiting process $\big( \kappa^{(0,t]}_{\max} \big)_{t\geq0}$ inheriting a similar structure:
\begin{restatable}[Maximum group size]{theorem}{kmax} \label{k_max}
If the group-size distribution is a power law, i.e., $\sum_{l\geq k}p_l \propto k^{-\alpha}$, with $\alpha>3$, then, as $n\to\infty$,
\begin{align}
    \bigg( \frac{K^{[0,t]}_{\max}}{n^{1/\alpha}} \bigg)_{t\geq0} \stackrel{\text{d}}{\longrightarrow} \bigg( \kappa^{(0,t]}_{\max} \bigg)_{t\geq0},
\end{align}
in the Skorokhod $J_1$ topology.
\end{restatable}
\noindent We prove Theorem \ref{k_max} in Section \ref{sec_pf_kmax}.

\begin{remark}
Note that in the case of the maximal group present in the union graph, we do not run into the same problem of creating connections that do not exist as was the case with degrees of vertices. Hence, the largest group ever active in the union graph is at the same time the largest group ever active in the dynamic graph.
\end{remark}

\begin{remark}
Note that one could also investigate a slightly different dynamic process, namely $\big(n^{-1/\alpha}K^{s}_{\max}\big)_{s\in[0,t]}$, where $K^{s}_{\max}$ is the maximal group that is ON at time $s$. We conjecture that a proof of convergence of such a process should be closely related to the proof of Theorem \ref{k_max} that we present in the next section. However, the dynamics of $\big(n^{-1/\alpha}K^{s}_{\max}\big)_{s\in[0,t]}$ are significantly more involved: when the so far largest group switches OFF, then the largest ON group becomes the previously second largest one. Hence, to analyse such a process it would be necessary to keep track of the large ON groups as a process of infinite length.    
\end{remark}

\subsection{Discussion} \label{sec_discussion}
In this section, we comment on the advantages and limitations of our model and results and present open problems.\\

\noindent \textbf{Dynamic factor.} An undeniably interesting feature of our model is the introduction of temporary connections between vertices. Intuitively, it makes sense that a dynamic model should describe real-world networks more accurately. It would undoubtedly be interesting to investigate the truthfulness of this statement by comparing our model with other popular models through simulations. The dynamic we introduce is also interesting purely from the perspective of random graph theory. The framework for the dynamic local limit we provide is an alternative to the one in \cite{dynweaklimit2023}. However, the limit itself is a rather complex construction and its interpretation is not very straightforward.\\

\noindent \textbf{Other work on bipartite graphs with group structure.} To derive results on the static $\bgrg$, which later on lead to results on $\drig$, we heavily rely on \cite{Hofstad2022} and \cite{Hofstad2018}. The authors of these papers derive statements on local convergence and giant component in the bipartite configuration model with communities, which can be transferred to our model by showing an appropriate relation between the bipartite configuration model conditioned on simplicity and the bipartite generalized random graph conditioned on its degree sequence. We provide more details in Section \ref{sec_overview}.\\

\noindent \textbf{Other work on dynamic local convergence.} To our best knowledge, the only paper that treats dynamic local convergence is a recent pre-print by Dort and Jacob \cite{dynweaklimit2023}. The approach of the authors is different from ours: we are taking a classic approach by treating the dynamic graph as a stochastic process on the space of rooted graphs with a traditional local metric, whereas they define a metric that incorporates time. Contrary to \cite{dynweaklimit2023}, we also consider the dynamic giant component.\\

\noindent \textbf{Group sizes.} One of the beneficial features of our model is the existence of big groups (see Theorem \ref{k_max}). Such groups are an important factor in real-world social networks, which are highly clustered. We set the model in a way that allows for flexibility in the choice of the group-size distribution, as one can consider various $(p_k)_{k\geq 2}$, not necessarily heavy-tailed ones. We are aware that many of our proof techniques require $p_k$ to have a finite second moment, which might not seem ideal. However, the finite second moment is needed to obtain a sparse graph, i.e., a graph with a bounded average degree.\\

\noindent \textbf{Choice of parameters and alternative interpretation.} The model is quite flexible in the choice of parameters, as we do not determine the weight variables or the group-size distribution and only keep some general assumptions about them. However, the model does not allow for the one-to-one transfer of degree distribution from real-world data, as opposed to the bipartite configuration model with communities (from \cite{Hofstad2018}). Moreover, at first glance, choosing the stationary distribution as in (\ref{pi_ON}) might seem un-intuitive. However,  there is a nice intuitive description of this model. Our model is very closely connected to a Poisson process dynamic: take a situation, where we form a new group according to a Poisson process with intensity $\ell_n$. When a group is formed, we choose its size $k\geq2$ according to some size distribution $p_k$. Lastly, from all the groups of the chosen size, we pick the one to appear proportionally to products of weights, i.e., with probability
\begin{align*}
    \frac{\prod_{i\in a}w_i}{\sum_{b\in[n]_k}\prod_{j\in b}w_j},
\end{align*}
for all $a\in[n]_k$. It can be shown that our model and the model we just described yield the same degree sequences and hence, produce very similar graphs. However, our model conditioned on its degree sequence has the advantage of being uniform over all bipartite graphs with such degree sequence, which proves to be a very useful feature.\\

\noindent \textbf{Equivalence with the bipartite configuration model.} An interesting byproduct of our paper is a relationship between the underlying bipartite structure in our model and the bipartite configuration model. More mathematical details regarding this equivalence are stated in the next section and \longversion{in the Appendix}\shortversion{in the extended version, see \textup{\cite[Appendix]{Milewska2023}}}. This raises the question of whether the dynamic versions of these models are also equivalent.\\

\noindent \textbf{Relationship to $\mathrm{GRG}$.} Note that by taking $p_2=1$ we can obtain the classic generalized random graph $\mathrm{GRG}_n(\bm{w})$ from our model. Hence, our results also apply to the generalized random graph.

\section{Overview of the proofs} \label{sec_overview}

In this section, we provide the ideas behind the proofs of our main results. We include shorter and straightforward proofs. We also state the auxiliary results that we think might be of independent interest.

\subsection{Bipartite generalized random graphs and configuration models are equivalent} \label{section_two_models_relation}
Two of the most popular approaches to modelling real-world networks are the generalized random graph and the configuration model. The generalized random graph denoted $\mathrm{GRG}$, was introduced in 2006 by Britton, Deijfen and Martin-L{\"o}f (see \cite{Britton2006}). In this model, each vertex $i\in[n]$ is given a weight $w_i$ and the probability that there is an edge between vertex $i$ and vertex $j$ is equal to
\begin{align*}
    p_{ij} = \frac{w_iw_j}{\ell_n+w_iw_j},
\end{align*}
with $\ell_n=\sum_{i\in[n]}w_i$, just like in our model. Naturally, assigning edge probabilities according to weights can also be done in a different way. For a more general version see \cite{Bollobas_2007}, for related models see the Chung-Lu model (for instance \cite{Chung-Lu}) or the Norros-Reitu model \cite{Norros_2006}. For an overview of results on the classic generalized random graph see \textup{\cite[Chapter 6]{Hofstad2016}}.\\

In contrast, in the configuration model (CM), the degrees of the vertices are fixed upfront. The concept of the configuration model originates in the early works of Bollob{\'{a}}s (see \cite{Bollobas_1980}). Since then, various configuration models have been proposed but in its most standard formulation, the configuration model refers to a uniform pairing of half-edges, which can be represented in a form of a graph by assigning an appropriately determined number of half-edges to every node and then connecting them uniformly to form edges. A graph obtained in this way is uniform over the space of all graphs with a given degree sequence \cite{NetworksBook}. Such a model was popularised and intensively studied by Molloy and Reed (see \cite{Molloy_Reed1995}, \cite{Molloy_Reed1998}). For an overview of results for the classic configuration model see \textup{\cite[Chapter 7]{Hofstad2016}}. Again, there are many modifications of the classic configuration model such as the configuration model with households \cite{Trapman2007}.\\

Despite their differences, it turns out that under certain conditions the generalized random graph and the configuration model are equivalent (see \textup{\cite[Theorem 7.18]{Hofstad2016}}). Also note that the static $\bgrg$ introduced by us can be perceived as a bipartite, multi-dimensional version of the generalized random graph (hence the name $\bgrg$), where also certain communities are present. It turns out that this model is accordingly equivalent to a bipartite configuration model with communities $\bcm$, introduced and studied in \cite{Hofstad2018} and \cite{Hofstad2022}, under the same conditions that guarantee equivalence of the classic generalized random graph and the configuration model. More precisely, it turns out that, under such conditions, both $\bgrg$ and $\bcm$ from \cite{Hofstad2018} are uniform random graphs, and hence have the same distribution.\\

This relationship is one of the most important building blocks in the proof of our results. Thanks to it, we can transfer results proved in \cite{Hofstad2022, Hofstad2018} to our graph. As these auxiliary statements can be of independent interest, we present them below:

\begin{restatable}[$\bgrg$ conditioned on degree sequence is uniform]{theorem}{ourgraphuniform}\label{our_graph_uniform}
$\bgrg$ conditioned on $\{d^{(l)}_i(X) = d^{(l)}_i \hspace{0.1cm} \forall i \in [n], d^{(r)}_a(X) = d^{(r)}_a \hspace{0.1cm} \forall a\in\cup_{k\geq2}[n]_k\}$, is uniform over all bipartite graphs with degree sequence $(\bm{d}^{(l)},\bm{d}^{(r)})=\big((d^{(l)}_i)_{i\in[n]},(d^{(r)}_a)_{a\in\cup_{k\geq2}[n]_k}\big)$.
\end{restatable}

An equivalent result follows for the $\bcm$ conditioned on simplicity:

\begin{restatable}[$\bcm$ conditioned on being simple is uniform]{theorem}{BCMuniform} \label{BCM_uniform} For any degree sequence $\bm{d}=(\bm{d}^{(l)},\bm{d}^{(r)})=\big((d^{(l)}_i)_{i\in[n]},(d^{(r)}_a)_{a\in\cup_{k\geq2}[n]_k}\big)$, and conditionally on the event \{$\bcm$ is a simple graph\}, $\bcm$ is a uniform bipartite graph with degree sequence $d$.
\end{restatable}
As a natural consequence, $\bgrg$ conditioned on its degree and $\bcm$ conditioned on simplicity have the same distribution. We properly state this result in the following theorem, however, we first note that under some extra assumptions, an even stronger connection between the two graphs can be shown. As this connection plays a crucial role in many of our proofs, we include it in the theorem. The mentioned assumptions are as follows:
\begin{restatable}[Regularity conditions]{condition}{regcond}\label{reg_cond}
    The random variables $D^{(l)}_n$ and $D^{(r)}_n$ have distribution function $F^{(l)}_n$ and $F^{(r)}_n$ respectively, given by
\begin{align}
    F^{(l)}_n(x) = \frac{1}{n} \sum_{i\in[n]} \mathds{1}_{\{d^{(l)}_i \leq x\}} \hspace{0.5cm} \text{and} \hspace{0.5cm} F^{(r)}_n(x) = \frac{1}{n} \sum_{a\in[n]_k} \mathds{1}_{\{d^{(r)}_a \leq x\}}.
\end{align}
We impose the following assumptions on these distribution functions:
\begin{enumerate}[a)]
    \item There exist random variables $D^{(l)}$, $D^{(r)}$ such that, as $n\to\infty$ and for every $l\geq0$, $k\geq2$, \label{cond_left_deg}
    \begin{align}
        \mathbf{P}(D^{(l)}_n=l\mid  G_n) \stackrel{\mathbf{P}}{\longrightarrow} \mathbf{P}(D^{(l)}=l) \hspace{0.2cm} \text{and} \hspace{0.2cm} \mathbf{P}(D^{(r)}_n=k\mid  G_n) \stackrel{\mathbf{P}}{\longrightarrow} \mathbf{P}(D^{(r)}=k),
    \end{align}
    where $(\cdot \mid  G_n)$ denotes conditioning with respect to a realisation of a random graph. Moreover, $\mathbf{E}[D^{(l)}]<\infty$, $\mathbf{E}[D^{(r)}]<\infty$ and, as $n\to\infty$,
    \begin{align}
        \mathbf{E}[D^{(l)}_n \mid  G_n] \stackrel{\mathbf{P}}{\longrightarrow} \mathbf{E}[D^{(l)}] \hspace{0.2cm} \text{and} \hspace{0.2cm}  \mathbf{E}[D^{(r)}_n\mid  G_n] \stackrel{\mathbf{P}}{\longrightarrow} \mathbf{E}[D^{(r)}].
    \end{align}
    Additionally, we put a constraint on the second moment of degree random variables:
    \item $\mathbf{E}\big[\big(D^{(l)}\big)^2\big],\mathbf{E}\big[\big(D^{(r)}\big)^2\big]<\infty$ and, as $n\to\infty$, \label{cond_left_deg_2nd_mom}
    \begin{align}
        \mathbf{E}\big[\big(D^{(l)}_n\big)^2\mid  G_n\big] \stackrel{\mathbf{P}}{\longrightarrow} \mathbf{E}\big[\big(D^{(l)}\big)^2\big] \hspace{0.2cm} \text{and} \hspace{0.2cm} \mathbf{E}\big[\big(D^{(r)}_n\big)^2 \mid  G_n\big] \stackrel{\mathbf{P}}{\longrightarrow} \mathbf{E}\big[\big(D^{(r)}\big)^2\big].
    \end{align}
\end{enumerate}
\end{restatable}
With all the above in mind, the following result relates $\bgrg$ and $\bcm$:
\begin{restatable}[Relation between $\bgrg$ and $\bcm$]{theorem}{twomodelsrelation} \label{two_models_relation}
Let $d^{(l)}_i$ be the degree of left-vertex $i$ in $\bgrg$, $d^{(r)}_a$ the degree of a group $a$ in $\bgrg$, and $D = (\bm{d}^{(l)},\bm{d}^{(r)}) = \big((d^{(l)}_i)_{i\in[n]}, (d^{(r)}_a)_{a\in[m]}\big)$. Then,
\begin{align} \label{two_models_distr_equality}
    \mathbf{P}(\bgrg = G \mid  \bm{D}=\bm{d}) = \mathbf{P}(\bcm = G \mid  \bcm \hspace{0.1cm} \text{simple}).
\end{align}
Let $\mathcal{E}_n$ be a subset of multi-graphs such that $\mathbf{P}(\bcm \in \mathcal{E}_n) \stackrel{\mathbf{P}}{\longrightarrow}1$ when $D$ satisfies Condition \ref{reg_cond}. Assume that the degree sequence $D$ of $\bgrg$ satisfies Condition \ref{reg_cond}. Then also $\mathbf{P}(\bgrg \in \mathcal{E}_n)\longrightarrow1$.
\end{restatable}
The above theorem is a bipartite equivalent of the relationship between the classic generalised random graph $\mathrm{GRG}_n(\bm{w})$ and the configuration model $\mathrm{CM}_n(\bm{d})$ (see for instance \textup{\cite[Theorem 7.18]{Hofstad2016}}). Hence, the idea behind the proof is also similar. We first show that $\bgrg$ conditioned on the degree sequence is uniform. After that, we establish that $\bcm$ conditioned on simplicity is also uniform. Finally, we use both statements to prove the desired result. The proofs of the auxiliary steps as well as of the final Theorem \ref{two_models_relation} \longversion{can be found in Appendix \ref{appendix_two_models_relation}}\shortversion{are deferred to the extended version, see \textup{\cite[Appendix A]{Milewska2023}}}.\\

\subsection{Static local limit and giant component} \label{sec__main_results_stationary}

In this section, we investigate the local convergence of our graph under the stationary distribution. We introduce and describe in more detail the limiting local objects of the underlying $\bgrg$ and of the resulting $\drig$. Further on, we examine the proportion of vertices that are in the giant connected component. We state a phase transition in the size of the largest component in terms of the model parameters and give the explicit criterion under which a unique giant component exists.

\subsubsection{Brief overview of local convergence} \label{sec_locconv_theory}

Before stating our results we briefly define local convergence in probability and local marked convergence in probability. Local convergence was introduced in \cite{Benjamini2001} and a few years later, independently, in \cite{Aldous2004}. It describes the resemblance of the neighborhood of a vertex chosen uniformly at random to a certain limiting graph. To formalise this resemblance we introduce the notion of neighborhood and isomorphism on graphs:

\begin{definition}[Rooted graph, rooted isomorphism and $r$-neighborhood] \label{rooted_graph_def}
\begin{enumerate} [(i)]
    \item We call a pair $(G, o)$ a rooted graph if $G$ is a locally finite, connected graph and $o$ is a distinguished vertex of $G$. We denote the space of rooted graphs by $\mathcal{G}_{\star}$.
    \item We say that the rooted graphs $(G_1, o_1)$, $(G_2, o_2)$ are rooted isomorphic if there exists a graph-isomorphism between $G_1$ and $G_2$ that maps $o_1$ to $o_2$. We denote this isomorphism of rooted graphs by $(G_1,o_1)\simeq(G_2,o_2)$.
    \item For $r\in\mathbf{N}$, we define $B_r(G,o)$, the (closed) $r$-ball around $o$ in $G$ or $r$-neighborhood of $o$ in $G$, as the subgraph of $G$ spanned by all vertices of graph distance at most $r$ from $o$. We think of $B_r(G,o)$ as a rooted graph with root $o$.
\end{enumerate}
\end{definition}
The notion of graph isomorphism enables defining a metric on the space of rooted graphs:
\begin{definition}(Metric on rooted graphs) \label{metric_rooted}
Let $(G_1,o_1)$ and $(G_2,o_2)$ be two rooted connected graphs, and write $B_r(G_i,o_i)$ for the neighborhood of vertex $o_i\in V(G_i)$. Let
\begin{align}
    R^{\star}=\sup\{r:B_r(G_1,o_1) \simeq B_r(G_2,o_2)\},
\end{align}
and define
\begin{align}
    \text{d}_{\mathcal{G}_{\star}}\big((G_1,o_1),(G_2,o_2) \big) = \frac{1}{R^{\star}+1}.
\end{align}
\end{definition}
The space $\mathcal{G}_{\star}$ of rooted graphs under the metric $\text{d}_{\mathcal{G}_{\star}}$ is separable and thus Polish (for proof see \textup{\cite[Appendix A]{Hofstad2023}}), which will later prove very useful in the case of dynamic local convergence. We now define local convergence:

\begin{definition}(Local convergence in probability) \label{loc_conv_probab}  Let $G_n=(\mathcal{V}_n,\mathcal{E}_n)$ with size $|\mathcal{V}_n| \stackrel{\mathbf{P}}{\to} \infty$ be a sequence of random graphs, and let $o_n\mid  G_n \sim Unif[\mathcal{V}(G_n)]$. Let $(G,o)$ denote a random element (with arbitrary distribution) of the set of rooted graphs, which we call a random rooted graph. We say that $(G_n,o_n)$ converges locally in probability to $(G,o)$, if for any fixed rooted graph $(H_{\star},o')$ and $r\in\mathbf{N}$,
\begin{align} \label{def_loc_weak_conv}
    \mathbf{P}(B_r(G_n,o_n) \simeq (H_{\star},o')\mid  G_n):=&\frac{1}{|G_n|} \sum_{i\in\mathcal{V}(G_n)} \mathds{1}_{\{B_r(G_n,i) \simeq (H_{\star},o')\}}\\
    \stackrel{\mathbf{P}}{\longrightarrow} \hspace{0.1cm}& \mathbf{P}(B_r(G,o) \simeq (H_{\star},o')). \nonumber
\end{align}
We say that $(G,o)$ is the local limit in probability of $(G_n)_{n\geq1}$.
\end{definition}

\noindent Thus, intuitively, local convergence is defined as the convergence of the proportion of vertices whose neighborhoods have some specified structure. For further reading about local convergence, see for instance \textup{\cite[Chapter 2-5]{Hofstad2023}} and the references therein for examples of local limits of various graph models. Since we will actually need a more general setting of marked graphs and their convergence, we now briefly present some of the theory behind it.\\ 

\noindent \textbf{Marked graphs and marked local convergence.} Marks allow us to include additional information about vertices and/or edges such as directions, colours, and so on. In particular, we use marks to indicate the belonging of a vertex to a certain partition (of left- or right-vertices) in the underlying $\bgrg$ and to denote the on and off times of edges in the dynamic graphs, based on the group activity.

\begin{definition}[Marked graphs]
Let $\mathcal{G}$ denote the set of all locally finite (multi)graphs on a countable (finite or countably infinite) vertex set. A marked (multi)graph is a (multi)graph $G =(V(G), E(G))$, $G\in\mathcal{G}$, together with a set $\mathcal{M}(G)$ of marks taking values in a complete separable metric space $\Xi$, called the mark space. $\mathcal{M}$ maps from $V(G)$ and $E(G)$ to $\Xi$. Images in $\Xi$ are called marks. Each edge is given two marks, one associated with (‘at’)
each of its endpoints, in particular, for $v \in V(G)$, $\Xi(v) \in \mathcal{M}$, and for $e \in E(G)$, $\Xi(e) \in \mathcal{M}^2$. Moreover, $\mathcal{M}$ contains the special symbol $\varnothing$ which is to be interpreted as “no mark”. We denote the set of graphs with marks from the mark set $\mathcal{M}$ by $\mathcal{G}(\mathcal{M})$.
\end{definition}

We generalize Definitions \ref{rooted_graph_def} and \ref{metric_rooted} to the setting of marked rooted graphs:

\begin{definition}[Rooted marked graph and $r$-neighborhood.]
\begin{enumerate}[(i)]
    \item We choose a vertex $o$ in a marked graph $(G,\mathcal{M}(G))$ to be distinguished as the root. We denote the rooted marked graph by $(G,\mathcal{M}(G),o)$.\\
    We also denote the set of rooted marked graphs by $\mathcal{G}_{\star}(\mathcal{M})$. We call a random element of $\mathcal{G}_{\star}(\mathcal{M})$ (with an arbitrary joint distribution) a random rooted marked graph.
    \item The (closed) ball $B_r(G,\mathcal{M}(G),o)$ can be defined analogously to the unmarked graph ball (Definition \ref{rooted_graph_def} (iii)), by restricting the mark function to the subgraph as well.
\end{enumerate}
\end{definition}

\begin{definition}[Metric on marked rooted graphs] \label{metric_marked_graphs}
Let $\text{d}_{\Xi}$ be a metric on the space of marks $\Xi$. Let
\begin{align}
    R^{\star}= \sup\{&r:B_r(G_1,o_1) \simeq B_r(G_2,o_2),\hspace{0.2cm} \text{and there exists $\phi$ such that}\\
    &\text{d}_{\Xi}((m_1(i),m_2(\phi(i))) \leq 1/r \hspace{0.15cm} \forall i\in V(B_r(G_1,o_1)),\nonumber\\
    &\text{d}_{\Xi}(m_1(i,j),m_2(\phi(i,j))) \leq 1/r \hspace{0.15cm} \forall \{i,j\}\in E(B_r(G_1,o_1)) \nonumber
    \},
\end{align}
with $\phi:V(B_r(G_1,o_1))\longrightarrow V(B_r(G_2,o_2))$ running over all isomorphisms between $B_r(G_1,o_1)$ and $B_r(G_2,o_2)$, that map $o_1$ to $o_2$. Then define
\begin{align}
    \text{d}_{\mathcal{G}_{\star}}\big((G_1,\mathcal{M}(G_1),o_1),(G_2,\mathcal{M}(G_2),o_2) \big) = \frac{1}{R^{\star}+1}.
\end{align}
This turns $\mathcal{G}_{\star}(\mathcal{M})$ into a Polish space, i.e., a complete, separable metric space.
\end{definition}

Definition \ref{metric_marked_graphs} puts a metric structure on marked rooted graphs. With this metric topology in hand, we can simply adapt all convergence statements to this setting. Hence, we generalize Definition \ref{def_loc_weak_conv} as follows:

\begin{definition}[Local convergence in probability of marked graphs with continuous marks]
Let $(G_n,\mathcal{M}(G_n))_{n\in\mathbf{N}}$, $(G_n,\mathcal{M}(G_n))\in\mathcal{G}_{\star}(\mathcal{M})$ be a sequence of (finite) random marked graphs such that $|G_n| \stackrel{\mathbf{P}}{\to} \infty$ and let $o_n\sim Unif[\mathcal{V}(G_n)]$. Let $\mathbf{P}\big(\cdot\mid (G_n,\mathcal{M}(G_n))\big)$ denote conditional probability with respect to the marked graph ($o_n$ is the free variable). We say that $(G_n,\mathcal{M}(G_n),o_n)_{n\in\mathbf{N}}$ converges locally in probability to a (possibly) random element $(G,\mathcal{M}(G),o)\in\mathcal{G}_{\star}(\mathcal{M})$ if for any fixed rooted graph $(H_{\star},\mathcal{M}(H_{\star}),o')$ and $r\in\mathbf{N}$, as $n\to\infty$,
\begin{align} \label{def_loc_conv_marked}
    \mathbf{P}&\bigg(\text{d}_{\mathcal{G}_{\star}}\bigg((G_n,\mathcal{M}(G_n),o_n),(H_{\star},\mathcal{M}(H_{\star}),o') \bigg) \leq \frac{1}{r+1}
    \mid (G_n,\mathcal{M}(G_n))\bigg)\\
    :=&\frac{1}{|G_n|} \sum_{i\in\mathcal{V}(G_n)} \mathds{1}_{\big\{\text{d}_{\mathcal{G}_{\star}}\big((G_n,\mathcal{M}(G_n),i),(H_{\star},\mathcal{M}(H_{\star}),o') \big) \leq \frac{1}{r+1}\big\}}\stackrel{\mathbf{P}}{\longrightarrow} \hspace{0.1cm} \mathbf{P}\bigg(\text{d}_{\mathcal{G}_{\star}}\bigg((G,\mathcal{M}(G),o),(H_{\star},\mathcal{M}(H_{\star}),o') \bigg) \leq \frac{1}{r+1}\bigg). \nonumber
\end{align}
We say that $(G,\mathcal{M}(G),o)$ is the local limit in probability of $(G_n,\mathcal{M}(G_n))_{n\geq1}$.
\end{definition}

\subsubsection{Static local convergence of the dynamic random intersection graph} \label{sec_overview_loc_lim_static}
To prove static local convergence of $\drig$ we first look at static local convergence of $\bgrg$. The authors of \cite{Hofstad2018} and \cite{Hofstad2022} derive results on the local convergence and the giant component of $\bcm$ under the assumption that $\bcm$ fulfils Condition \ref{reg_cond}(i). From the previous section we know that under Condition \ref{reg_cond}(i)-(ii), results applying to $\bcm$ also apply to $\bgrg$. Hence, it suffices to show that our model under stationarity and assuming Condition \ref{cond_weights} fulfils Condition \ref{reg_cond}(i) and argue that Condition \ref{reg_cond}(ii) is in fact not necessary to obtain local convergence and the giant component in $\bgrg$.\\
Due to the one-node projection present in \cite{Hofstad2018} and in our model, the statements about the resulting intersection graphs, i.e, of $\mathrm{RIGC}$ and $\drig$, automatically follow from results shown for $\bcm$ and $\bgrg$ respectively. Since verification of the regularity conditions is quite elementary and the remaining results follow directly from \cite{Hofstad2018} and \cite{Hofstad2022}, the proofs of all static results can be found in \longversion{Appendix \ref{appx_static_graph}}\shortversion{the extended version, see \textup{\cite[Appendix B]{Milewska2023}}}. Here we only describe the limiting object and state the theorems.\\

\noindent \textbf{The static limiting object $(\mathrm{BP}_{\gamma},0)$.} We start by introducing $(\mathrm{BP}_{\gamma},0)$, the local limit in probability of $\bgrg$. Naturally, as we are dealing with two types of vertices - the left and the right ones - a typical neighborhood in this graph will be different depending on the type of the root. However, it is not possible to determine whether a uniformly chosen root was a left- or a right-vertex just on the basis of its neighborhood. Hence, we introduce marks to keep track of different types of vertices. Let $\mathcal{M}_b=\{l,r\}$ be the set of marks. We mark left-vertices as $l$ and right-vertices as $r$. Formally,
\begin{align}
    \mathcal{E}^b_n=
    \begin{cases}
             l \hspace{0.2cm} \text{if} \hspace{0.2cm} i\in [n],\\
             r \hspace{0.2cm} \text{if} \hspace{0.2cm} a\in [n]_{k\geq2}.
    \end{cases}
\end{align}
Now we introduce the limiting object $(\mathrm{BP}_{\gamma},\mathcal{E}_{\gamma},0)$, equipped with the mark function $\mathcal{E}^b_n$, while $(\mathrm{BP}_{\gamma}, 0)$ is then obtained by ignoring the mark function. Define a mixing variable ${\gamma}$ as
\begin{align}
    \mathbf{P}({\gamma}=l)=\frac{1}{1+\frac{M_n}{n}} \hspace{0.3cm} \text{and} \hspace{0.3cm} \mathbf{P}({\gamma}=r)=\frac{\frac{M_n}{n}}{1+\frac{M_n}{n}}.
\end{align}
Then, $(\mathrm{BP}_{\gamma},\mathcal{E}_{\gamma},0)$ is a mixture of two marked ordered BP-trees, $(\mathrm{BP}_l,\mathcal{E}_l,0)$ and $(\mathrm{BP}_r,\mathcal{E}_r,0)$:
\begin{align}
    (\mathrm{BP}_{\gamma},\mathcal{E}_{\gamma},0) \stackrel{d}{=} \mathds{1}_{\{\gamma=l\}}(\mathrm{BP}_l,\mathcal{E}_l,0)+\mathds{1}_{\{\gamma=r\}}(\mathrm{BP}_r,\mathcal{E}_r,0),
\end{align}
where $(\mathrm{BP}_l,\mathcal{E}_l,0)$ describes the neighborhood of a left-vertex and $(\mathrm{BP}_r,\mathcal{E}_r,0)$ of a right one. Hence,
\begin{align}
    (\bgrg,\mathcal{E}_l,V_n^{(l)}) \stackrel{\mathbf{P}}{\to} (\mathrm{BP}_l,\mathcal{E}_l,0) \hspace{0.3cm} \text{and} \hspace{0.3cm} (\bgrg,\mathcal{E}_r,V_n^{(r)}) \stackrel{\mathbf{P}}{\to} (\mathrm{BP}_r,\mathcal{E}_r,0),
\end{align}
where $V_n^{(l)}$ and $V_n^{(r)}$ denote vertices chosen uniformly from the set of all left- and right-vertices respectively. The mixing variable $\gamma$ can thus be re-interpreted as the random mark of the root.\\

Before we proceed, we need to introduce the size-biased version of a random variable:
\begin{definition} \label{def_shift_variable}
For an $\mathbb{N}$-valued random variable $X$ with $\mathbf{E}[X] < \infty$, we define its size-biased distribution $X^{\star}$ and
the shift variable $\Tilde{X}$ by their probability mass functions, for all $k\in\mathbb{N}$,
\begin{align}
    \mathbf{P}(X^{\star}=k)=\frac{k\mathbf{P}(X=k)}{\mathbf{E}[X]} \hspace{1cm} \text{and} \hspace{1cm} \mathbf{P}(\Tilde{X}=k)=\mathbf{P}(X^{\star}-1=k).
\end{align}
\end{definition}
Now we can continue with the description of the random ordered marked tree $(\mathrm{BP}_l,\mathcal{E}_l,0)$ itself. We consider a discrete-time branching process where the offspring of any two individuals are independent. We then give the individuals in even and odd generations marks $l$ and $r$, respectively. Generation 0 contains
the root alone and the root's offspring distribution is $D^{(l)}$ (the limit of the degree of a uniformly chosen left-vertex, see Condition \ref{reg_cond}). In consecutive generations, the offspring distribution of individuals marked with $l$ will be $\Tilde{D}^{(l)}$ and of individuals marked with $r$ will be $\Tilde{D}^{(r)}$. $(\mathrm{BP}_r,\mathcal{E}_r,0)$ is defined analogously with reversed roles of $l$ and $r$.\\

\noindent \textbf{Static local limit of $\drig$.} Having specified the local limit of the underlying $\bgrg$, we proceed to the limit of the resulting graph $\drig$.\\

\noindent \textbf{The static limiting object $(\mathrm{CP},o)$.} The limit that we denote by $(\mathrm{CP},o)$ is a random rooted graph and the “community projection” (see (\ref{com_proj})) of $(\mathrm{BP}_{\gamma}, \mathcal{E}_{\gamma},0)$ in the same way that $\drig$ is the “community projection” of the underlying $\bgrg$: it extracts only vertices marked as $l$ and builds links between these which were previously connected to the same vertex with mark $r$. Let us accentuate that even though the final limit is not a tree, it relies on the tree-like structure of the underlying $\bgrg$. This constructs the local limit $(\mathrm{CP},o)$ of $\drig$.

\subsubsection{Degree distribution}
The average degree in $\drig$ is asymptotically a sum of rescaled Poisson variables whose rates depend on the limiting weight variable $W$ (see Condition \ref{cond_weights}):

\begin{restatable}[Convergence of degree of a random vertex in $\drig$]{corollary}{convunifv} \label{conv_unif_v}
Let $\bm{w}$ satisfy Condition \ref{cond_weights}. Then,
\begin{align}
    D_n \stackrel{d}{\longrightarrow} \sum_{l\geq2} (l-1)X_l,
\end{align}
where $X_l$ is a mixed-Poisson variable with mixing distribution $lp_lW$.
\end{restatable}
Corollary \ref{conv_unif_v} is a direct consequence of the static local convergence and we can actually prove a stronger result about the degree distribution in the static graph. Define
\begin{align}
    Q_k^{(n)} = \frac{1}{n} \sum_{i\in[n]}\mathds{1}_{\{d_i=k\}}.
\end{align}
Then, the following theorem shows that $\big(Q_k^{(n)}\big)_{k\geq0}$ converges in total variation distance:
\begin{restatable}[Degree sequence in $\drig$]{theorem}{degseqstat} \label{deg_seq_stat}
For every $\varepsilon>0$,
\begin{align}
    \mathbf{P}\Big(\sum_{k=0}^{\infty}|Q^{(n)}_k- q_k|>\varepsilon \Big) \longrightarrow 0,
\end{align}
where $q_k=\mathbf{P}\big(\sum_{l\geq 2}(l-1)X_l=k\big)$ with $\big(X_l\big)_{l\geq2}$ - independent mixed-Poisson variables with mixing distribution $lp_lW$, i.e., such that
\begin{align}
    \mathbf{P}(X_l=k) = \mathbf{E}\Big[e^{-lp_lW}\frac{(lp_lW)^k}{k!}\Big].
\end{align}
\end{restatable}
\noindent \longversion{Theorem \ref{deg_seq_stat} is proven in Appendix \ref{appx_static_intersection}}\shortversion{The proof of Theorem \ref{deg_seq_stat} is deferred to the extended version, see \textup{\cite[Appendix B.3]{Milewska2023}}}.\\

We next investigate the sparsity of our model by investigating the average degree:
\begin{restatable}[Convergence of average degree in $\drig$]{theorem}{firstmomunifv} \label{1st_mom_unif_v}
As $n \to \infty$,
\begin{align}
    \mathbf{E}[D_n\mid  G_n] \stackrel{\mathbf{P}}{\longrightarrow} (\mu_{(2)}-\mu)\mathbf{E}[W].
\end{align}
\end{restatable}
The proof of this result \longversion{(see Appendix \ref{appx_static_intersection})}\shortversion{(see the extended version \textup{\cite[Appendix B.3]{Milewska2023}})} shows why Conditions \ref{assump_1_mom} and \ref{assump_2_mom} are necessary for the sparsity of our model.

\subsubsection{Static giant component}
Since real-world networks tend to be highly connected and a large fraction of individuals very often lies in a single connected component, it is useful to study the behavior of this component. We denote the  cluster or connected component of a vertex $i\in[n]$ in the graph $G=([n],E)$ by $\mathscr{C}(i)$. We denote the graph distance in $G$, i.e., the minimal number of edges in a path linking $i$ and $j$, by $\text{dist}_G(i,j)$. We define

\begin{definition}[Giant connected component]
\begin{align}
    \mathscr{C}(i) = \{j\in[n]: \text{dist}_G(i,j)<\infty\}.
\end{align}
Let $\mathscr{C}_1$ denote the largest connected component, i.e., let $\mathscr{C}_1$ satisfy
\begin{align}
    |\mathscr{C}_1| = \max_{i\in[n]}|\mathscr{C}(i)|,
\end{align}
where $|\mathscr{C}(i)|$ denotes the number of vertices in $\mathscr{C}(i)$ and we break ties arbitrarily.
\end{definition}
Of course, it can happen that there are two or more maximal clusters in a graph. For that reason, the uniqueness of the giant is often investigated. Another popular question is the existence of a component containing a linear proportion of vertices - the so-called giant component problem. It was first studied by Erd{\H o}s and R{\'e}nyi (\cite{ER_1960}) and has since been investigated on multiple other models (for instance the Chung-Lu model \cite{ChungLu2002, ChungLu2006}, or configuration model \cite{Bollobas2015, JansonLuczak2009, Molloy_Reed1995, Molloy_Reed1998}).\\

Due to the structure of intersection graphs, the giant component exists when it exists in the underlying bipartite graphs. Hence, the results on the giant component in $\drig$ follow from the results on the giant component in $\bgrg$. Similarly, as in the case of local convergence, thanks to the link between our model under stationarity and the $\bcm$ and the fact that regularity conditions we impose on the weights variables imply regularity conditions of degrees in $\bgrg$, we are allowed to transfer the statements on the giant component for $\bcm$ and $\mathrm{RIGC}$, proven in \cite{Hofstad2022}. We again state the results and \longversion{prove them in more detail in Appendix \ref{appx_static_sec_giant}}\shortversion{defer the details to the extended version (see \textup{\cite[Appendix B.5]{Milewska2023}}}).\\

\noindent \textbf{Static giant component in the $\bgrg$.} We start with the bipartite graph. Denote the giant component in $\bgrg$ by $\mathscr{C}_{1,b}$. Its giant is studied in the next theorem:

\begin{restatable}[Giant component in $\bgrg$]{theorem}{giantbipartite} \label{giant_bipartite}
Under the supercriticality condition $\mathbf{E}[\Tilde{D}^{(l)}]\mathbf{E}[\Tilde{D}^{(r)}] > 1$, as $n\to\infty$,
\begin{align}
    \frac{|\mathscr{C}_{1,b} \cap \mathcal{V}^{(l)}|}{n} \stackrel{\mathbf{P}}{\longrightarrow} \xi_l,
\end{align}
where $\xi_l= 1-G_{D^{(l)}}(\eta_l) \in[0,1]$ and $\eta_l \in [0,1]$ is the smallest solution of the fixed point equation
\begin{align}
    \eta_l = G_{\Tilde{D}^{(r)}}(G_{\Tilde{D}^{(l)}}(\eta_l)).
\end{align}
\end{restatable}

\noindent \longversion{We prove Theorem \ref{giant_bipartite} in Appendix \ref{appx_static_sec_giant}}\shortversion{The proof of Theorem \ref{giant_bipartite} is deferred to the extended version. See \textup{\cite[Appendix B.5]{Milewska2023}}}.\\

\noindent \textbf{Static giant component in the $\drig$.}
The statement on the giant in $\drig$ follows immediately. See Theorem \ref{giant_stationary}.\\

\begin{remark}
The fact that the fixed point equation (\ref{giant_fixed_pt_eq}) shows up here can be intuitively explained as follows: For a vertex to be in giant component its local neighborhood has to survive. In the results on the local convergence we showed that local neighborhoods are locally tree-like and are well approximated by branching processes with offspring distributions $\Tilde{D}^{(l)}$ and $\Tilde{D}^{(r)}$. Hence, the fixed point equation follows from the general theory of branching processes and their extinction probability.
\end{remark}

\subsection{The union graph}
Having shown multiple results for the static situation, we proceed to the dynamic part. In Section \ref{sec__main_results} we explained that in order to describe the graph dynamically for every time point $s$, it is helpful to first look collectively at everything that happens during a time interval $[0,t]$, for $t$ fixed. For that reason, we create a union graph $\bgrgunion$ (and accordingly, a resulting $\drigunion$) which includes any group that was ON at time $t=0$, but also all the groups that ever switched on within the time interval $(0,t]$.\\

\noindent\textbf{Group probabilities in $\bgrgunion$.}
Note that
\begin{align} \label{pi_a_on_union}
    \mathbf{P}&(\text{$a$ ON within $[0,t]$}) = \piaon + \piaoff\mathbf{P}(\text{$a$ switches ON within $(0,t]$})\\
    &=\piaon + \piaoff\big(1-\mathbf{P}(\text{$a$ never ON within $(0,t]$})\big) = \piaon + \piaoff\big(1-e^{-t\laoff}\big) \nonumber\\
    &=\frac{f(|a|)\prod_{i\in a}w_i}{\ell_n^{|a|-1}+f(|a|)\prod_{i\in a}w_i} + \frac{\ell_n^{|a|-1}}{\ell_n^{|a|-1}+f(|a|)\prod_{i\in a}w_i} \big(1 - e^{-t\frac{f(|a|)\prod_{i\in a}w_i}{\ell_n^{|a|-1}}} \big) \leq \piaon(1+t), \nonumber
\end{align}
using the fact that $1-e^{-x}\leq x$. Hence, we see that even though the group probability in the union graph is somewhat complicated, it can be bounded from above by $\piaon(1+t)$, which is similar to the group probability in the static graph. Thus, instead of deriving convergence results directly for $\bgrgunion$, it is more convenient to couple it with a graph that is closer to the static graph.\\

\subsubsection{The rescaled bipartite graph \texorpdfstring{$\bgrgrescaled$}{}}
Remember that since $\bgrg$ is uniform and fulfils the required degree regularity conditions, we can relate it to the $\bcm$ model from \cite{Hofstad2022} and \cite{Hofstad2018} and hence deduce its local convergence. It is not difficult to see that a graph with group ON probability $\mathbf{P}(\text{$a$ is \rm{ON}})=\piaon(1+t)$ would also satisfy equivalent regularity conditions. However, a graph with $\mathbf{P}(\text{$a$ is \rm{ON}})=\piaon(1+t), \mathbf{P}(\text{$a$ is OFF})=1-\piaon(1+t)$ will not \emph{exactly} be uniform (which is easy to see after analysing the proofs \longversion{of Proposition \ref{fct_left_right} and Theorem \ref{our_graph_uniform} in Appendix \ref{appendix_two_models_relation}}\shortversion{provided in the extended version, see \textup{\cite[Proposition A.1 and Theorem 2.1, Appendix A]{Milewska2023}}}) ). Therefore, we define a graph with group probability also closely related to $\piaon(1+t)$, but with a more convenient structure:\\

\noindent\textbf{Definition of $\bgrgrescaled$.} We introduce $\bgrgrescaled$, a graph following the same dynamic as the static $\bgrg$ but with slightly modified group probabilities: we fix the holding times to be exponentially distributed with rates
\begin{align} \label{holding_times_rescaled}
    \lambda^a_{\text{\rm{ON}}}=1 \hspace{1cm} \text{and} \hspace{1cm} \lambda^a_{\text{\rm{OFF}}}=\frac{(1+t)f(|a|)\prod_{i\in a}w_i}{\ell_n^{|a|-1}}.
\end{align}
Hence, the new stationary distribution $\pi^{(t)} = [\pi^{(t)}_{\text{\rm{ON}}},\pi^{(t)}_{\text{\rm{OFF}}}]$ is given by
\begin{align} \label{pi_ON_rescaled}
    \pi^{a,(t)}_{\text{\rm{ON}}} = \frac{(1+t)f(|a|)\prod_{i\in a}w_i}{\ell^{|a|-1}+(1+t)f(|a|)\prod_{i\in a}w_i} \hspace{1cm} \text{and} \hspace{1cm} \pi^{a,(t)}_{\text{\rm{OFF}}} = \frac{\ell_n^{|a|-1}}{\ell_n^{|a|-1}+(1+t)f(|a|)\prod_{i \in a}w_i},
\end{align}
for every $a\in[n]_k$. We again impose Condition \ref{cond_weights} on the weights $\bm{w} = (w_i)_{i \in [n]}$ and assume finite first and second moment of the group-size distribution $p_{|a|}$, taking $f(|a|)=|a|!p_{|a|}$ (see (\ref{assump_1_mom}) and (\ref{assump_2_mom})). It turns out (\longversion{see Remark \ref{appx_results_union_rescaled}}\shortversion{see the extended version \textup{\cite[Remark B.14]{Milewska2023}}}) that $\bgrgrescaled$ conditioned on its degree sequence is uniform and that its degree sequences fulfil the same regularity conditions as the degree sequences of $\bgrg$. The limiting variables of left- and right-degrees are also analogous to the limiting degree variables in $\bgrg$, with Poisson parameters rescaled by a factor $t+1$ (for an explicit statement of the regularity conditions \longversion{see Remark \ref{appx_results_union_rescaled}}\shortversion{see the extended version \textup{\cite[Remark B.14]{Milewska2023}}}). Hence, $\bgrgrescaled$ is just like $\bgrg$ with slightly bigger group probabilities and thus, it asymptotically behaves in the same way. To draw the same conclusion about the union graph $\bgrgunion$ it then suffices to show it is asymptotically equivalent to $\bgrgrescaled$, which we explain further in the next section.

\subsubsection{Asymptotic equivalence of multi-graphs}

In this section, we briefly introduce the theory of asymptotic equivalence of graph sequences. In particular, we extend the condition determining when two inhomogeneous random graphs are asymptotically equivalent to the case of random multi-graphs. We start by introducing the notion of asymptotic equivalence for general random variables. We say that $(\mathcal{X},\mathcal{F})$ is a measurable space when $\mathcal{X}$ is the state space (the space of all possible outcomes) and $\mathcal{F}$ the set of all possible events. We are particularly interested in finite measurable spaces, in which case $\mathcal{X}$ is a finite set and $\mathcal{F}$ can be taken to be the set of all subsets of $\mathcal{X}$. 

\begin{definition}[Asymptotic equivalence of sequences of random variables] \label{as_equiv_def} Let $(\mathcal{X}_n,\mathcal{F}_n)$ be a sequence of measurable spaces. Let $\mathbf{P}_n$ and $\mathbf{Q}_n$ be two probability measures on $(\mathcal{X}_n,\mathcal{F}_n)$. We say that the sequences $\big(\mathbf{P}_n\big)_{n\geq 1}$ and $\big(\mathbf{Q}_n\big)_{n\geq 1}$ are asymptotically equivalent if, for every sequence
$\mathcal{E}_n\in \mathcal{F}_n$ of events, $\lim_{n\to\infty} \mathbf{P}_n(\mathcal{E}_n) - \mathbf{Q}_n(\mathcal{E}_n)=0$. Thus, $\big(\mathbf{P}_n\big)_{n\geq 1}$ and $\big(\mathbf{Q}_n\big)_{n\geq 1}$ are asymptotically equivalent when they have asymptotically equal probabilities.
\end{definition}

\noindent In the following theorem, we give a criterion guaranteeing that two bipartite multigraph sequences are asymptotically equivalent. The section follows from results by \cite{Janson_2010}. We denote $(p_a)_{a\in [n]_k}$ for the group probabilities in some bipartite graph $\mathrm{BRG}_n(\bm{p})$ for which the probability that a group $a$ is present equals $p_a$ and all groups exist independently of each other.

\begin{restatable}[Asymptotic equivalence of bipartite multi-graphs]{theorem}{asequivmulti} \label{as_equiv_multi}
Let $\mathrm{BRG}_n(\bm{p})$ and $\mathrm{BRG}_n(\bm{q})$ be two random graphs with group probabilities $\bm{p} = (p_a)_{a\in [n]_{k\geq2}}$ and $\bm{q} = (q_a)_{a\in [n]_{k\geq2}}$ respectively. If there exists $\hat{\varepsilon}>0$ such that $\max_{a\in\cup_{k\geq2}[n]_k}q_a \leq 1-\hat{\varepsilon}$, then $\mathrm{BRG}_n(\bm{p})$ and $\mathrm{BRG}_n(\bm{q})$ are asymptotically equivalent if
\begin{align} \label{cond_for_multi_as_equiv}
    \lim_{n\to\infty} \sum_{a\in\cup_{k\geq2}[n]_k}\frac{(p_a-q_a)^2}{q_a} = 0.
\end{align}
\end{restatable}

\begin{remark}
If additionally there exists $\varepsilon>0$ such that $\max_{a\in\cup_{k\geq2}[n]_k}p_a \leq 1-\varepsilon$, $\mathrm{BRG}_n(\bm{p})$ and $\mathrm{BRG}_n(\bm{q})$ are asymptotically equivalent if and only if
\begin{align}
    \lim_{n\to\infty} \sum_{a\in\cup_{k\geq2}[n]_k}\frac{(p_a-q_a)^2}{p_a \wedge q_a} = 0.
\end{align}
\end{remark}

\noindent In particular, $\mathrm{BRG}_n(\bm{p})$ and $\mathrm{BRG}_n(\bm{q})$ are asymptotically equivalent when they can be coupled in such a way that $\mathbf{P}(\mathrm{BRG}_n(\bm{p})\neq\mathrm{BRG}_n(\bm{q}))=o(1)$. Indeed, there is a strong relationship between the asymptotic equivalence and coupling, which becomes obvious after the proof. We prove Theorem \ref{as_equiv_multi} in Section \ref{sec_pf_cond_as_equiv_multi}.\\

\noindent Having explained what is meant by asymptotic equivalence, we state our main equivalence result:

\begin{restatable}[Asymptotic equivalence of $\bgrgunion$ and $\bgrgrescaled$]{theorem}{unionrescaledequiv} \label{union_rescaled_equiv}
Under Condition \ref{cond_weights}, the random graphs $\bgrgunion$ and $\bgrgrescaled$ are asymptotically equivalent.
\end{restatable}

\noindent We prove Theorem \ref{union_rescaled_equiv} in Section \ref{sec_pf_union_rescaled_equiv}. Thanks to the equivalence, all results that we derive for $\bgrgrescaled$ automatically hold for $\bgrgunion$.

\subsubsection{Local limit of union graphs}
We now state the results on the local convergence of the union graph, starting with the bipartite one:\\

\noindent\textbf{Local limit of $\bgrgunion$.} Following the appropriate statement for the rescaled graph $\bgrgrescaled$, we conclude that the bipartite union graph behaves very regularly and fulfils the equivalent of Condition \ref{reg_cond} (i). Moreover, the limiting variables of left- and right-degrees in $\bgrgunion$, denoted further on as $\tilde{D}^{(l),[0,t]}$ and $\tilde{D}^{(r),[0,t]}$ respectively, are just the limiting variables of left- and right-degrees in $\bgrg$ with Poisson parameters and proportion of groups of size $k$ rescaled by a factor $t+1$. Hence, the union graph asymptotically behaves in the same way as the stationary graph with accordingly larger edge probabilities.\\

\noindent \textbf{The limiting object $(\mathrm{BP}_{\gamma}^{[0,t]},o)$.} The local limit of $\bgrgunion$ is again a mixture of two branching processes corresponding to two types of vertices and everything is analogous to the limit of $\bgrg$, with $D^{(l),[0,t]}$ and $D^{(r),[0,t]}$ taking the place of $D^{(l)}$ and $D^{(r)}$ as the offspring distribution of the root, $\tilde{D}^{(l),[0,t]}$ replacing $\tilde{D}^{(l)}$ as the offspring distribution of the rest of the $l-$vertices and $\tilde{D}^{(r),[0,t]}$ taking the place of $\tilde{D}^{(r)}$ as the offspring distribution of the rest of the $r-$vertices.\\

\noindent \textbf{Local limit of $\drigunion$.} The local limit $(\mathrm{CP}^{[0,t]},o)$ of $\drigunion$ can be constructed from the local limit of $\bgrgunion$ via the appropriate community projection in exactly the same way as the local limit of $\drig$ was constructed via the community projection from the local limit of $\bgrg$ (see the previous section):

\begin{restatable}[Local limit of $\bgrgunion$ and $\drigunion$]{theorem}{loclimunion}\label{loc_lim_union} Assume that Condition \ref{cond_weights} holds. Then $\big(\bgrgunion, V_n^b\big)$ converges locally in probability to $(\mathrm{BP}_{\gamma}^{[0,t]},o)$, where $(\mathrm{BP}_{\gamma}^{[0,t]},o)$ is described above. Also $\big(\drigunion, o_n\big)$ converges locally in probability to $(\mathrm{CP}^{[0,t]},o)$, where $(\mathrm{CP}^{[0,t]},o)$ is described above.
\end{restatable}

\noindent We prove Theorem \ref{loc_lim_union} in Section \ref{sec_pf_limunion}. From this analysis we also obtain results on the degrees $D_n^{[0,t]}$ and its expectation as in Corollary \ref{conv_unif_v} \& Theorem \ref{deg_seq_stat}. We refrain from stating those.

\subsubsection{Marked union graph} \label{sec_edge_marks_intro}
\noindent \textbf{Edge marks.} We remark that the union graph is already dynamic, as it takes dynamically appearing groups into account. However, it does not equal the actual dynamic graph, as it does not take into account whether certain groups were actually active \noindent \textbf{at the same time}. Therefore it might show connections that were actually never made at any time in $[0,t]$. Yet, it enables tracking the actual interactions between vertices. For this purpose, we add marks along the edges of $\bgrgunion$ indicating the switch on and switch off times within $[0,t]$. These times are determined by the activity of the groups responsible for the creation of these edges. We first mark the groups (right-vertices): write $\sigma^a_{\rm{ON}}$ to denote the exact time that a group $a$ switches on within $[0,t]$ and $\sigma^a_{\rm{OFF}}$ to denote the first time it switches off in $(0,t]$. Hence, the new mark-set is $\mathcal{M}_d=\{l,r,[0,\infty)\times(0,\infty)\}$. We then transfer the marks to the edges, i.e., every edge in $\bgrgunion$ copies the marks of the right-vertex it is adjacent to. \\

The above marks are also well-behaved. The marks of group $a$ present in the union graph are independent for different $a$ and they converge in distribution with respect to the probability measure of the union graph to some limiting marks $(t^a_{\rm{ON}},t^a_{\rm{OFF}})$:

\begin{restatable}[Convergence of the law of the edge marks.]{lemma}{convmarks}\label{conv_marks}
Let $F^{\textrm{ON,OFF}}_{n|t}$ denote the joint law of $\big(\sigma^a_{\rm{ON}},\sigma^a_{\rm{OFF}}\big)$ for $a\in\cup_{k\geq2}[n]_k$, conditioned on the fact that such $a\in\cup_{k\geq2}[n]_k$ is ON during $[0,t]$, i.e.,
\begin{align}
    F^{\rm{ON,OFF}}_{n|t}(s_1,s_2) = \mathbf{P}&(\sigma_{\rm{ON}}^a\leq s_1,\sigma_{\rm{OFF}}^a \leq s_2|\text{\textrm{ON} at some point in $[0,t]$}),
\end{align}
where $s_1\in[0,t]$ and $s_2\geq s_1$. Then, as $n\to\infty$,
\begin{align}
    F^{\rm{ON,OFF}}_{n|t}(s_1,s_2) \longrightarrow F^{\rm{ON,OFF}}_t(s_1,s_2),
\end{align}
with
\begin{align*}
    F^{\rm{ON,OFF}}_t(s_1,s_2) = \frac{1-e^{-s_2+s_1}+s_1}{1+t}.
\end{align*}
Consequently,
\begin{align}
    \big(\sigma^a_{\rm{ON}},\sigma^a_{\rm{OFF}}\big) \stackrel{\text{d}}{\longrightarrow} (t^a_{\rm{ON}},t^a_{\rm{OFF}}),
\end{align}
where $(t^a_{\rm{ON}},t^a_{\rm{OFF}})$ has joint cumulative distribution function $F^{\rm{ON,OFF}}_t$.
\end{restatable}

\noindent We prove Lemma \ref{conv_marks} in Section \ref{main_proofs}.\\

\noindent \textbf{The limit of the marked union graph.} Lemma \ref{conv_marks} is crucial in showing the dynamic local convergence. Thanks to the facts that edge marks are independent, as all the groups switch on and off independently of each other, and that they converge in distribution to limiting marks $(t^a_{\rm{ON}},t^a_{\rm{OFF}})$ (as shown in Lemma \ref{conv_marks}) we know that they converge jointly for all groups and all $s\in[0,t]$. Hence, the marked bipartite union graph will converge to the marked limit of $\bgrgunion$, which will imply that also the marked intersection union graph converges: 

\begin{restatable}[Local limit of marked $\bgrgunion$ and $\drigunion$]{theorem}{loclimunionmarked}\label{loc_lim_union_marked}
$\big(\bgrgunion,\big((\sigma^a_{\rm{ON}},\sigma^a_{\rm{OFF}})\big)_{a\in[M_n^{[0,t]}]}, V_n^b\big)$ converges locally in probability to $(\mathrm{BP}_{\gamma}^{[0,t]},\big((t^a_{\rm{ON}},t^a_{\rm{OFF}})\big),o)$, where $(\mathrm{BP}_{\gamma}^{[0,t]},\big((t^a_{\rm{ON}},t^a_{\rm{OFF}})\big),o)$ is a marked version of $(\mathrm{BP}_{\gamma}^{[0,t]},o)$. It follows that $\big(\mathrm{DRIG}^{[0,t]}_n(\bm{w}),\big((\sigma^a_{\rm{ON}},\sigma^a_{\rm{OFF}})\big)_{a\in[M_n^{[0,t]}]}, o_n\big)$ converges locally in probability to\\ $(\mathrm{CP}^{[0,t]},\big((t^a_{\rm{ON}},t^a_{\rm{OFF}})\big),o)$, where $(\mathrm{CP}^{[0,t]},\big((t^a_{\rm{ON}},t^a_{\rm{OFF}})\big),o)$ is a marked version of $(\mathrm{CP}^{[0,t]},o)$.
\end{restatable}
\noindent We prove Theorem \ref{loc_lim_union_marked} in Section \ref{sec_pf_Thm_limmarked}. 

\subsection{Dynamic local convergence}
\noindent \textbf{Switching pace.} Note that we defined $\sigma^a_{\rm{ON}}$ and $\sigma^a_{\rm{OFF}}$ as first switch-on and switch-off times within $[0,t]$ and we did not comment on the possibility that some groups might switch $\rm{ON}$ again during this period of time. We address this issue by arguing that it is unlikely to encounter such a group in a neighborhood of a uniformly chosen vertex, which is the subject of the following lemma:
\begin{lemma}[Existence of groups that switch on more than once in the union graph] \label{switching_pace_lemma}
Denote the neighbourhood of a uniformly chosen vertex in $\bgrgunion$ by $B_r^{[0,t]}(V_n^{(l)})$. As $n\to\infty$,
\begin{align}
    \mathbf{P}\big(\exists a\in B_r^{[0,t]}(V_n^{(l)}): \text{$a$ \rm{ON} twice in $[0,t]$}\big)\big) \longrightarrow 0.
\end{align}
\end{lemma}
We prove Lemma \ref{switching_pace_lemma} in Section \ref{sec_pf_switching_pace_lemma}.\\
Lemma \ref{switching_pace_lemma} implies that groups that switch on more than once in the union graph do not contribute significantly to the structure of a neighborhood of a uniformly chosen vertex. Hence, the union graph neglecting these groups is a good approximation of the actual situation and it can be used to construct the dynamic graph. We sum up this statement in the following corollary:
\begin{corollary} \label{corollary_switching_pace}
Denote the neighborhood of a uniformly chosen vertex in $\bgrgs$ by $B^s_r(V^{(l)}_n)$ and the neighborhood of a uniformly chosen vertex in $\bgrgunion$ restricted to groups that switch $\rm{ON}$ only once in $[0,t]$ and are present at time $s$ by $\Tilde{B}^{s,[0,t]}_r(V^{(l)}_n)$. As $n\to\infty$, for every finite $r>0$ and every $s\in[0,t]$, the event
 \begin{align}
    \big\{B^s_r(V^{(l)}_n) = \Tilde{B}^{s,[0,t]}_r(V^{(l)}_n)\big\}
 \end{align}
holds with high probability.
\end{corollary}
We refrain from formally proving Corollary \ref{corollary_switching_pace} and instead provide a short justification here: the fact that $\Tilde{B}^{s,[0,t]}_r(V^{(l)}_n)$ is contained in $B^s_r(V^{(l)}_n)$ follows immediately. The other direction follows directly from Lemma \ref{switching_pace_lemma}.

Thus, for every $s\in[0,t]$, $\bgrgs$ is a subgraph of $\bgrgunion$ containing only these groups that are active at time $s$, i.e., the groups $a$ with $(\sigma^a_{\rm{ON}},\sigma^a_{\rm{OFF}})$ such that $s\in[\sigma^a_{\rm{ON}},\sigma^a_{\rm{OFF}}]$. $\drigs$ is then a community projection (see (\ref{com_proj})) of such a $\bgrgs$. Hence, the convergence of the union graph and joint convergence of edge marks, guaranteed by the independence of the marks and their convergence in distribution, yield convergence of finite-dimensional distributions of the dynamic graph $\bgrgs$ for every $s\in[0,t]$. Thanks to that, we can describe local limits of $\big(\bgrgs\big)_{s\in[0,t]}$ and consecutively of $\big(\drigs\big)_{s\in[0,t]}$.

When looking at $\big(\bgrgs\big)_{s\in[0,t]}$ as a process in time, we choose a random root $o_n$ \emph{only} once and then we investigate its evolution in time. Since such a process encounters jumps with respect to the local metric, convergence of finite-dimensional distributions is not sufficient to deduce the convergence of the entire process. However, if we treat $\big(\bgrgs\big)_{s\in[0,t]}$ for every $s\in[0,t]$ as a function from the compact space of $[0,t]$ into the Polish space of rooted graphs with the $\text{d}_{\mathcal{G}_{\star}}$ metric, it suffices to add a suitable tightness criterion to deduce the dynamic convergence of the process in time (see \textup{\cite[Chapter 16]{kallenberg2002foundations}}). Details are given in the proof of Theorem \ref{dyn_loc_lim} in Section \ref{sec_pf_dyn_loclim}.

\subsection{Dynamic giant component}
We want to show that the process $J_n(s)=\mathds{1}_{\big\{o_n \in \mathscr{C}^s_1\big\}}$ converges for all $s\in[0,t]$ to another appropriate indicator process. Since both of these processes encounter jumps, we need to use the Skorokhod $J_1$ topology in order to obtain the desired convergence. The Skorokhod $J_1$ topology on $D[0,t]$ - the space of c{\`a}dl{\`a}g functions on $[0,t]$ - is given by the metric $d^0$ (see \textup{\cite[Eq. 12.16]{Billingsley2013}}), which takes care of the time deformation present in processes with jumps. For more explanation, see a well-known characterization of convergence in $D[0,t]$ (see \cite{Billingsley2013}).\\

The most important step in our proofs of conditions from Lemma \ref{conv_Skorokhod} is \textbf{localization of the giant}, i.e., noticing that the sequence of processes $\big(J_n(s)\big)_{s\in[0,t]}$ is close in distribution to the sequence of processes $\big(J^{(r)}_n(s)\big)_{s\in[0,t]}$, with
\begin{align}
    J^{(r)}_n(s) = \mathds{1}_{\big\{\partial B^{G_n^s}_r(o_n) \neq \varnothing\big\}},
\end{align}
where $\partial B^{G_n^s}_r(o_n)$ denotes the set of vertices at distance $r$ from the root in $\big(\bgrgs\big)_{s\geq0}$ at time $s$. Indeed, for two distinct time points $s_1,s_2\in[0,t]$,
\begin{align}
    \mathbf{P}(J_n(s_1)=J_n(s_2)=1) &= \mathbf{P}(J^{(r)}_n(s_1)=J^{(r)}_n(s_2)=1)\\
    &+ \mathbf{P}(J_n(s_1)=J_n(s_2)=1) 
    - \mathbf{P}(J^{(r)}_n(s_1)=J^{(r)}_n(s_2)=1). \nonumber 
\end{align}
Note that
\begin{align*}
    \big|&\mathbf{P}(J_n(s_1)=J_n(s_2)=1)
    - \mathbf{P}(J^{(r)}_n(s_1)=J^{(r)}_n(s_2)=1)\big|\leq \mathbf{P}(J_n(s_1)=J_n(s_2)=1, \neg \big(J^{(r)}_n(s_1)=J^{(r)}_n(s_2)=1\big))\\
    &=\mathbf{P}(J_n(s_1)=J_n(s_2)=1,J^{(r)}_n(s_1)\neq 1\hspace{0.2cm}\text{OR}\hspace{0.2cm}J_n(s_1)=J_n(s_2)=1,J^{(r)}_n(s_2)\neq 1)\leq 2\mathbf{P}\big(J^{(r)}_n(s_1) \neq J_n(s_1)\big),
\end{align*}
where the last step follows from stationarity. Taking $n\to\infty$, thanks to the static local limit and our result on the static giant component (see Theorem \ref{giant_bipartite}), i.e., the fact for any $s$:
\begin{align} \label{line1}
    \frac{|\mathscr{C}^s_1|}{n} \stackrel{n\to\infty}{\longrightarrow} \mu\big(|\mathscr{C}^s(o)|=\infty \big),
\end{align}
we obtain
\begin{align}
    \lim_{n\to\infty}\mathbf{P}\big(J^{(r)}_n(s_1) \neq J_n(s_1)\big) = \mathbf{P}\bigg(\mathds{1}_{\big\{\partial B^{G^s}_r(o) \neq \varnothing\big\}}\neq \mathds{1}_{\big\{|\mathscr{C}^s(o)|=\infty\big\}}\bigg),
\end{align}
which after taking additionally $r\to\infty$ yields 
\begin{align*}
    \lim_{r\to\infty}\limsup_{n\to\infty}\big|&\mathbf{P}(J_n(s_1)=J_n(s_2)=1)
    - \mathbf{P}(J^{(r)}_n(s_1)=J^{(r)}_n(s_2)=1)\big| = 0.
\end{align*}
Thus,
\begin{align*}
    \lim_{n\to\infty}\mathbf{P}(J_n(s_1)=J_n(s_2)=1) = \lim_{r\to\infty}\lim_{n\to\infty} \mathbf{P}(J^{(r)}_n(s_1)=J^{(r)}_n(s_2)=1).
\end{align*}
Thanks to this link, we can deduce the convergence of the dynamic giant process from the dynamic local weak convergence (see Theorem \ref{dyn_loc_lim}), which states that, as $n\to\infty$,
\begin{align} \label{line2}
    \Big(J^{(r)}_n(s)\Big)_{s\in[0,t]} \stackrel{\text{d}}{\longrightarrow} \Big(\mathcal{J}^{(r)}(s)\Big)_{s\in[0,t]}.
\end{align}
In the proof of Theorem \ref{giant_dynamic} we show how to extend the above argument to all finite-dimensional distributions. As a result, the convergence of all finite-dimensional distributions derived via localization paired with the tightness of the process will guarantee convergence of $\big(J_n(s)\big)_{s\in[0,t]}$. We remark that this technique is not restricted to our model and it can be applied to any other dynamic graph for which (\ref{line1}) and (\ref{line2}) hold.

\subsection{Dynamic largest group}
We investigate the behavior of the process $\big(n^{-1/\alpha}K^{[0,t]}_{\max}\big)_{t\geq0}$, where $K^{[0,t]}_{\max}$ is the maximum group size in the time interval $[0,t]$. It is an increasing process describing the largest group observed by a certain time point $t$. As such, the process in question also encounters jumps, just like the previous dynamic processes we have described, and hence, to deduce its convergence, we once again use the theory of convergence in Skorokhod topology.

\section{Proofs of the main results} \label{main_proofs}

 Here we provide proofs of all mentioned results, unless we stated we would prove them \longversion{in the Appendix}\shortversion{in the extended version \cite{Milewska2023}}.

\subsection{Proof of the condition for asymptotic equivalence for bipartite multi-graphs} \label{sec_pf_cond_as_equiv_multi}

\begin{proof}[Proof of Theorem \ref{as_equiv_multi}]
The proof of (\ref{cond_for_multi_as_equiv}) for simple random graphs can be found in \textup{\cite[Theorem 6.18]{Hofstad2016}}, which follows \textup{\cite[Corollary 2.12]{Janson_2010}}. The multi-graph version follows analogously, as $\mathrm{BRG}_n(\bm{p})$ and $\mathrm{BRG}_n(\bm{q})$ can be entirely encoded by the group presence, just like simple random graphs are encoded by the edge presence, i.e., the asymptotic equivalence of two graphs $\mathrm{BRG}_n(\bm{p})$ and $\mathrm{BRG}_n(\bm{q})$ is equivalent to the asymptotic equivalence of their group variables, which are independent Bernoulli random variables with success probabilities $\bm{p} = (p_a)_{a\in\cup_{k\geq2}[n]_k}$ and $\bm{q} = (q_a)_{a\in\cup_{k\geq2}[n]_k}$.
\end{proof}

\subsection{Proof of asymptotic equivalence of the union graph and the rescaled graph} \label{sec_pf_union_rescaled_equiv}

\begin{proof}[Proof of Theorem \ref{union_rescaled_equiv}]
Recall (\ref{pi_a_on_union}) and (\ref{pi_ON_rescaled}). To verify condition (\ref{cond_for_multi_as_equiv}) we first compute 
\begin{align}
    0 \leq \piaonunion - \piaonrescaled &\leq \frac{(1+t)f(|a|)\prod_{i\in a}w_i}{\ell^{|a|-1}+f(|a|)\prod_{i\in a}w_i} - \frac{(1+t)f(|a|)\prod_{i\in a}w_i}{\ell^{|a|-1}+(1+t)f(|a|)\prod_{i\in a}w_i}\\
    &\leq \frac{(1+t)^2f^2(|a|)\big(\prod_{i\in a}w_i\big)^2}{\ell^{|a|-1}\big(\ell^{|a|-1}+(1+t)f(|a|)\prod_{i\in a}w_i\big)},\nonumber
\end{align}
where we have used the fact that $1-e^{-x}\leq x$. Hence, by the fact that $k!<k^k$,
\begin{align} \label{equiv_cond_bound}
    \sum_{a\in\cup_{k\geq2}[n]_k} \frac{\big(\piaonunion - \piaonrescaled \big)^2}{\piaonrescaled} &\leq (1+t)^3 \sum_{k=2}^{\infty}\sum_{j_1<...<j_k\in[n]} \frac{(k!)^3(p_k)^3(w_{j_1}\cdots w_{j_k})^3}{\big(\ell_n^{k-1}\big)^3}\\
    &\leq (1+t)^3 \sum_{k=2}^{\infty} k^4p_k^3 \frac{1}{\ell_n^{k-1}} \bigg(\frac{k^2}{\ell_n} \bigg)^{k-2} \bigg(\frac{\mathbf{E}[W_n^3]}{\mathbf{E}[W_n]} \bigg)^k=o(1), \nonumber
\end{align}
which can be shown using suitable truncation arguments: one with respect to the group size and one with respect to the weights. For the first truncation, we can fix a sequence $b_n\to\infty$ and show that the contribution from groups $a$ with $|a|>b_n$ vanishes. Then take $b_n=o(\sqrt{n})$ to bound (\ref{equiv_cond_bound}) for $a$ with $|a|\leq b_n$. For the second truncation, we eliminate vertices with large weights in a similar manner. For more technical details see \longversion{Appendix \ref{appx_static_graph}}\shortversion{the extended version \textup{\cite[Appendix B]{Milewska2023}}}, where analogous truncation arguments occur frequently. Hence, for some sequence $\varepsilon_n$, as $n\to\infty$,
\begin{align}
    \mathbf{P}\Bigg(\sum_{a\in\cup_{k\geq2}[n]_k} \frac{\big(\piaonunion - \piaonrescaled \big)^2}{\piaonrescaled} \geq \varepsilon_n \Bigg) \longrightarrow 0.
\end{align}
The desired equivalence of $\bgrgunion$ and $\bgrgrescaled$ follows.
\end{proof}

\subsection{Proof of local convergence of the union graph} \label{sec_pf_limunion}

\begin{proof}[Proof of Theorem \ref{loc_lim_union}]
Convergence of $\mathrm{BGRG}^{0}_n(\bm{w})$ follows from convergence of the stationary graph $\bgrg$. For time interval $[0,t]$, we prove convergence of the union graph by showing that it is asymptotically equivalent to $\bgrgrescaled$ (see Theorem \ref{union_rescaled_equiv}). The latter, in turn, converges as it fulfils the same conditions (\longversion{see Remark \ref{appx_results_union_rescaled}}\shortversion{see the extended version \textup{\cite[Remark B.14]{Milewska2023}}}) that guaranteed convergence of $\bgrg$ (described in Section \ref{sec__main_results_stationary}, \longversion{proven in Appendix \ref{appx_static_graph}}\shortversion{proven in the extended version; See \textup{\cite[Appendix B]{Milewska2023}}}).
Hence, it also turns out that the limiting degree sequences in the union graph satisfy similar properties as the ones in the static graph, with the Poisson parameter and proportion of groups of size $k$ rescaled by the factor $t+1$. Hence, the bipartite union graph asymptotically behaves like the static graph with slightly larger edge probabilities and converges locally in probability to a related limiting object with accordingly larger offspring distributions.
\end{proof}

\subsection{Proof of the law of the marks} \label{sec_pf_Lm_marks}

\begin{proof}[Proof of Lemma \ref{conv_marks}]
We compute the law of the marks of a fixed group $a$ taking into account starting in ON and OFF state:
\begin{align} \label{marks_line1}
    \mathbf{P}&(\sigma_{\rm{ON}}^a\leq s_1,\sigma_{\rm{OFF}}^a \leq s_2|\text{ON in $[0,t]$}) = \frac{\mathbf{P}(\text{$\sigma_{\rm{ON}}^a\leq s_1,\sigma_{\rm{OFF}}^a \leq s_2$, ON in $[0,t]$})}{\mathbf{P}(\text{ON in $[0,t]$})}\\
    &=\frac{\mathbf{P}(\sigma_{\rm{ON}}^a\leq s_1,\sigma_{\rm{OFF}}^a \leq s_2)}{\mathbf{P}(\text{ON in $[0,t]$})} = \frac{\mathbf{P}(\sigma_{\rm{ON}}^a=0,\sigma_{\rm{OFF}}^a \leq s_2) + \mathbf{P}(\text{OFF at $0$},\sigma_{\rm{ON}}^a\leq s_1,\sigma_{\rm{OFF}}^a \leq s_2)}{\mathbf{P}(\text{ON in $[0,t]$})}.\nonumber
\end{align}
We compute all three ingredients separately.\\
\noindent \textbf{Step 1.}
\begin{align}
    \mathbf{P}(\sigma_{\rm{ON}}^a=0,\sigma_{\rm{OFF}}^a \leq s_2) = \piaon\mathbf{P}(\text{$a$ goes OFF in $(0,s_2]$}) = \piaon(1-e^{-s_2}).
\end{align}
\noindent \textbf{Step 2.}
\begin{align}
    \mathbf{P}(\text{OFF at $0$},\sigma_{\rm{ON}}^a\leq s_1,\sigma_{\rm{OFF}}^a \leq s_2) = \piaoff\mathbf{P}&(\text{$a$ goes ON in $(0,s_1]$, goes OFF in $(\sigma_{\rm{ON}}^a,s_2]$ }).
\end{align}
The second term can be computed as
\begin{align}
    \mathbf{P}&(\text{$a$ goes ON in $(0,s_1]$ and goes OFF in $(\sigma_{\rm{ON}}^a,s_2]$ })\\
    &= \int_0^{s_1} \mathbf{P}(Exp(\laon)\leq s_2-Exp(\laoff)\mid Exp(\laoff)=x) f_{Exp(\laoff)}(x) \text{d}x \nonumber\\
    &= \int_0^{s_1} (1-e^{-\laon(s_2-x)})\laoff e^{-\laoff x} \text{d}x.\nonumber
\end{align}
Splitting the terms and using the fact that $\int_0^{s_1}\laoff e^{-\laoff x} \text{d}x = \mathbf{P}(Exp(\laoff)\leq s_1)$ we obtain
\begin{align}
    \mathbf{P}&(\text{$a$ goes ON in $(0,s_1]$ and goes OFF in $(\sigma_{\rm{ON}}^a,s_2]$ })\\
    &= \mathbf{P}(Exp(\laoff)\leq s_1) - \laoff \int_0^{s_1} e^{-(s_2-x)} e^{-\laoff x} \text{d}x \nonumber\\
    &=1-e^{-\laoff s_1} - \frac{\laoff e^{-s_2}}{\laoff-1} (1-e^{-s_1(\laoff-1)}). \nonumber
\end{align}
Hence,
\begin{align}
    \mathbf{P}(\text{OFF at $0$},\sigma_{\rm{ON}}^a\leq s_1,\sigma_{\rm{OFF}}^a \leq s_2) = \piaoff\bigg(1-e^{-\laoff s_1} - \frac{\laoff}{\laoff-1} (e^{-s_2}-e^{-s_1\laoff}e^{s_1-s_2}) \bigg).
\end{align}
\noindent \textbf{Step 3.} For the probability in the denominator of (\ref{marks_line1}) recall once more (\ref{pi_a_on_union}).\\

Gathering all three steps together, the expression in (\ref{marks_line1}) becomes
\begin{align} \label{marks_prelimit}
    \mathbf{P}&(\sigma_{\rm{ON}}^a\leq s_1,\sigma_{\rm{OFF}}^a \leq s_2\mid \text{ON in $[0,t]$})\\
    &=\frac{\piaon(1-e^{-s_2})}{\piaon+\piaoff\big(1-e^{-\laoff t}\big)} + \frac{\piaoff\bigg(1-e^{-\laoff s_1} - \frac{\laoff}{\laoff-1} (e^{-s_2}-e^{-s_1\laoff}e^{s_1-s_2}) \bigg)}{\piaon+\piaoff\big(1-e^{-\laoff t}\big)}.\nonumber
\end{align}
We will now take the limit of the two fractions separately as $n\to\infty$. We simplify
\begin{align}
    \frac{\piaon(1-e^{-s_2})}{\piaon+\piaoff\big(1-e^{-\laoff t}\big)} = \frac{1-e^{-s_2}}{1+1/(\laoff)\big(1-e^{-\laoff t}\big)}
\end{align}
As we know, $\lim_{n\to\infty} \laoff = \lim_{n\to\infty} \frac{f(|a|)\prod_{i\in a}w_i}{\ell_n^{|a|-1}}=0.$ Substituting $x=\frac{f(|a|)\prod_{i\in a}w_i}{\ell_n^{|a|-1}}$ we obtain
\begin{align}
    \lim_{n\to\infty} {\laoff}^{-1} \big(1-e^{-\laoff t}\big) = \lim_{x\to0} \frac{1-e^{-xt}}{x} =
    \lim_{x\to0} \frac{t e^{-xt}}{1}=t,
\end{align}
so that
\begin{align}
    \lim_{n\to\infty} \frac{\piaon(1-e^{-s_2})}{\piaon+\piaoff\big(1-e^{-\laoff t}\big)} = \frac{1-e^{-s_2}}{1+t}.
\end{align}
Now we compute the limit as $n\to\infty$ of the second term in (\ref{marks_prelimit}). We again simplify, dividing by $\piaoff$, to obtain
\begin{align}
   &\frac{1-e^{-\laoff s_1} + \frac{\laoff}{1-\laoff} (e^{-s_2}-e^{-s_1\laoff}e^{s_1-s_2})}{\laoff+1-e^{-\laoff t}}\\
   &=\frac{1-e^{-\laoff s_1}}{\laoff+1-e^{-\laoff t}} + \frac{\laoff (e^{-s_2}-e^{-s_1\laoff}e^{s_1-s_2})}{(1-\laoff)(\laoff+1-e^{-\laoff t})} = A_n + B_n. \nonumber
\end{align}
Then, using the same substitution as previously,
\begin{align}
    \lim_{n\to\infty} A_n = \lim_{x\to0} \frac{1-e^{-xs_1}}{x+1-e^{-xt}} =
    \lim_{x\to0} \frac{s_1e^{-xs_1}}{1+te^{-xt}}=\frac{s_1}{1+t},
\end{align}
and
\begin{align}
    \lim_{n\to\infty} B_n &= \lim_{x\to0} \frac{xe^{-s_2} -xe^{-s_1x}e^{s_1-s_2}}{1-e^{-xt}-x^2+xe^{-xt}}=\frac{e^{-s_2}-e^{s_1-s_2}}{1+t}.
\end{align}
Combining,
\begin{align}
    \lim_{n\to\infty} \frac{\piaoff\bigg(1-e^{-\laoff s_1} - \frac{\laoff}{\laoff-1} (e^{-s_2}-e^{-s_1\laoff}e^{s_1-s_2}) \bigg)}{\piaon+\piaoff\big(1-e^{-\laoff t}\big)} = \frac{s_1+e^{-s_2}-e^{s_1-s_2}}{1+t}.
\end{align}
Thus, by (\ref{marks_prelimit}),
\begin{align}
    \lim_{n\to\infty} \mathbf{P}&(\sigma_{\rm{ON}}^a\leq s_1,\sigma_{\rm{OFF}}^a \leq s_2\mid \text{ON in $[0,t]$}) = \frac{1-e^{s_1-s_2}+s_1}{1+t},
\end{align}
as required.
\end{proof}

\subsection{Proof of local limit of marked union graphs} \label{sec_pf_Thm_limmarked}

\begin{proof}[Proof of Theorem \ref{loc_lim_union_marked}]
Theorem \ref{loc_lim_union} shows local convergence of the unmarked $\bgrgunion$, which means that for any fixed rooted graph $(H_{\star},o')$ and $r\in\mathbf{N}$,
\begin{align*}
    \mathbf{P}(B_r(G_n^{[0,t]},V^{(l)}_n) \simeq (H_{\star},o')\mid  G_n)&:=\frac{1}{|G_n|} \sum_{i\in[n]} \mathds{1}_{\{B_r(G_n^{[0,t]},i) \simeq (H_{\star},o')\}}\\
    &\stackrel{\mathbf{P}}{\longrightarrow} \hspace{0.1cm}\mathbf{P}(B_r(\mathrm{BP}^{[0,t]}_{\gamma},o) \simeq (H_{\star},o')),
\end{align*}
where we have written $(G_n^{[0,t]},o)$ instead of $\mathrm{BGRG}^{[0,t]}_n$ for the sake of simplicity of notation in this proof.
If the marked version of the union graph converges locally in probability, then for any fixed marked rooted graph $(H_{\star},(\Bar{m}_1,\Bar{m}_2),o')$ and $r\in\mathbf{N}$,
\begin{align*}
    \mathbf{P}&\bigg(\text{d}_{\mathcal{G}_{\star}}\big((G_n^{[0,t]},\big((\sigma^a_{\rm{ON}},\sigma^a_{\rm{OFF}} )\big)_{a\in\cup_{k\geq2}[n]_k},V^{(l)}_n),(H_{\star},(\Bar{m}_1,\Bar{m}_2),o') \big) \leq \frac{1}{r+1}
    \mid  (G_n^{[0,t]},\big((\sigma^a_{\rm{ON}},\sigma^a_{\rm{OFF}} )\big)_{a\in\cup_{k\geq2}[n]_k}\bigg)\\
    &:=\frac{1}{n} \sum_{i\in[n]} \mathds{1}_{\big\{\text{d}_{\mathcal{G}_{\star}}\big((G_n^{[0,t]},\big((\sigma^a_{\rm{ON}},\sigma^a_{\rm{OFF}} )\big)_{a\in\cup_{k\geq2}[n]_k},i),(H_{\star},(\Bar{m}_1,\Bar{m}_2),o') \big) \leq \frac{1}{r+1}\big\}}\\
    &\stackrel{\mathbf{P}}{\longrightarrow} \hspace{0.1cm}\mathbf{P}\bigg(\text{d}_{\mathcal{G}_{\star}}\big((\mathrm{BP}^{[0,t]}_{\gamma},\big((t^a_{\rm{ON}},t^a_{\rm{OFF}})\big),o),(H_{\star},(\Bar{m}_1,\Bar{m}_2),o')\big)\leq\frac{1}{r+1}\bigg).
\end{align*}
Note that the edge marks are independent since all the groups switch on and off independently of each other. Hence, if the marks converge in distribution to some limiting marks they will converge jointly for all groups and all $s\in[0,t]$. In Lemma \ref{conv_marks} we showed that they indeed converge and also that the marks of all groups present within $[0,t]$ are also identically distributed. This implies that we can couple each pair of marks with their limiting marks so that they are appropriately close to each other. Hence, the proportion of vertices whose neighborhoods look like $(H_{\star},(\Bar{m}_1,\Bar{m}_2),o')$ must converge to the probability that neighborhoods in $(B_r(\mathrm{BP}^{[0,t]}_{\gamma},\big((t^a_{\rm{ON}},t^a_{\rm{OFF}})\big),o)$ look like $(H_{\star},(\Bar{m}_1,\Bar{m}_2),o')$, which precisely means that the marked version of $\bgrgunion$ converges. Next, the convergence of the marked intersection graph follows from the convergence of the underlying bipartite structure, as community projection preserves marked graphs' distance.
\end{proof}

\subsection{Proof of the switching pace of the groups in the union graph} \label{sec_pf_switching_pace_lemma}

\begin{proof}[Proof of Lemma \ref{switching_pace_lemma}]
For any finite $r>0$, we investigate
\begin{align} \label{switch_twice_final_expect}
    \mathbf{P}\big(\exists a\in B_r^{[0,t]}(V_n^{(l)}): \text{$a$ \rm{ON} twice in $[0,t]$}\big) = \mathbf{E}\bigg[\mathbf{P}\big(\exists a\in B_r^{[0,t]}(V_n^{(l)}): \text{$a$ \rm{ON} twice in $[0,t]$}\mid  G_n^{[0,t]}\big) \bigg],
\end{align}
where we denote the $r$-neighbourhood of a uniformly chosen vertex in $\bgrgunion$ by $B_r^{[0,t]}(V_n^{(l)})$ and for simplification we write $G_n^{[0,t]}$ instead of $\bgrgunion$ to denote conditioning on the union graph. Applying the union bound yields
\begin{align} \label{union_bound_gr_on_twice}
    \mathbf{P}\big(\exists a\in B_r^{[0,t]}(V_n^{(l)}): \text{$a$ \rm{ON} twice in $[0,t]$}\mid  G_n^{[0,t]}\big) \leq 1\wedge \sum_{a\in B_r^{[0,t]}(V_n^{(l)})} \mathbf{P}\big(\text{$a$ \rm{ON} twice in $[0,t]$}\mid  G_n^{[0,t]}\big)
\end{align}
We first apply a suitable truncation (for more details \longversion{see similar cases, for instance, proof of Theorem \ref{1st_mom_unif_bipartite} or Remark \ref{appx_eliminating_cond_c}}\shortversion{see the extended version \textup{\cite[Theorem B.4, Remark B.11]{Milewska2023}}}): we truncate the maximum vertex weight by $a_n$ and the maximum group size by $b_n$, with $a_n=\log(n), b_n=\log(n)$ to obtain
\begin{align} \label{switch_twice_truncation}
    \sum_{a\in B_r^{[0,t]}(V_n^{(l)})}&\mathbf{P}\big(\text{$a$ \rm{ON} twice in $[0,t]$}\mid  G_n^{[0,t]}\big)\\
    &=\sum_{a\in B_r^{[0,t]}(V_n^{(l)})}\mathbf{P}\big(\text{$a$ \rm{ON} twice in $[0,t]$}\mid  G_n^{[0,t]}\big)\mathds{1}_{\{\max_{i} w_i\leq a_n, \max_{a}|a|\leq b_n\}} + o_{\mathbf{P}}(1),\nonumber
\end{align}
where the maximums are taken w.r.t.\ a vertex $i\in B_r^{[0,t]}(V_n^{(l)})$ and group $a\in B_r^{[0,t]}(V_n^{(l)})$. The remainder is small since Condition \ref{cond_weights} implies that our union graph is sparse, and hence observing large groups and vertices with large weights in a uniformly chosen $r$-neighborhood is unlikely. We now compute
\begin{align} \label{probab_gr_on_twice_if_on_in_0_t}
    \mathbf{P}(\text{$a$ switches ON twice in $[0,t]$} \mid  G_n^{[0,t]}) &= \frac{\mathbf{P}(\text{$a$ switches ON twice in $[0,t]$}\mid \text{$a$ active in $[0,t]$})}{\mathbf{P}(\text{$a$ active in $[0,t]$})} \\
    &= \frac{\mathbf{P}(\text{$a$ switches ON twice in $[0,t]$})}{\mathbf{P}(\text{$a$ active in $[0,t]$})}, \nonumber
\end{align}
where the first equality follows from the independence of groups. We have
\begin{align} \label{groups_that_switch_twice}
    \mathbf{P}&(\text{$a$ switches ON twice in $[0,t]$}) = \piaon\mathbf{P}(\text{$a$ switches OFF and ON again in $[0,t]$})\\
    &+\piaoff\mathbf{P}(\text{$a$ switches ON, then OFF and ON again in $[0,t]$}).\nonumber
\end{align}
Recall that the times that groups spend in the ON and OFF states are exponentially distributed with rates $Exp(\laon)$ and $Exp(\laoff)$ respectively. Hence, using the fact that for all $x, 1-e^{-x}\leq x$,
\begin{align} \label{switch_twice_1st_bound}
    \mathbf{P}(\text{$a$ switches OFF and ON again in $[0,t]$})&= 
    \mathbf{P}(Exp(\laon)+Exp(\laoff)\leq t)\\
    &\leq \mathbf{P}(Exp(\laoff)\leq t) = \big(1-e^{-\laoff t}\big) \leq t\laoff \wedge1, \nonumber
\end{align}
and, applying the same inequality again,
\begin{align} \label{switch_twice_2nd_bound}
    \mathbf{P}&(\text{$a$ switches ON, then OFF and ON again in $[0,t]$})=\mathbf{P}(Exp(\laoff) + Exp(\laon)+Exp'(\laoff)\leq t)\\
    &\leq \mathbf{P}(Exp(\laoff)+Exp'(\laoff)\leq t) = 1-e^{-t\laoff}-t\laoff e^{-t\laoff} \leq t\laoff(1-e^{-t\laoff}) \leq (t\laoff)^2\wedge1. \nonumber
\end{align}
Substituting (\ref{switch_twice_1st_bound}) and (\ref{switch_twice_2nd_bound}) into (\ref{groups_that_switch_twice}) and using the facts that $\piaon \leq \laon, \piaoff\leq1$ yields
\begin{align}
    \mathbf{P}&(\text{$a$ switches ON twice in $[0,t]$}) \leq (t+t^2)(\laoff)^2\wedge1. 
\end{align}
To obtain the probability in the denominator of (\ref{probab_gr_on_twice_if_on_in_0_t}) recall (\ref{pi_a_on_union}). Combining all of the above, we arrive at the bound
\begin{align} \label{switch_twice_final_bound}
    \mathbf{P}&(\text{$a$ switches ON twice in $[0,t]$} \mid  G_n^{[0,t]}) \leq c(t)\laoff \wedge 1,
\end{align}
where $c(t)$ is a constant. Thus, substituting (\ref{switch_twice_final_bound}) into (\ref{union_bound_gr_on_twice}) and invoking (\ref{switch_twice_truncation}) yields
\begin{align}\label{switch_pace_final_exp_bound}
    \mathbf{E}&\bigg[\mathbf{P}\big(\exists a\in B_r^{[0,t]}(V_n^{(l)}): \text{$a$ \rm{ON} twice in $[0,t]$} \mid  G_n^{[0,t]}\big)\bigg]\\
    &\leq \mathbf{E}\bigg[\bigg(\sum_{a\in B_r^{[0,t]}(V_n^{(l)})} (\laoff\wedge1)\cdot \mathds{1}_{\{\max_{i} w_i\leq a_n, \max_{a}|a|\leq b_n\}}\bigg)\wedge1 \bigg].\nonumber
\end{align}
Note that by the local limit of $\bgrgunion$ shown in Theorem \ref{loc_lim_union} we know that $|B_r^{[0,t]}(V_n^{(l)})|$ is  tight. We also know that $\laoff$ is small for $n$ large for a group $a$ with $\max_{i\in a} w_i\leq a_n$ and $|a|\leq b_n$. Therefore, as $n\to\infty$,
\begin{align}
    \sum_{a\in B_r^{[0,t]}(V_n^{(l)})} \big(\laoff \wedge 1\big) \mathds{1}_{\{\max_{i\in a} w_i\leq a_n, \max|a|\leq b_n\}} \stackrel{\mathbf{P}}{\longrightarrow}0.
\end{align}
Hence, by applying the dominated convergence theorem to (\ref{switch_pace_final_exp_bound}) we conclude
\begin{align}
    \mathbf{P}\big(\exists a\in B_r^{[0,t]}(V_n^{(l)}): \text{$a$ \rm{ON} twice in $[0,t]$}\big)\big) = o(1).
\end{align}
\end{proof}

\subsection{Proof of dynamic local limit of the intersection graph} \label{sec_pf_dyn_loclim}

\begin{proof}[Proof of Theorem \ref{dyn_loc_lim}]
We first show the dynamic local weak convergence of $\big(\mathrm{BGRG}^{s}_n(w)\big)_{s\in[0,t]}$, i.e.,
\begin{align}
    \bigg(\big(\mathrm{BGRG}^s_n,V^{(l)}_n\big)\bigg)_{s\in[0,t]} \stackrel{\text{d}}{\longrightarrow} \bigg((\mathrm{BP}^s_{\gamma},o) \bigg)_{s\in[0,t]}.
\end{align}
The proof follows in two steps: we first show convergence of finite-dimensional distributions and then tightness - the two conditions for weak convergence of processes from compact space to separable and complete space given in literature \textup{\cite[Lemma 16.2, Theorem 16.3]{kallenberg2002foundations}}. For the convenience of the reader, we reproduce these results \longversion{in Appendix \ref{appx_jumps_conv} (see Lemma \ref{conv_compact_to_separable} and Theorem \ref{rel_compact_vs_tight})}\shortversion{in the extended version \textup{\cite[Lemma C.2 and Theorem C.3, Appendix C]{Milewska2023}}}.\\

\noindent \textbf{Condition (i): Convergence of finite-dimensional distributions.} Once more, for the sake of simplicity of notation, throughout the proof we abbreviate $\bgrgs$ to $\mathrm{G}^{s}_n$. We need to show that for all $s_1\leq s_2\leq...\leq s_k\in[0,t]$
\begin{align}
    \mathbf{P}(\forall j\in[k]: B_r(G^{s_j}_n,V^{(l)}_n)\simeq (H_j,o)\mid  G_n) &= \frac{1}{n}\mathbf{E}\Bigg[\sum_{i\in[n]}\mathds{1}_{\big\{B_r(G^{s_j}_n,i)\simeq (H_j,o)\big\}} \Bigg] \\
    &\longrightarrow \mathbf{P}(\forall j\in[k]: B_r(\mathrm{BP}^{s_j}_{\gamma},o)\simeq (H_j,o)). \nonumber
\end{align}
\noindent The convergence follows immediately from the convergence of the marked union graph. Indeed, if the marked union graph converges, appropriate marked-graph isomorphisms must hold. In particular, since marks converge,
\begin{align}
    \mathbf{P}\big(s\in[\sigma^a_{\rm{ON}}, \sigma^a_{\rm{OFF}}]\big) \longrightarrow \mathbf{P}\big(s\in[t^a_{\rm{ON}},t^a_{\rm{OFF}}]\big).
\end{align}
Thus indeed, if a neighborhood of a uniformly chosen vertex in the marked union graph resembles a neighborhood in a marked $(\mathrm{BP}^{[0,t]}_{\gamma},o)$, then for every $s\in[0,t]$ a neighborhood of a uniformly chosen vertex in a subgraph of $\bgrgunion$ restricted to groups $a$ such that $s\in[\sigma^a_{\rm{ON}}, \sigma^a_{\rm{OFF}}]$ must resemble a subgraph of $(\mathrm{BP}_{\gamma}^{[0,t]},\big((t^a_{\rm{ON}},t^a_{\rm{OFF}})\big),o)$ incorporating accordingly only right-vertices $a$ such that $s\in[t^a_{\rm{ON}},t^a_{\rm{OFF}}]$.\\

\noindent \textbf{Condition (ii): Tightness of the process.} Since for random processes between a compact and a Polish space, the convergence of finite-dimensional distributions combined with tightness (which, in separable and complete spaces, is equivalent to relative compactness in distribution) yields process convergence (see \longversion{Appendix \ref{appx_jumps_conv}}\shortversion{the extended version \textup{\cite[Appendix C]{Milewska2023}}} for a brief summary of results we use, taken from  \cite{kallenberg2002foundations}, Chapter 16), it remains to show that our dynamic graph process is tight, with respect to its local topology. This translates to verifying if for all $0<s_1<s<s_2<t$, $\varepsilon, \eta>0$ there exists $n_0\geq1$ and $\delta>0$ such that for all $n\geq n_0$
\begin{align} \label{cond_tight_for_dyn_loc_lim}
    \mathbf{P}\bigg(\sup_{(s,s_1,s_2)\in\mathscr{S}_{\delta}} \min\big[\text{d}_{\mathcal{G}_{\star}}\big((G^{s_1}_n,V^{(l)}_n),(G^{s}_n,V^{(l)}_n)\big),\text{d}_{\mathcal{G}_{\star}}\big((G^{s}_n,V^{(l)}_n),(G^{s_2}_n,V^{(l)}_n)\big)\big]>\varepsilon \bigg) \leq \eta,
\end{align}
with $\mathscr{S}_{\delta} = \{(s,s_1,s_2):s\in[s_1,s_2], |s_2-s_1|\leq \delta\}$. Note that the above is equivalent to
\begin{align} \label{dyn_loc_lim_tight_tobebounded}
    \mathbf{P}\bigg(\exists s\in[s_1,s_2],s_2-s_1<\delta: B_{1/\varepsilon}(G^{s_1}_n,V^{(l)}_n) \not\simeq  B_{1/\varepsilon}(G^{s}_n,V^{(l)}_n), B_{1/\varepsilon}(G^{s}_n,V^{(l)}_n) \not\simeq  B_{1/\varepsilon}(G^{s_2}_n,V^{(l)}_n) \bigg) \leq \eta.
\end{align}
We partition $[0,t]$ into intervals of length $\delta$ and introduce
\begin{align} \label{ineq_interval_union}
    S_l = \{\text{two jumps in the neighbourhood of $V^{(l)}_n$ in the $l$-th interval of length $\delta$}\}.
\end{align}
Note that thanks to the stationarity, the probability of a jump in $[0,s]$ and then another jump in $[s,\delta]$ can be bounded by the probability of two jumps in $[0,\delta]$. Hence,
\begin{align} \label{bound_delta_intervals}
    (\ref{dyn_loc_lim_tight_tobebounded}) \leq \mathbf{P}(\bigcup_{l=1}^{t/\delta} S_l) \leq \frac{t}{\delta}\mathbf{P}(S_1),
\end{align}
where $S_1$ is then accordingly the event of two jumps in the time interval $[0,\delta]$. Recall that the only aspect that can cause changes in neighborhoods is group activation and deactivation. Note that the groups that changed their status within $[0,t]$ must in particular be active during $[0,t]$, which means they are in the union graph $\bgrgunion$. Therefore, as it turns out to be useful for the upper bound, we now condition event $S_1$ on the union graph. For simplicity, we denote the union graph by $G_n^{[0,t]}$. We also denote the $r$-neighborhood of $V_n^{(l)}$ in the union graph by $B_r^{[0,t]}(V_n^{(l)})$. We compute
\begin{align}
    \mathbf{P}(S_1\mid  G^{[0,t]}_n) = & \mathbf{P}(\text{$\exists a_1\neq a_2\in B_{1/\varepsilon}^{[0,t]}(V_n^{(l)}): a_1,a_2$ switch OFF in $[0,\delta]$}\mid G^{[0,t]}_n)\\
    &+ \mathbf{P}(\text{$\exists a_1\neq a_2\in B_{1/\varepsilon}^{[0,t]}(V_n^{(l)}): a_1$ switches OFF, $a_2$ switches ON in $[0,\delta]$}\mid G^{[0,t]}_n)\nonumber\\
    &+ \mathbf{P}(\text{$\exists a_1\neq a_2\in B_{1/\varepsilon}^{[0,t]}(V_n^{(l)}): a_1,a_2$ switch ON in $[0,\delta]$}\mid  G^{[0,t]}_n).\nonumber
\end{align}
Hence,
\begin{align} \label{dyn_loc_lim_3_terms}
    \mathbf{P}(S_1\mid  G^{[0,t]}_n) =& \big(1-e^{-\delta \#\{a\in B_{1/\varepsilon}^{[0,t]}(V_n^{(l)})\}}\big)^2\\
    &+ \big(1-e^{-\delta\#\{a\in B_{1/\varepsilon}^{[0,t]}(V_n^{(l)})\}}\big)\big(1-e^{-\delta \sum_{a\in B_{1/\varepsilon}^{[0,t]}(V_n^{(l)})} \laoff }\big) \wedge 1\nonumber\\
    &+ \big(1-e^{-\delta \sum_{a\in B_{1/\varepsilon}^{[0,t]}(V_n^{(l)})} \laoff}\big)^2 \wedge 1.\nonumber
\end{align}
Fix a large constant $b_{\delta}$, and consider
\begin{align}
    \mathscr{E} = \big\{ |B_{1/\varepsilon}^{[0,t]}(V_n^{(l)})| \leq b_{\delta} \big\}.
\end{align}
Given $\mathscr{E}$,
\begin{align} \label{final_bound_e_delta}
    \mathbf{P}(S_1\mid G^{[0,t]}_n) & \leq 3 \delta^2 b^2_{\delta}.
\end{align}
Recall that in Theorem \ref{loc_lim_union} we derived local convergence of the union graph, which guarantees that the neighbourhood of a uniformly chosen vertex in the union graph is bounded. Thus, for every $\varepsilon>0$, we can find a $b_{\delta}$ sufficiently large, such that $\mathbf{P}(\mathscr{E})\geq 1-\varepsilon$. After taking the expectation of (\ref{final_bound_e_delta}) with respect to the union graph, and substituting the result into (\ref{bound_delta_intervals}), we obtain
\begin{align}
    \mathbf{P}&\bigg(\exists s\in[s_1,s_2],s_2-s_1<\delta\colon  B_{1/\varepsilon}(G^{s_1}_n,V^{(l)}_n) \not\simeq  B_{1/\varepsilon}(G^{s}_n,V^{(l)}_n), B_{1/\varepsilon}(G^{s}_n,V^{(l)}_n) \not\simeq  B_{1/\varepsilon}(G^{s_2}_n,V^{(l)}_n) ; \mathscr{E}\bigg)\\
    &\leq \frac{t}{\delta}\mathbf{E}\big[3 \delta^2 b^2_{\delta}\mathds{1}_{\mathscr{E}} \big] = 3t \mathbf{E}\big[b^2_{\delta}\mathds{1}_{\mathscr{E}}\big] \delta,\nonumber
\end{align}
which, for every $b_{\delta}$, can be made arbitrarily small by taking $\delta$ small. As we argued that for such a choice of $b_{\delta}$, event $\mathscr{E}$ holds with probability at least $1-\varepsilon$, we conclude that (\ref{dyn_loc_lim_tight_tobebounded}) holds by taking $\delta$ and $\varepsilon$ small as a function of $\eta$.\\

\noindent\textbf{Consequence.} For every $s\in[0,t]$, $\drigs$ can be built from $\bgrgs$ via a community projection, which preserves graph isomorphism and tightness. Hence, its convergence follows from the just shown convergence of $\bgrgs$, and its local limit, $(\mathrm{CP}^{s},o)$, is a community projection of the limit of $\bgrgs$.
\end{proof}

\begin{remark}
In the proof of Theorem \ref{dyn_loc_lim} we show dynamic local weak convergence. However, we argue that in the same manner, we could derive dynamic local convergence in probability. Indeed, note that neighborhood processes of two distinct uniformly chosen vertices are i.i.d. stochastic processes, and hence, their convergence can be derived in the same way as in the proof above. The authors of \cite{dynweaklimit2023} prove this result rigorously (see \textup{\cite[Lemma 3.10]{dynweaklimit2023}}).
\end{remark}

\subsection{Proof of convergence of the dynamic giant membership process} \label{sec_pf_dyn_giant}

\begin{proof}[Proof of Theorem \ref{giant_dynamic}]
Note again that the giant component in the intersection graph is strictly connected to the giant component in the underlying bipartite structure. Hence, in this proof, we only focus on the underlying bipartite structure. To show the desired convergence it suffices to show that $\big(J_n(s)\big)_{s\in[0,t]}$ and $\big(\mathcal{J}(s)\big)_{s\in[0,t]}$ satisfy Conditions (i)-(iii) from Lemma \ref{conv_Skorokhod}.

\noindent \textbf{Condition (i): Convergence of finite-dimensional distributions.}
Note that thanks to our results on the static giant component (see Theorem \ref{giant_bipartite}),
\begin{align} \label{giant_at_s_conv}
    \mathbf{P}(J_n(s)=1) = \mathbf{P}(o_n \in \mathscr{C}^s_1) = \frac{\mathbf{E}[|\mathscr{C}^s_1|]}{n} \stackrel{n\to\infty}{\longrightarrow} \mu\big(|\mathscr{C}^s(o)|=\infty \big),
\end{align}
where $(G^s,o)$ is the limiting, rooted graph at time $s$. Furthermore, as a consequence of the dynamic local weak convergence (Theorem \ref{dyn_loc_lim}), as $n\to\infty$,
\begin{align} \label{loc_conv_r_neighb}
    \Big(J^{(r)}_n(s)\Big)_{s\in[0,t]} = \Big(\mathds{1}_{\big\{\partial B^{G_n^s}_r(o_n) \neq \varnothing\big\}}\Big)_{s\in[0,t]} \stackrel{\text{d}}{\longrightarrow} \Big(\mathds{1}_{\big\{\partial B_r^{G^s}(o) \neq \varnothing\big\}}\Big)_{s\in[0,t]} = \big(\mathcal{J}^{(r)}(s)\big)_{s\in[0,t]},
\end{align}
with $\partial B^{G_n^s}_r(o_n) = \{i\in[n]:\text{d}(o,v)=r\}$, i.e., the set of vertices of graph distance $r$ from the root. Since we know what is happening in local neighborhoods jointly for all $s\in[0,t]$ we try to link the distribution of $\big(J_n(s)\big)_{s\in[0,t]}$ to the distribution of $\big(J^{(r)}_n(s)\big)_{s\in[0,t]}$. For any $r>0$ and for all $\{s_1,...,s_k\}\in[0,t]$,
\begin{align} \label{probab_diff}
    \mathbf{P}&(J_n(s_1)=...=J_n(s_k)=1) = \mathbf{P}(J^{(r)}_n(s_1)=...=J^{(r)}_n(s_k)=1)\\
    &+ \mathbf{P}(J_n(s_1)=...=J_n(s_k)=1) 
    - \mathbf{P}(J^{(r)}_n(s_1)=...=J^{(r)}_n(s_k)=1). \nonumber 
\end{align}
We look at the difference of probabilities in (\ref{probab_diff}):
\begin{align}
    \big|\mathbf{P}&(J_n(s_1)=...=J_n(s_k)=1)
    - \mathbf{P}(J^{(r)}_n(s_1)=...=J^{(r)}_n(s_k)=1)\big|\\
    &\leq \mathbf{P}(J_n(s_1)=...=J_n(s_k)=1, \neg \big(J^{(r)}_n(s_1)=...=J^{(r)}_n(s_k)=1\big))\nonumber\\
    &=\mathbf{P}\big(J_n(s_1)=...=J_n(s_k)=1, \bigcup_{i=1}^k \big\{ J^{(r)}_n(s_i)\neq1\big\}\big)\nonumber\\ &=\mathbf{P}\big(\bigcup_{i=1}^k \big\{J_n(s_1)=...=J_n(s_k)=1,J^{(r)}_n(s_i)\neq1\big\}\big) \leq \mathbf{P}\big(\bigcup_{i=1}^k \big\{ J^{(r)}_n(s_i) \neq J_n(s_i)\big\}\big). \nonumber
\end{align}
Note that by the static local limit and (\ref{giant_at_s_conv}),
\begin{align}
    \lim_{r\to\infty} \lim_{n\to\infty} \mathbf{P}\big(\bigcup_{i=1}^k \big\{ J^{(r)}_n(s_i) \neq J_n(s_i)\big\}\big) \leq k \cdot\lim_{r\to\infty} \lim_{n\to\infty} \mathbf{P}(J^{(r)}_n(s) \neq J^{n}(s)) = 0.
\end{align}
Hence,
\begin{align}
    \lim_{r\to\infty} \lim_{n\to\infty} \big|\mathbf{P}&(J_n(s_1)=...=J_n(s_k)=1)
    - \mathbf{P}(J^{(r)}_n(s_1)=...=J^{(r)}_n(s_k)=1)\big|= 0,
\end{align}
and thus,
\begin{align}
    \lim_{n\to\infty} \mathbf{P}&(J_n(s_1)=...=J_n(s_k)=1) = \lim_{r\to\infty} \lim_{n\to\infty} \mathbf{P}(J^{(r)}_n(s_1)=...=J^{(r)}_n(s_k)=1).
\end{align}
However, by (\ref{loc_conv_r_neighb}),
\begin{align}
    \mathbf{P}(J^{(r)}_n(s_1) = 1,...,J^{(r)}_n(s_k)=1) &\stackrel{n\to\infty}{\longrightarrow} \mathbf{P}(\mathcal{J}^{(r)}(s_1)=1,...,\mathcal{J}^{(r)}(s_k)=1 )\\
    &\stackrel{r\to\infty}{\longrightarrow} \mu(|\mathscr{C}^{G^{s_1}}(o)|=...=|\mathscr{C}^{G^{s_k}}(o)|=\infty). \nonumber
\end{align}
Hence,
\begin{align*}
    \mathbf{P}&(J_n(s_1)=...=J_n(s_k)=1) \stackrel{r\to\infty}{\longrightarrow} \mu(|\mathscr{C}^{G^{s_1}}(o)|=...=|\mathscr{C}^{G^{s_k}}(o)|=\infty).
\end{align*}
Similarly, for another combination of values of the finite-dimensional distribution,
\begin{align}
    \mathbf{P}&(J_n(s_1)=...=1,J_n(s_k)=0) = \mathbf{P}(J^{(r)}_n(s_1)=...=1,J^{(r)}_n(s_k)=0)\\
    &+ \mathbf{P}(J_n(s_1)=...=1,J_n(s_k)=0)
    - \mathbf{P}(J^{(r)}_n(s_1)=...=1,J^{(r)}_n(s_k)=0). \nonumber
\end{align}
We again look at the difference in the second line:
\begin{align}
    \mathbf{P}&(J_n(s_1)=...=1,J_n(s_k)=0)
    - \mathbf{P}(J^{(r)}_n(s_1)=...=1,J^{(r)}_n(s_k)=0)\\
    &\leq \mathbf{P}\big((J_n(s_1)=...=1,J_n(s_k)=0), \neg (J^{(r)}_n(s_1)=...=1,J^{(r)}_n(s_k)=0)\big) \nonumber\\
    &=\mathbf{P}\big((J_n(s_1)=...=1,J_n(s_k)=0), (J^{(r)}_n(s_1)\neq1 \hspace{0.15cm} \text{or} \hspace{0.1cm}
    ... \hspace{0.1cm} \text{or} \hspace{0.15cm}J^{(r)}_n(s_k)\neq0)\big), \nonumber
\end{align}
which would again imply that for at least some $s_i\in[0,t]$ we have $\mathbf{P}(J^{(r)}_n(s_i) \neq J_n(s_i))$, so by the same argument as previously we have that the above vanishes. Hence, we can use the same argument to show the convergence of all finite-dimensional distributions.\\

\noindent \textbf{Condition (ii): Tightness of the limiting process.}
We want to show that, for all $\varepsilon>0$, as $\delta\to0$,
\begin{align} \label{cond_ii_cited}
    \mathbf{P}\big(|\mathcal{J}(t) - \mathcal{J}(t-\delta)| > \varepsilon \big) \longrightarrow 0.
\end{align}
Since $\big(\mathcal{J}(s)\big)_{s\in[0,t]}$ is an indicator process, the difference in absolute value between any two points of the process equals either $0$ or $1$. Hence $\mathbf{P}\big(|\mathcal{J}(t) - \mathcal{J}(t-\delta)| > \varepsilon \big)$ is equivalent to $\mathbf{P}\big(|\mathcal{J}(t) - \mathcal{J}(t-\delta)|=1 \big)$, which is equivalent to $\mathbf{P}(\mathcal{J}(t-\delta)=0,\mathcal{J}(t)=1)+\mathbf{P}(\mathcal{J}(t-\delta)=1,\mathcal{J}(t)=0)$. We investigate these two factors separately. From the proof of condition $(i)$,
\begin{align} \label{cond_ii_rewritten_f1}
    \lim_{\delta\to0} \mathbf{P}(\mathcal{J}(t-\delta)=0,\mathcal{J}(t)=1) &=  \lim_{\delta\to0} \lim_{n\to\infty} \mathbf{P}(J_n(t-\delta)=0,J_n(t)=1)\\
    &= \lim_{\delta\to0} \lim_{r\to\infty} \lim_{n\to\infty} \mathbf{P}(J^{(r)}_n(t-\delta)=0,J^{(r)}_n(t)=1)\nonumber. 
\end{align}
We compute
\begin{align} \label{switch_event_one}
    \mathbf{P}(J^{(r)}_n(t-\delta)=0,J^{(r)}_n(t)=1) = \mathbf{P}\bigg(\partial B^{G^{t-\delta}_n}_r(V^{(l)}_n) = \varnothing, \partial B^{G^t_n}_r(V^{(l)}_n) \neq \varnothing\bigg)
\end{align}
which means that the boundary of the $r$-neighborhood of a uniformly chosen vertex is empty at time $t-\delta$ but non-empty at time point $t$. For that to happen there has to be a change in groups' statuses. Similarly as in the proof of Theorem \ref{dyn_loc_lim} we use the link with the union graph:
\begin{align} \label{factor1_final_lim}
    \mathbf{P}\bigg(\partial B^{G^{t-\delta}_n}_r(V^{(l)}_n) = \varnothing, \partial B^{G^t_n}_r(V^{(l)}_n) \neq \varnothing\mid  G_n^{[0,t]}\bigg) &= \mathbf{P}(\text{$\exists a\in B^{[0,t]}_r(V^{(l)}_n): a$ switches ON in $[t-\delta,t]\mid  G_n^{[0,t]}$})\\
    &= \bigg(1-e^{-\delta \sum_{a\in B^{[0,t]}_r(V^{(l)}_n)} \laoff}\bigg)\wedge1. \nonumber
\end{align}
Hence, 
\begin{align} \label{switch_event_one_giant}
    \mathbf{P}(J^{(r)}_n(t-\delta)=0,J^{(r)}_n(t)=1) \leq \mathbf{E}_{G_n^{[0,t]}}\bigg[\delta |B^{[0,t]}_r(V^{(l)}_n)| \bigg],
\end{align}
which can be bounded in the same way as the terms in (\ref{dyn_loc_lim_3_terms}) and hence converges to $0$ as $\delta\to0$. We can compute the complementary probability, $\mathbf{P}\big(\partial B^{G^{t-\delta}_n}_r(V^{(l)}_n) \neq \varnothing, \partial B^{G^t_n}_r(V^{(l)}_n) = \varnothing\big)$, using similar reasoning. Note that, conveniently, the probability of switching off is the same for all groups. Thanks to this and the independence of groups we obtain
\begin{align}
    \mathbf{P}\big(\partial B^{G^{t-\delta}_n}_r(V_n^{(l)}) \neq \varnothing, \partial B^{G^t_n}_r(V_n^{(l)}) = \varnothing\mid  G_n^{t-\delta}\big) &= \mathbf{P}(\text{all $a\in \partial B^{G^{t-\delta}_n}_r(V^{(l)}_n)$ switch OFF$\mid  G_n^{t-\delta}$})\\
    &= \prod_{a\in \partial B^{G^{t-\delta}_n}_r(V^{(l)}_n)} (1-e^{-\delta}),\nonumber
\end{align}
and thus,
\begin{align}  \label{factor2_final_lim}
    \mathbf{P}\big(\partial B^{G^{t-\delta}_n}_r(V_n^{(l)}) \neq \varnothing, \partial B^{G^t_n}_r(V_n^{(l)}) = \varnothing\big) = \mathbf{E}_{G_n^{t-\delta}}\bigg[(1-e^{-\delta})^{\#\{a:a\in \partial B^{G^{t-\delta}_n}_r(V^{(l)}_n)\}} \bigg].
\end{align}
With arguments similar to those used before, we can  show that the above vanishes as $\delta\to0$ and $n\to\infty$. Combining (\ref{factor1_final_lim}) and (\ref{factor2_final_lim}) we conclude that condition (\ref{cond_ii_cited}) holds.\\

\noindent \textbf{Condition (iii): Tightness of the original process.}
We want to show that for any $\varepsilon,\eta>0$ there exists $n_0\geq1$ and $\delta>0$ such that, for all $n\geq n_0$,
\begin{align} \label{cond_iii_cited}
    \mathbf{P}\Bigg( \sup_{(s,s_1,s_2)\in\mathscr{S}_{\delta}} \min \Big(\Big|J_n(s)-J_n(s_1)\Big|,\Big|J_n(s_2)-J_n(s)\Big|\Big) > \varepsilon \Bigg) \leq \eta,
\end{align}
with $\mathscr{S}_{\delta} = \{(s,s_1,s_2):s\in[s_1,s_2], |s_2-s_1|\leq \delta\}$. Note that since $\big(J_n(s)\big)_{s\in[0,t]}$ is an indicator process, $\min \Big(\Big|J_n(s)-J_n(s_1)\Big|,\Big|J_n(s_2)-J_n(s)\Big|\Big) > \varepsilon$ if and only if $\Big|J_n(s)-J_n(s_1)\Big|=\Big|J_n(s_2)-J_n(s)\Big|=1$, which is equivalent to $J_n(s)\neq J_n(s_1), J_n(s_2)\neq J_n(s)$. This means two mutually exclusive events might occur: either $J_n(s_1)=J_n(s_2)=1$ and $J_n(s)=0$, or the opposite $J_n(s_1)=J_n(s_2)=0$ and $J_n(s)=1$. Note that we can skip the supremum since for any $s\in[s_1,s_2]$ the value of $\Big|J_n(s)-J_n(s_1)\Big|$ and $\Big|J_n(s_2)-J_n(s)\Big|$ is at most 1. Taking this all into consideration, (\ref{cond_iii_cited}) becomes
\begin{align} \label{cond_iii_rewritten}
    \mathbf{P}&\big(\exists \hspace{0.1cm} s,s_1,s_2: \hspace{0.1cm} J_n(s_1)=J_n(s_2)=1,J_n(s)=0 \hspace{0.1cm}\text{or}\hspace{0.1cm}J_n(s_1)=J_n(s_2)=0,J_n(s)=1\big)
\end{align}
with $s\in[s_1,s_2]$ and $s_1,s_2$ such that $s_2-s_1<\delta$. We apply the same approach as in the proof of tightness for Theorem \ref{dyn_loc_lim}: we partition $[0,t]$ into intervals of length $\delta$ and denote
\begin{align}
    \mathbf{P}(R_l)=\mathbf{P}(\text{two jumps of $\big(J_n(s)\big)_{s\in[0,t]}$ in the $l$-th interval of length $\delta$}).
\end{align}
Then, by stationarity,
\begin{align}
    (\ref{cond_iii_rewritten}) = \mathbf{P}(\bigcup_{l=1}^{t/\delta} R_l) \leq \frac{t}{\delta}\mathbf{P}(R_1).
\end{align}
From the proof of the Condition (i) from Lemma \ref{conv_Skorokhod} we know that for some $n_0$ big enough for all $n\geq n_0$ and some $s\in[0,\delta]$
\begin{align}
    \mathbf{P}&\big(J_n(0)=J_n(\delta)=1,J_n(s)=0\big) = \lim_{r\to\infty} \mathbf{P}\big(J^{(r)}_n(0)=J^{(r)}_n(\delta)=1,J^{(r)}_n(s)=0\big)\\
    &=\lim_{r\to\infty} \mathbf{P}\bigg(\mathds{1}_{\big\{\partial B^{G^{0}_n}_r(V^{(l)}_n) \neq \varnothing\big\}}=\mathds{1}_{\big\{\partial B^{G^{\delta}_n}_r(V^{(l)}_n) \neq \varnothing\big\}} = 1, \mathds{1}_{\big\{\partial B^{G_n^s}_r(V^{(l)}_n) \neq \varnothing\big\}} = 0\bigg),\nonumber
\end{align}
and naturally, the analogous will hold for the complementary probability $\mathbf{P}\big(J_n(0)=J_n(\delta)=0, J_n(s)=1\big)$. Hence, from the proof of tightness from Theorem \ref{dyn_loc_lim} it follows that
\begin{align}
    \mathbf{P}(R_1) \leq \mathbf{P}(S_1) = o(\delta),
\end{align}
with
\begin{align}
    \mathbf{P}(S_l) = \mathbf{P}(\text{two jumps in the neighbourhood of $V^{(l)}_n$ in the $l$-th interval of length $\delta$}).
\end{align}
Thus, the required condition holds.\\

\noindent \textbf{Conclusion.} Since all three conditions of Lemma \ref{conv_Skorokhod} hold, the convergence follows.
\end{proof}

\subsection{Proof of convergence of the size of the largest group in \texorpdfstring{$[0,t]$}{}} \label{sec_pf_kmax}

\begin{proof}[Proof of Theorem \ref{k_max}]
The proof consists of two parts. We start by deriving convergence in distribution for $K^{[0,t]}_{\max}/n^{\frac{1}{\alpha}}$ and afterwards proceed to show that $\big(K^{[0,t]}_{\max}/n^{\frac{1}{\alpha}}\big)_{t\geq0}$ and the limiting process $\big(\kappa^{(0,t]}_{\max}\big)_{t\geq0}$ satisfy conditions of Lemma \ref{conv_Skorokhod}, which will yield the desired convergence.\\

\noindent\textbf{Part 1: Convergence in distribution.} To shorten the computations in the next part of the proof, we first derive convergence in distribution of the random variable $K^{[0,t]}_{\max}/n^{\frac{1}{\alpha}}$. We compute
\begin{align}
    \mathbf{P}(K^{[0,t]}_{\max}\leq k n^{1/\alpha}) = \mathbf{P}(\max\{K^{\{0\}}_{\max},K^{(0,t]}_{\max}\}\leq k n^{1/\alpha}) = \prod_{l>k n^{1/\alpha}}\prod_{a\in[n]_{l}}\piaoff \cdot \mathbf{P}(\text{$a$ never ON in $(0,t]$}). \nonumber
\end{align}
Hence,
\begin{align}
    \mathbf{P}&(K^{[0,t]}_{\max}\leq k n^{1/\alpha})=  \prod_{l>k n^{1/\alpha}}\prod_{a\in[n]_l} \frac{\ell_n^{l-1}}{\ell_n^{l-1}+l!p_l\prod_{i\in a} w_i} \times \prod_{l>k n^{1/\alpha}}\prod_{a\in[n]_l} e^{-\frac{l!p_l\prod_{i\in a} w_i}{\ell_n^{l-1}}t}. \label{distr_k_max_line1}
\end{align}
Note that
\begin{align*}
    &\prod_{l>k n^{1/\alpha}}\prod_{a\in[n]_l} e^{-\frac{l!p_l\prod_{i\in a} w_i}{\ell_n^{l-1}}t}=\prod_{l>k n^{1/\alpha}} e^{-t\sum_{j_1<...<j_l\in[n]}\frac{l!p_l\prod_{i\in a}w_i}{\ell_n^{l-1}}}\\
    =& \prod_{l>k n^{1/\alpha}} e^{-\frac{l!p_lt\ell_n}{l!}\sum_{j_1,...,j_l\in[n]}\frac{w_{j_1}...w_{j_l}}{\ell_n^{l}}} =e^{-t\ell_n\sum_{l>k n^{1/\alpha}}p_l} + o_{\mathbf{P}}(1) = e^{-t\ell_n(k n^{1/\alpha})^{-\alpha}} + o_{\mathbf{P}}(1),
\end{align*}
which plugged into (\ref{distr_k_max_line1}) yields
\begin{align}
    \mathbf{P}&(K^{[0,t]}_{\max}\leq k n^{1/\alpha})= e^{-t\ell_n(k n^{1/\alpha})^{-\alpha}} \prod_{l>k n^{1/\alpha}}\prod_{a\in[n]_l} \frac{\ell_n^{l-1}}{\ell_n^{l-1}+l!p_l\prod_{i\in a}w_i} + o_{\mathbf{P}}(1).
\end{align}
We have that
\begin{align}
    \lim_{n\to\infty} \prod_{l>k n^{1/\alpha}}\prod_{a\in[n]_l} \frac{\ell_n^{l-1}}{\ell_n^{l-1}+l!p_l\prod_{i\in a}w_i} = \lim_{n\to\infty} \prod_{l>k n^{1/\alpha}}\prod_{a\in[n]_l} \frac{1}{1+\frac{l!p_l\prod_{i\in a}w_i}{\ell_n^{l-1}}} = 1.
\end{align}
Further, since $\mathbf{E}[W_n]\to\mathbf{E}[W]$ as $n\to\infty$,
\begin{align}
    \lim_{n\to\infty} \ell_n (k n^{1/\alpha})^{-\alpha} = \lim_{n\to\infty} \frac{\ell_n}{k^{\alpha}n} = \lim_{n\to\infty} \frac{\mathbf{E}[W_n]}{k^{\alpha}} = \frac{\mathbf{E}[W]}{k^{\alpha}},  
\end{align}
we obtain for every $t\geq0$
\begin{align}
    \lim_{n\to\infty} \mathbf{P}\bigg(\frac{K^{[0,t]}_{\max} }{n^{1/\alpha}}\leq k \bigg)= e^{-tk^{-\alpha}\mathbf{E}[W]}.
\end{align}
Note that $g(k) = e^{-tk^{-\alpha}\mathbf{E}[W]}$ is a CDF of the Fr{\'e}chet distribution.

\noindent\textbf{Part 2: Verifying conditions of Lemma \ref{conv_Skorokhod}.}
We again check if our processes fulfil conditions $(i)-(ii)$ of Lemma \ref{conv_Skorokhod}.\\
\noindent\textbf{Condition (i): Convergence of the finite-dimensional distribution.} We want to show that for all $\{s_1,...,s_t\}\in[0,t]$ and $k_1,k_2,...,k_t\in\mathbf{N}$, as $n\to\infty$,
\begin{align} \label{k_max_desired_conv}
    \mathbf{P}\bigg(\frac{K^{[0,s_1]}_{\max}}{n^{1/\alpha}}\leq k_1, \frac{K^{[0,s_2]}_{\max}}{n^{1/\alpha}}\leq k_2,..., \frac{K^{[0,s_t]}_{\max} }{n^{1/\alpha}}\leq k_t \bigg) \longrightarrow \mathbf{P}\bigg(\kappa^{(0,s_1]}_{\max} \leq k_1,\kappa^{(0,s_2]}_{\max} \leq k_2,...,\kappa^{(0,s_t]}_{\max} \leq k_t \bigg).
\end{align}
Note that for every $s_i,s_j\in[0,t]: s_i<s_j$ it holds that $\mathbf{P}(K^{[0,s_i]}_{\max}>K^{[0,s_j]}_{\max})=0.$ Thus, computing the joint distribution function for every non-decreasing sequence $k_1,k_2,...,k_t\in\mathbf{N}$ is straightforward. For other sequences, we perform the following loop:
\begin{align}
    &i=1,\\
    &k_i=\min\{k_1,k_2,...,k_t\},\nonumber\\
    &\Bar{m}=\arg\min k_i,\nonumber\\
    &\text{Remove $\{k_1,k_2,...,k_{\Bar{m}-1}\}$ from the sequence},\nonumber\\
    &\text{Iterate for $\{k_{\Bar{m}+1},...,k_t\}$}.\nonumber
\end{align}
Denote the modified $k_1,k_2,...,k_t$ by $\bar{k}_1,\bar{k}_2,...,\bar{k}_t$ and note that the modified version is a non-decreasing sequence. It follows that
\begin{align}
    \mathbf{P}\bigg(\frac{K^{[0,s_1]}_{\max}}{n^{1/\alpha}}\leq k_1, \frac{K^{[0,s_2]}_{\max}}{n^{1/\alpha}}\leq k_2,..., \frac{K^{[0,s_t]}_{\max} }{n^{1/\alpha}}\leq k_t \bigg) = \mathbf{P}\bigg(\frac{K^{[0,s_1]}_{\max}}{n^{1/\alpha}}\leq \bar{k}_1, \frac{K^{[0,s_2]}_{\max}}{n^{1/\alpha}}\leq \bar{k}_2,..., \frac{K^{[0,s_t]}_{\max} }{n^{1/\alpha}}\leq \bar{k}_t \bigg).
\end{align}
Hence, we further assume that we are working with non-decreasing sequences. Note that, by definition (see (\ref{max_gr_size_split})), for a partition $\{0,s_1,s_2,...,s_{t-1},s_t\}$ of the time interval $[0,s_t]$ it holds
\begin{align}
    K^{[0,s_t]}_{\max}=\max\{K^{[0,s_1]}_{\max},K^{(s_1,s_2]}_{\max},...,K^{(s_{t-1},s_t]}_{\max}\}.
\end{align}
Thus,
\begin{align}
    \mathbf{P}&(K^{[0,s_1]}_{\max}\leq k_1 n^{1/\alpha}, K^{[0,s_2]}_{\max}\leq k_2 n^{1/\alpha},..., K^{[0,s_t]}_{\max}\leq k_t n^{1/\alpha})\\
    =& \mathbf{P}(K^{[0,s_1]}_{\max}\leq k_1 n^{1/\alpha}, \max\{K^{[0,s_1]}_{\max},K^{(s_1,s_2]}_{\max}\}\leq k_2 n^{1/\alpha},..., \max\{K^{[0,s_1]}_{\max},K^{(s_1,s_2]}_{\max},...,K^{(s_{t-1},s_t]}_{\max}\}\leq k_t n^{1/\alpha}) \nonumber\\
    =& \mathbf{P}(K^{[0,s_1]}_{\max}\leq k_1 n^{1/\alpha},K^{[0,s_1]}_{\max}\leq k_2 n^{1/\alpha}, K^{(s_1,s_2]}_{\max}\leq k_2n^{1/\alpha},...,K^{[0,s_1]}_{\max}\leq k_t n^{1/\alpha}, K^{(s_1,s_2]}_{\max}\leq k_tn^{1/\alpha}\nonumber\\
    &\hspace{10.7cm},...,K^{(s_{t-1},s_t]}_{\max}\leq k_t n^{1/\alpha}) \nonumber\\
    =& \mathbf{P}(K^{[0,s_1]}_{\max}\leq k_1 n^{1/\alpha}, K^{(s_1,s_2]}_{\max}\leq k_2n^{1/\alpha},...,K^{(s_{t-1},s_t]}_{\max}\leq k_t n^{1/\alpha}). \nonumber
\end{align}
Note that for non-overlapping time intervals, $\big(K^{(s_i,s_{i+1}]}_{\max}\big)_{i}$ are all independent. Hence, using the computation from Part 1 we obtain
\begin{align}
    \mathbf{P}&(K^{[0,s_1]}_{\max}\leq k_1 n^{1/\alpha}, K^{(s_1,s_2]}_{\max}\leq k_2n^{1/\alpha},...,K^{(s_{t-1},s_t]}_{\max}\leq k_t n^{1/\alpha}) = \mathbf{P}(K^{[0,s_1]}_{\max}\leq k_1 n^{1/\alpha}) \prod_{i=2}^t \mathbf{P}(K^{(s_{i-1},s_i]}_{\max}\leq k_in^{1/\alpha}) \nonumber\\
    &= \mathbf{P}(K^{[0,s_1]}_{\max}\leq k_1 n^{1/\alpha}) \prod_{i=2}^t \Big(\prod_{l>k_i n^{1/\alpha}}\prod_{a\in[n]_l} \mathbf{P}\big(\text{$a$ never ON in $(s_{i-1},s_i]$}\big) \Big)\\
    &=e^{-s_1\mathbf{E}[W_n]k_1^{-\alpha}} \prod_{l>k_1 n^{1/\alpha}}\prod_{a\in[n]_l} \frac{\ell_n^{l-1}}{\ell_n^{l-1}+l!p_l\prod_{i\in a}w_i} \prod_{i=2}^t e^{-(s_i-s_{i-1})\mathbf{E}[W_n]k_i^{-\alpha}}\nonumber\\
    &\stackrel{n\to\infty}{\longrightarrow} e^{-s_1\mathbf{E}[W]k_1^{-\alpha}} \prod_{i=2}^t e^{-(s_i-s_{i-1})\mathbf{E}[W]k_i^{-\alpha}}.\nonumber
\end{align}
Note that, since $k_1\leq k_2\leq...\leq k_t$,
\begin{align}
    &e^{-s_1\mathbf{E}[W]k_1^{-\alpha}} \prod_{i=2}^t e^{-(s_i-s_{i-1})\mathbf{E}[W]k_i^{-\alpha}} = \mathbf{P}(\kappa^{(0,s_1]}_{\max}\leq k_1) \prod_{i=2}^t \mathbf{P}(\kappa^{(s_{i-1},s_i]}_{\max}\leq k_i)\\
    &= \mathbf{P}(\kappa^{(0,s_1]}_{\max}\leq k_1,\kappa^{(s_1,s_2]}_{\max}\leq k_2,...,\kappa^{(s_{t-1},s_t]}_{\max}\leq k_t)\nonumber\\
    &= \mathbf{P}(\kappa^{(0,s_1]}_{\max}\leq k_1,\max\{\kappa^{(0,s_1]}_{\max},\kappa^{(s_1,s_2]}_{\max}\}\leq k_2,...,\max\{\kappa^{(0,s_1]}_{\max},\kappa^{(s_1,s_2]}_{\max},...,\kappa^{(s_{t-1},s_t]}_{\max}\}\leq k_t)\nonumber\\
    &=\mathbf{P}(\kappa^{(0,s_1]}_{\max}\leq k_1,\kappa^{(0,s_2]}_{\max}\leq k_2,...,\kappa^{(0,s_t]}_{\max}\leq k_t).\nonumber
\end{align}
This proves the desired convergence from (\ref{k_max_desired_conv}).\\

\noindent\textbf{Condition (ii): Tightness of the limiting process.} We want to show that
\begin{align}
    \lim_{\delta\to0}\mathbf{P}\big( \big|\kappa^{(0,t]}_{\max}-\kappa^{(0,t-\delta]}_{\max}\big|>\varepsilon\big) = 0,
\end{align}
which is equivalent to
\begin{align}
    \lim_{\delta\to0}\mathbf{P}\big( \kappa^{(0,t]}_{\max}>\kappa^{(0,t-\delta]}_{\max}+\varepsilon\big) = 0.
\end{align}
By the proof of condition $(i)$
\begin{align}    \label{goal_lim_ii_k_max}  
\lim_{\delta\to0}\mathbf{P}\big( \kappa^{(0,t]}_{\max}>\kappa^{(0,t-\delta]}_{\max}+\varepsilon\big) = \lim_{\delta\to0}\lim_{n\to\infty}
\mathbf{P}\big( K^{[0,t]}_{\max}>K^{[0,t-\delta]}_{\max}+\varepsilon n^{1/\alpha}\big).
\end{align}
We compute 
\begin{align} 
    \mathbf{P}&\big(K^{[0,t]}_{\max}>K^{[0,t-\delta]}_{\max}+\varepsilon n^{1/\alpha}\big) = \sum_{k=2}^{\infty} \mathbf{P}(\text{group bigger than $k+\varepsilon n^{1/\alpha}$ arrives in $(t-\delta,t]$})\mathbf{P}(K^{[0,t-\delta]}_{\max}=k)\\
    &=\sum_{k=2}^{\infty} \big(1-e^{-\delta\ell_n\sum_{l>k+\varepsilon n^{1/\alpha}}p_l}\big)\mathbf{P}(K^{[0,t-\delta]}_{\max}=k) \leq \delta\ell_n\sum_{k=2}^{\infty} (k+\varepsilon n^{1/\alpha})^{-\alpha}\mathbf{P}(K^{[0,t-\delta]}_{\max}=k)\nonumber\\
    &\leq \delta\ell_n\sum_{k=2}^{\infty} (\varepsilon n^{1/\alpha})^{-\alpha}\mathbf{P}(K^{[0,t-\delta]}_{\max}=k) = \frac{\delta\ell_n}{\varepsilon^{\alpha}n}=\frac{\delta\mathbf{E}[W_n]}{\varepsilon^{\alpha}}.\nonumber
\end{align}
Substituting into (\ref{goal_lim_ii_k_max}) yields
\begin{align}
    \lim_{\delta\to0}\mathbf{P}\big( \kappa^{(0,t]}_{\max}>\kappa^{(0,t-\delta]}_{\max}+\varepsilon\big) = \lim_{\delta\to0}\lim_{n\to\infty}\frac{\delta\mathbf{E}[W_n]}{\varepsilon^{\alpha}} = \lim_{\delta\to0}\frac{\delta\mathbf{E}[W]}{\varepsilon^{\alpha}}=0.
\end{align}

\noindent\textbf{Condition (iii): Tightness of the original process.}
We want to show that for any $\varepsilon,\eta>0$ there exists $n_0\geq1$ and $\delta>0$ such that, for all $n\geq n_0$,
\begin{align}
    \mathbf{P}\Bigg( \sup_{a,s_1,s_2:s\in[s_1,s_2],s_2-s_1<\delta} \min \Bigg(\Bigg|\frac{K^{[0,s]}_{\max}}{n^{1/\alpha}}-\frac{K^{[0,s_1]}_{\max}}{n^{1/\alpha}}\Bigg|,\Bigg|\frac{K^{[0,s_2]}_{\max}}{n^{1/\alpha}}-\frac{K^{[0,s]}_{\max}}{n^{1/\alpha}}\Bigg|\Bigg) > \varepsilon \Bigg) \leq \eta.
\end{align}
Note that
\begin{align}
    \mathbf{P}&\Bigg( \sup_{s\in[s_1,s_2],s_2-s_1<\delta} \min \Bigg(\Bigg|\frac{K^{[0,s]}_{\max}}{n^{1/\alpha}}-\frac{K^{[0,s_1]}_{\max}}{n^{1/\alpha}}\Bigg|,\Bigg|\frac{K^{[0,s_2]}_{\max}}{n^{1/\alpha}}-\frac{K^{[0,s]}_{\max}}{n^{1/\alpha}}\Bigg|\Bigg) > \varepsilon \Bigg)\\
    &\leq \mathbf{P}\Bigg( \exists s\in[s_1,s_2],s_2-s_1<\delta: \min \Big(\Big|K^{[0,s]}_{\max}-K^{[0,s_1]}_{\max}\Big|,\Big|K^{[0,s_2]}_{\max}-K^{[0,s]}_{\max}\Big|\Big) > \varepsilon n^{1/\alpha} \Bigg).\nonumber
\end{align}
For the minimum of two terms to be bigger than $\varepsilon n^{1/\alpha}$, both of them have to be bigger than $\varepsilon n^{1/\alpha}$. As we are dealing with a non-decreasing random variable, this can only happen if a group which is bigger by $\varepsilon n^{1/\alpha}$ than the so far largest group arrives in $(s_1,s]$ and then the same happens in $(s,s_2]$. Once again we partition $[0,t]$ into intervals of length $\delta$ and apply stationarity to deduce
\begin{align}
    \mathbf{P}&\Bigg( \sup_{s\in[s_1,s_2],s_2-s_1<\delta} \min \Bigg(\Bigg|\frac{K^{[0,s]}_{\max}}{n^{1/\alpha}}-\frac{K^{[0,s_1]}_{\max}}{n^{1/\alpha}}\Bigg|,\Bigg|\frac{K^{[0,s_2]}_{\max}}{n^{1/\alpha}}-\frac{K^{[0,s]}_{\max}}{n^{1/\alpha}}\Bigg|\Bigg) > \varepsilon \Bigg) \leq \frac{t}{\delta}\mathbf{P}(T_1),
\end{align}
where $T_l$ is the event that the dynamic largest group process encounters two jumps in the $l$-th time interval of lenght $\delta$.
We compute
\begin{align}
    \mathbf{P}&(T_1)\leq \mathbf{P}(K^{[0,s]}_{\max}>K^{[0,s_1]}_{\max}+\varepsilon n^{1/\alpha},K^{[0,s_2]}_{\max}>K^{[0,s]}_{\max}+\varepsilon n^{1/\alpha})\\
    &\sum_{k=2}^{\infty}\sum_{l=k+1}^{\infty}\mathbf{P}(\text{$a: |a|>l+\varepsilon n^{1/\alpha}$ arrives in $(s_1,s]$})\mathbf{P}(\text{$a: |a|=l+\varepsilon n^{1/\alpha}$ arrives in $(s,s_2]$})\mathbf{P}(K^{[0,t-\delta]}_{\max}=k).\nonumber
\end{align}
From previous points we know that for $l\geq 2$ $\mathbf{P}(\text{$a: |a|>l+\varepsilon n^{1/\alpha}$ arrives in $(s_1,s]$})\leq \delta\ell_n(\varepsilon n^{1/\alpha})^{-\alpha}$. Substituting this bound in the above calculation yields
\begin{align}
    \mathbf{P}(T_1)&\leq\sum_{k=2}^{\infty}\sum_{l=k+1}^{\infty}\delta\ell_n(\varepsilon n^{1/\alpha})^{-\alpha}\mathbf{P}(\text{$a: |a|=l+\varepsilon n^{1/\alpha}$ arrives in $(s,s_2]$})\mathbf{P}(K^{[0,s_1]}_{\max}=k)\\
&=\delta\ell_n(\varepsilon n^{1/\alpha})^{-\alpha}\leq\sum_{k=2}^{\infty}\mathbf{P}(\text{$a: |a|>k+\varepsilon n^{1/\alpha}$ arrives in $(s,s_2]$})\mathbf{P}(K^{[0,s_1]}_{\max}=k)\nonumber\\
&\leq \bigg(\delta\ell_n(\varepsilon n^{1/\alpha})^{-\alpha}\bigg)^2=\frac{\delta^2(\mathbf{E}[W_n])^2}{\varepsilon^{2\alpha}}\nonumber
\end{align}
Note that for $n_0$ big enough we have that $\mathbf{E}[W_n]$ is close to $\mathbf{E}[W]$ for any $n\geq n_0$. Hence, if we take $\delta < \frac{\eta\varepsilon^{2\alpha}}{(\mathbf{E}[W])^2}$ the condition will hold for any pair $\eta,\varepsilon>0$.
\end{proof}

\begin{wrapfigure}{r}{4cm}
\label{wrap-fig:1}
\includegraphics[width=0.17\textwidth]{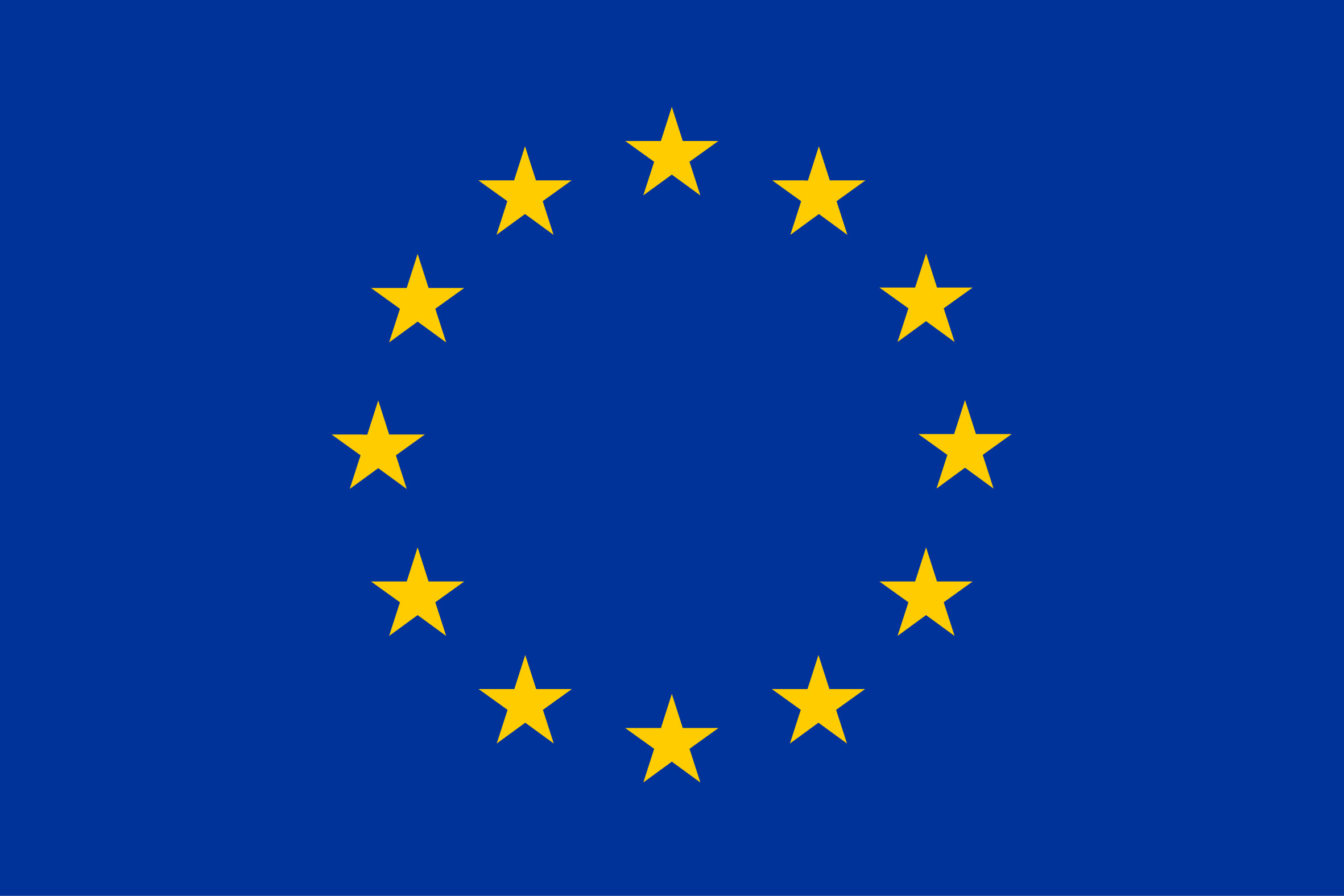}
\end{wrapfigure}
\paragraph{\bf Acknowledgement.} The work of MM is supported by the European Union’s Horizon 2020 research and innovation programme under the Marie Skłodowska-Curie grant agreement no. 945045, and by the NWO Gravitation project NETWORKS under grant no. 024.002.003.\\
The work of RvdH is supported in parts by the NWO through the Gravitation {\sc Networks} grant 024.002.003.

\printbibliography

@unpublished{Milewska2023,
  author = {van der Hofstad, R. and Milewska, M. and Zwart, B.},
  title = {Dynamic random intersection graph: dynamic local convergence and giant structure},
  note = {Preprint},
  year = {2023},
}

@Article{Angel2016,
  author  = {Angel, O. and van der Hofstad, R. and Holmgren, C.},
  journal = {Ann. Inst. Henri Poincar\'{e} Probab. Stat.},
  title   = {Limit laws for self-loops and multiple edges in the configuration model},
  volume  = {55},
  year    = {2016},
  number = {3},
  pages = {1509--1530},
  doi     = {10.1214/18-AIHP926},
}

@Article{Bollobas_1980,
  author  = {Bollob{\'{a}}s, B.},
  title   = {A probabilistic proof of an asymptotic formula for the number of labelled regular graphs},
  journal = {European J. Combin.},
  volume  = {1},
  number = {4},
  pages   = {311-316},
  DOI = {10.1016/S0195-6698(80)80030-8},
  year    = {1980},
}

@Article{Bollobas_2007,
  author  = {Bollob{\'{a}}s, B. and Janson, S. and Riordan, O.},
  year    = {2007},
  volume  = {31},
  number  = {1},
  pages   = {3-122},
  title   = {The phase transition in inhomogeneous random graphs},
  journal = {Random Structures Algorithms},
  DOI = {10.1002/rsa.20168},
}

@Article{Britton2006,
  author = {Britton, T. and Deijfen, M. and Martin-L{\"o}f, A.},
  journal = {J. Stat. Phys.},
  title = {Generating simple random graphs with prescribed degree distribution},
  year = {2006},
  volume = {124},
  number = {6},
  pages = {1377-1397},
  DOI = {10.1007/s10955-006-9168-x},
}

@book{Chung-Lu,
  author = {Chung, F. and Lu, L.},
  title = {Complex graphs and networks},
  year = {2006},
  publisher = {Published for the Conference Board of the Mathematical
              Sciences, Washington, DC; by the American Mathematical
              Society, Providence, RI},
}

@Article{Deijfen2009,
  author  = {Deijfen, M. and Kets, W.},
  journal = {Probab. Engrg. Inform. Sci.},
  title   = {Random intersection graphs with tunable degree distribution and clustering},
  year    = {2009},
  volume  = {23},
  number = {4},
  pages = {661-674},
  DOI = {10.1017/S0269964809990064},
}

@book{Hofstad2016,
  author     = {van der Hofstad, R.},
  publisher  = {Cambridge University Press},
  title      = {Random graphs and complex networks Volume 1},
  year       = {2016},
}

@Article{Hofstad2018,
  author    = {van der Hofstad, R. and Komj{\'a}thy, J. and Vadon, V.},
  journal   = {Adv. in Appl. Probab.},
  title     = {Random intersection graphs with communities},
  year      = {2021},
  volume    = {53},
  number    = {4},
  pages     = {1061–1089},
  doi       = {10.1017/apr.2021.12},
}

@Article{Hofstad2022,
  author    = {van der Hofstad, R. and Komj{\'a}thy, J. and Vadon, V.},
  journal   = {Random Structures Algorithms},
  title     = {Phase transition in random intersection graphs with communities},
  year      = {2022},
  number    = {3},
  pages     = {406--461},
  volume    = {60},
  DOI = {10.1002/rsa.21063},
}

@book{Hofstad2023,
  author  = {van der Hofstad, R.},
  year    = {2024},
  title   = {Random graphs and complex networks. Volume 2.},
  note    = {In preparation},
  url     = {http://www.win.tue.nl/∼rhofstad/NotesRGCNII.pdf},
}

@article{Janson_2010,
  author = {Janson, S.},
  title = {Asymptotic equivalence and contiguity of some random graphs},
  journal = {Random Structures Algorithms},
  year = {2010},
  volume = {36},
  pages = {26--45},
  DOI = {10.1002/rsa.20297},
}

@Article{Networks,
  author = {Girvan, M. and Newman, M. E. J.},
  journal = {Proc. Natl. Acad. Sci. USA},
  title = {Community structure in social and biological networks},
  year = {2002},
  pages = {7821--7826},
  volume = {99},
  number = {12},
  DOI = {10.1073/pnas.122653799},
}

@Article{Molloy_Reed1995,
  author = {Molloy, M. and Reed, B.},
  title = {A critical point for random graphs with a given degree sequence},
  year = {1995},
  journal = {Random Structures Algorithms},
  volume = {6},
  number = {2-3},
  pages = {161--179},
  DOI = {10.1002/rsa.3240060204},
}

@Article{Molloy_Reed1998,
  author = {Molloy, M. and Reed, B.},
  title = {The size of the giant component of a random graph with a given degree sequence},
  year = {1998},
  journal = {Combin. Probab. Comput.},
  volume = {7},
  number = {3},
  pages = {295--305},
  DOI = {10.1017/S0963548398003526},
}

@book{NetworksBook,
  author = {Newman, M. E. J.},
  publisher = {Oxford University Press},
  title = {Networks},
  year = {2010},
}

@Article{Norros_2006,
  author = {Norros, I. and Reittu, H.},
  title = {On a conditionally poissonian graph process},
  journal = {Adv. in Appl. Probab.},
  year = {2006},
  pages = {59--75},
  volume = {38},
  number = {1},
  DOI = {10.1239/aap/1143936140},
}

@Article{Trapman2007,
  author = {Trapman, P.},
  title = {On analytical approaches to epidemics on networks},
  year = {2007},
  journal = {Theoret. Population Biol.},
  volume = {71},
  number = {2},
  pages = {160--173},
}

@Article{Blackburn2009,
  author = {Blackburn, S. R. and Gerke, S.},
  title = {Connectivity of the uniform random intersection graph},
  year = {2009},
  journal = {Discrete Math.},
  volume = {309},
  number = {16},
  pages = {5130–-5140},
  DOI = {10.1016/j.disc.2009.03.042},
}

@Article{Bloznelis2010,
  author = {Bloznelis, M.},
  title = {Component evolution in general random intersection graphs},
  year = {2010},
  journal = {SIAM J. Discrete Math.},
  volume = {24},
  number = {2},
  pages = {639–-654},
  DOI = {10.1137/080713756},
}

@Article{Bloznelis2013,
  author = {Bloznelis, M.},
  title = {Degree and clustering coefficient in sparse random intersection graphs},
  year = {2013},
  journal = {Ann. Appl. Probab.},
  volume = {23},
  number = {3},
  pages = {1254–-1289},
  DOI = {10.1214/12-AAP874},
}

@incollection{Bloznelis2015,
  author   = {Bloznelis, M. and Godehardt, E. and Jaworski, J. and Kurauskas, V. and Rybarczyk, K.},
  title    = {Recent progress in complex network analysis: properties of random intersection graphs},
  booktitle = {Data science, learning by latent structures, and knowledge discovery},
  series = {Stud. Classification Data Anal. Knowledge Organ.},
  pages = {79--88},
  publisher = {Springer, Heidelberg},
  year = {2015},
}

@Article{Bloznelis2017,
  author = {Bloznelis, M.},
  title = {Degree-degree distribution in a power law random intersection graph with clustering},
  year = {2017},
  journal = {Internet Math.},
  pages = {25},
  doi = {10.24166/im.03.2017},
}

@Article{Bloznelis_Damarackas2013,
  author = {Bloznelis, M. and Damarackas, J.},
  title = {Degree distribution of an inhomogeneous random intersection graph},
  year = {2013},
  journal = {Electron. J. Combin.},
  volume = {20},
  number = {3},
  DOI = {10.37236/2786},
}

@Article{Bloznelis_Jaworski_Kurauskas2013,
  author = {Bloznelis, M. and Jaworski, J. and Kurauskas, V.},
  title = {Assortativity and clustering of sparse random intersection graphs},
  year = {2013},
  journal = {Electron. J. Probab.},
  volume = {18},
  number = {no. 38, 24},
  DOI = {10.1214/EJP.v18-2277},
}

@article{Fill2000,
  author = {Fill, J. A. and Scheinerman, E. R. and Singer-Cohen, K. B.},
  title = {Random intersection graphs when {$m=\omega(n)$}: an equivalence theorem relating the evolution of the {$G(n,m,p)$} and {$G(n,p)$} models},
  year = {2000},
  journal = {Random Structures Algorithms},
  volume = {16},
  number = {2},
  pages = {156-–176},
  DOI = {10.1002/(SICI)1098-2418(200003)16:2<156::AID-RSA3>3.3.CO;2-8},
}

@Article{Lelarge2015,
  author = {Coupechoux, E. and Lelarge, M.},
  title = {Contagions in random networks with overlapping communities},
  journal = {Adv. in Appl. Probab.},
  volume = {47},
  number = {4},
  year = {2015},
  pages = {973--988},
  DOI = {10.1239/aap/1449859796},
}

@incollection{Godehardt_Jaworski,
  author = {Godehardt,  E. and Jaworski, J.},
  title = {Two models of random intersection graphs for classification},
  booktitle = {Exploratory data analysis in empirical research},
  series = {Stud. Classification Data Anal. Knowledge Organ.},
  publisher = {Springer, Berlin},
  year = {2003},
  pages = {67–-81},
}

@Article{Karonski1999,
  author = {Karonski, M. and Scheinerman, E. R. and Singer-Cohen, K. B.},
  title = {On random intersection graphs: The subgraph problem},
  journal = {Combin. Probab. Comput.},
  volume = {8},
  number = {1-2},
  year = {1999},
  pages = {131--159},
  DOI = {10.1017/S0963548398003459},
}

@Article{Newman2003,
  author = {Newman, M. E. J.},
  title = {Properties of highly clustered networks},
  journal = {Phys. Rev. E},
  volume = {68},
  number = {2},
  year = {2003},
  pages = {131-159},
}

@Article{Rybarczyk2011,
  author = {Rybarczyk, K.},
  title = {Diameter, connectivity, and phase transition of the uniform random intersection graph},
  journal = {Discrete Math.},
  volume = {311},
  number = {17},
  year = {2011},
  pages = {1998-2019},
  DOI = {10.1016/j.disc.2011.05.029},
}

@book{Singer_thesis,
  author = {Singer, K. B.},
  title = {Random intersection graphs},
  note = {Thesis (Ph.D.)--The Johns Hopkins University},
  publisher = {ProQuest LLC, Ann Arbor, MI},
  year = {1996},
  pages = {219},
}

@Article{Yule1925,
  author = {Yule, G. U.},
  title = {A mathematical theory of evolution, based on the conclusions of Dr. J. C. Willis, F. R. S},
  year = {1925},
  journal = {Phil. Trans. R. Soc. London, B},
  volume = {213},
  pages = {21–-87},
}

@Article{Simon1955,
  author = {Simon, H. A.},
  title = {On a class of skew distribution functions},
  year = {1955},
  journal = {Biometrika},
  volume = {42},
  pages = {425--440},
  DOI = {10.1093/biomet/42.3-4.425},
}

@Article{Barabasi1999,
  author = {Barab{\'{a}}si, A.-L. and Albert, R.},
  title = {Emergence of scaling in random networks},
  journal = {Science},
  volume = {286},
  number = {5439},
  year = {1999},
  pages = {509-512},
  DOI = {10.1126/science.286.5439.509},
}

@Article{Benjamini2001,
  author = {Benjamini, I. and Schramm, O.},
  title = {Recurrence of distributional limits of finite planar graphs},
  year = {2001},
  journal = {Electron. J. Probab.},
  volume = {6},
  pages = {no. 23, 13},
  DOI = {10.1214/EJP.v6-96},
}

@incollection{Aldous2004,
  author = {Aldous, D. and Steele, J. M.},
  title = {The objective method: probabilistic combinatorial optimization and local weak convergence},
  year = {2004},
  booktitle = {Probability on discrete structures},
  series = {Encyclopaedia Math. Sci.},
  volume = {110},
  pages = {1--72},
  publisher = {Springer, Berlin.},
}

@Article{ER_1960,
  author = {Erd{\H o}s, P. and R{\'e}nyi, A.},
  title = {On the evolution of random graphs},
  journal = {Magyar Tud. Akad. Mat. Kutat\'{o} Int. K\"{o}zl.},
  volume = {5},
  year = {1960},
  pages = {17--61},
}

@Article{ChungLu2002,
  author = {Chung, F. and Lu, L.},
  title = {Connected components in random graphs with given expected degree sequences},
  journal = {Ann. Comb.},
  volume = {6},
  number = {2},
  year = {2002},
  pages = {125–-145},
  doi = {10.1007/PL00012580},
}

@Article{ChungLu2006,
  author = {Chung, F. and Lu, L.},
  title = {The volume of the giant component of a random graph with given expected degrees},
  journal = {SIAM J. Discrete Math.},
  volume = {20},
  number = {2},
  year = {2006},
  pages = {395–-411},
  DOI = {10.1137/050630106},
}

@Article{Bollobas2015,
  author = {Bollob{\'a}s, B. and Riordan, O.},
  title = {An old approach to the giant component problem},
  journal = {J. Combin. Theory Ser. B},
  volume = {113},
  year = {2015},
  pages = {236–-260},
  DOI = {10.1016/j.jctb.2015.03.002},
}

@Article{JansonLuczak2009,
  author = {Janson, S. and Luczak, M. J.},
  title = {A new approach to the giant component problem},
  journal = {Random Structures Algorithms},
  volume = {34},
  number = {2},
  year = {2009},
  pages = {197–216},
  DOI = {10.1002/rsa.20231},
}

@article{DynamicPercol,
title = {Dynamical Percolation},
author = {H{\"a}ggstr{\"o}m, O. and Peres, Y. and Steif, J.},
journal = {Ann. Inst. H. Poincar\'{e} Probab. Statist.},
volume = {33},
number = {4},
pages = {497-528},
year = {1997},
DOI = {10.1016/S0246-0203(97)80103-3},
}

@incollection{DynamicPercol2,
author = {Steif, J.},
title = {A survey of dynamical percolation},
booktitle = {Fractal geometry and stochastics {IV}},
series = {Progr. Probab.},
volume = {61},
publisher = {Birkh\"{a}user Verlag, Basel},
pages = {145--174},
year = {2009},
}

@article{AvenadHvdH2018,
author = {Avena, L. and G\"ulda\c{s}, H. and van der Hofstad, R. and den Hollander, F.},
year = {2019},
title = {Random walks on dynamic configuration models: a trichotomy},
volume = {129},
number = {9},
pages = {3360--3375},
journal = {Stochastic Process. Appl.},
DOI = {10.1016/j.spa.2018.09.010},
}

@article{Avena2022,
author = {Avena, L. and G{\"u}lda{\c{s}}, H. and van der Hofstad, R. and den Hollander, F. and Nagy, O.},
title = {Linking the mixing times of random walks on static and dynamic random graphs},
journal = {Stochastic Process. Appl.},
year = {2022},
volume = {153},
pages = {145--182},
DOI = {10.1016/j.spa.2022.07.009},
}

@book{Billingsley2013,
  author = {Billingsley, P.},
  title= {Convergence of probability measures},
  publisher = {John Wiley \& Sons, Inc., New York},
  year = {1999},
}

@book{kallenberg2002foundations,
  author = {Kallenberg, O.},
  publisher = {Springer-Verlag, New York},
  title = {Foundations of modern probability},
  year = {2002},
}

@Unpublished{dynweaklimit2023,
      title={Local weak limit of dynamical inhomogeneous random graphs}, 
      author={Dort, L. and Jacob, E.},
      year={2023},
      eprint={2303.17437},
      archivePrefix={arXiv},
}

\appendix
\newpage
\section{Proof of the link between \texorpdfstring{$\bgrg$}{} and \texorpdfstring{$\bcm$}{} from \texorpdfstring{\cite{Hofstad2018}}{}} 
\label{appendix_two_models_relation}
Here we prove the theorem that lies at the heart of our results - the fact that, under certain conditions, $\bgrg$ and $\mathrm{BCM}_n(\bm{d}^{(l)},\bm{d}^{(r)})$ are equivalent. The proof follows in four steps: First, we show that the $\bgrg$ conditioned on its degree sequence is uniform. Secondly, we show that $\mathrm{BCM}_n(\bm{d}^{(l)},\bm{d}^{(r)})$ conditioned on simplicity is uniform. We also state regularity conditions that allow us to draw an even stronger link between the two models. Finally, we conclude that under such circumstances $\bgrg$ and $\mathrm{BCM}_n(\bm{d}^{(l)},\bm{d}^{(r)})$ are equivalent. 

\subsection{\texorpdfstring{$\bgrg$}{} conditioned on degree sequence is uniform}
We adapt the derivation of a similar result for $\mathrm{GRG}_n(\bm{w})$. (See \textup{\cite[Section 6.6]{Hofstad2016}}). Note that $\bgrg$ is entirely determined by the group activity, i.e., if we know which group is active we automatically know which vertices are in it. Hence, we can encode the probability of $\bgrg$ taking a particular form via a sequence of indicator random variables: Take $x = (x_{a})_{a\in\cup_{k\geq2}[n]_k}$ - a sequence of $0$s and $1$s - and $X = (X_{a})_{a\in\cup_{k\geq2}[n]_k}$ - sequence of independent random variables describing the existence of particular groups i.e.,
\begin{align}
    \mathbf{P}(X_{a}=1) = 1-\mathbf{P}(X_{a}=0) = \pi^a_{\text{\rm{ON}}}.
\end{align}
Then, we have the following identification of the law of $\bgrg$:
\begin{proposition}[$\bgrg$ as a function of left- and right-degrees.] \label{fct_left_right}
The probability that the sequence $X = (X_{a})_{a\in\cup_{k\geq2}[n]_k}$ takes a form $x = (x_{a})_{a\in\cup_{k\geq2}[n]_k}$ can be expressed as a function of left- and right-degree sequences $(\bm{d}^{(l)},\bm{d}^{(r)})=\big((d^{(l)}_i)_{i\in[n]},(d^{(r)}_a)_{a\in\cup_{k\geq2}[n]_k}\big)$:
\begin{align}
    \mathbf{P}(X=x) = H\big((d^{(l)}(x),d^{(r)}(x))\big)\big) \cdot \bigg(\prod_{a\in\cup_{k\geq2}[n]_k} (1+\laoff)\bigg)^{-1},
\end{align}
where $d_i^{(l)}(x)=\sum_{a\in[n]_k:a\ni i}x_a$, $d_a^{(r)}(x) = |a|\cdot x_a$ and $H$ satisfies
\begin{align}
    H\big((d^{(l)}(x),d^{(r)}(x))\big)\big) = \prod_{i\in[n]} w_i^{d^{(l)}_i(x)} \prod_{a\in\cup_{k\geq2}[n]_k} \frac{f\big(d^{(r)}_a(x)\big)}{\ell_n^{d^{(r)}_a(x)\big(1-\frac{1}{|a|}\big)}}.
\end{align}
\end{proposition}
\begin{proof}
Taking $X = (X_{a})_{a\in\cup_{k\geq2}[n]_k}$ and $x = (x_{a})_{a\in\cup_{k\geq2}[n]_k}$ as above, we obtain
\begin{align} \label{our_model}
    \mathbf{P}(X=x) & = \prod_{a\in\cup_{k\geq2}[n]_k} (\pi^a_{\text{\rm{ON}}})^{x_a} (1-\pi^a_{\text{\rm{ON}}})^{1-x_a}\\
    & = \prod_{a\in\cup_{k\geq2}[n]_k} \Bigg(\prod_{i\in a} \frac{f(|a|)^{\frac{1}{|a|}}w_i}{\ell_n^{\frac{|a|-1}{|a|}}} \Bigg)^{x_a} \prod_{a\in\cup_{k\geq2}[n]_k} \frac{1}{1+\frac{f(|a|)\prod_{i\in a}w_i}{\ell_n^{|a|-1}}}.\nonumber
\end{align}
Note that $\frac{\prod_{i\in a}f(|a|)w_i}{\ell_n^{|a|-1}} = \laoff$ and hence we can abbreviate
\begin{align} \label{main_prod}
    \mathbf{P}(X=x) & = \Big(\prod_{a\in\cup_{k\geq2}[n]_k} (1+\laoff)\Big)^{-1} \prod_{a\in\cup_{k\geq2}[n]_k} \bigg(\prod_{i\in a} w_i \bigg)^{x_a} \Bigg(\frac{f(|a|)}{\ell_n^{{|a|-1}}} \Bigg)^{x_a}.
\end{align}
We observe that
\begin{align} \label{l_deg_prod}
    \prod_{a\in\cup_{k\geq2}[n]_k} \bigg(\prod_{i\in a} w_i \bigg)^{x_a} = \prod_{a\in\cup_{k\geq2}[n]_k} \prod_{i\in a} w_i^{x_a} = \prod_{i\in[n]} \prod_{a:a \ni i}  w_i^{x_a} = \prod_{i\in[n]} w_i^{\sum_{a\in[n]_k:a\ni i}x_a} = \prod_{i\in[n]} w_i^{d^{(l)}_i(x)}
\end{align}
is a function of the left-degrees. Similarly, extending the definition of the function of group-size distribution by fixing $f(0)=1$,
\begin{align} \label{r_deg_prod}
    \prod_{a\in\cup_{k\geq2}[n]_k} \Big(\frac{f(|a|)}{\ell_n^{|a|-1}}\Big)^{x_a} = \prod_{a\in\cup_{k\geq2}[n]_k} \frac{(f(|a|))^{x_a}}{\ell_n^{(|a|-1)x_a}} = \prod_{a\in\cup_{k\geq2}[n]_k} \frac{f\big(d^{(r)}_a(x)\big)}{\ell_n^{d^{(r)}_a(x)\big(1-\frac{1}{|a|}\big)}}
\end{align}
is a function of the right-degrees, as $d^{(r)}_a(x)=|a|\cdot x_a$. After substituting (\ref{l_deg_prod}) and (\ref{r_deg_prod}) into (\ref{main_prod}) the claim follows.
\end{proof}

Given Proposition \ref{fct_left_right} it is not difficult to show that the static bipartite graph conditioned on its degree sequence is uniform:

\ourgraphuniform*

\begin{proof}
With $x = (x_{a})_{a\in\cup_{k\geq2}[n]_k}$ satisfying $d^{(l)}_i(x) = d^{(l)}_i$ for all $i \in [n]$ and $d^{(r)}_a(x) = d^{(r)}_a$ for all groups $a$ we can write
\begin{align}
    \mathbf{P}(X = x&\mid  d^{(l)}_i(X) = d^{(l)}_i \hspace{0.1cm} \forall i \in [n], d^{(r)}_a(X) = d^{(r)}_a \hspace{0.1cm} \forall a)\\
    & = \frac{\mathbf{P}(X=x)}{\mathbf{P}(d^{(l)}_i(X) = d^{(l)}_i \hspace{0.1cm} \forall i \in [n], d^{(r)}_a(X) = d^{(r)}_a \hspace{0.1cm} \forall a \in [n]_k)} \nonumber\\
    &= \frac{\mathbf{P}(X=x)}{\sum_{\{y:d^{(l)}_i(y) = d^{(l)}_i \hspace{0.1cm} \forall i \in [n], d^{(r)}_a(y) = d^{(r)}_a \hspace{0.1cm} \forall a \in [n]_k\}}\mathbf{P}(X=y)}\nonumber\\
    & = \frac{G(a)^{-1} H\big((d^{(l)}(x),d^{(r)}(x))\big)}{\sum_{\{y:d^{(l)}_i(y) = d^{(l)}_i \hspace{0.1cm} \forall i \in [n], d^{(r)}_a(y) = d^{(r)}_a \hspace{0.1cm} \forall a \in [n]_k\}}G(a)^{-1} H\big((d^{(l)}(x),d^{(r)}(x)\big)}\nonumber\\
    & = \frac{H\big((d^{(l)},d^{(r)})\big)}{\sum_{\{y:d^{(l)}_i(y) = d^{(l)}_i \hspace{0.1cm} \forall i \in [n], d^{(r)}_a(y) = d^{(r)}_a \hspace{0.1cm} \forall a \in [n]_k\}}H\big((d^{(l)},d^{(r)})\big)}\nonumber\\
    & = \frac{1}{\#\{y:d^{(l)}_i(y) = d^{(l)}_i \hspace{0.1cm} \forall i \in [n], d^{(r)}_a(y) = d^{(r)}_a \hspace{0.1cm} \forall a \in [n]_k\}},\nonumber
\end{align}
which means that the distribution is uniform over all bipartite graphs with the prescribed left- and right-degree sequences.
\end{proof}

\subsection{Bipartite graph with communities conditioned on simplicity is uniform}

Note: $\bcm$ from \cite{Hofstad2018} does not have a community structure. Hence, we prove the equivalence between $\bgrg$ and $\bcm$ with complete communities. Before proving the main result, i.e., the fact that the bipartite configuration model is uniform given simplicity, we need an auxiliary proposition (which is analogous to a similar result for $\mathrm{CM}_n(\bm{d})$ - see \textup{\cite[Proposition 7.7]{Hofstad2016}}):

\begin{proposition} (The law of $\bcm$) \label{proposition_law}
Denote by $G = (x_{ia})_{i\in [n], a\in[n]_k}$ a bipartite multigraph on left-vertices $i\in [n]$ and right-vertices $a\in [n]_k$, such that $d^{(l)}_i =\sum_{a\in\cup_{k\geq2}[n]_k} x_{ia}$ and $d^{(r)}_a =\sum_{i\in[n]} x_{ia}$, where $x_{ia}$ is the number of edges between $i\in [n]$ and $a\in[n]_k$. Then,
\begin{align} \label{BCM_G}
    \mathbf{P}(\bcm = G) = \frac{1}{h_n!} \frac{\prod_{i \in [n]}d^{(l)}_i! \prod_{a\in\cup_{k\geq2}[n]_k}d^{(r)}_a!}{\prod_{i\in[n],a\in[n]_k} x_{ia}!},
\end{align}
with $h_n = \sum_{i \in [n]}d^{(l)}_i = \sum_{a\in\cup_{k\geq2}[n]_k} d^{(r)}_a$. By $d$ in $\bcm$ we mean a double degree sequence $(\bm{d}^{(l)},\bm{d}^{(r)})=\big((d^{(l)}_i)_{i\in[n]},(d^{(r)}_a)_{a\in\cup_{k\geq2}[n]_k}\big)$.
\end{proposition}
\begin{remark}
Note that \textup{\cite[(2.38)]{Hofstad2022}} yields the same formula as (\ref{BCM_G}). However, the authors of \cite{Hofstad2022} deliver this result in a form of a remark, giving justification rather than formal proof. We provide a formal proof.
\end{remark}
\begin{proof}
We start by computing the number of all possible matchings between the left and right sides. Imagine we want to assign a right-half-edge to every left-half-edge uniformly at random. For the first fixed left-half-edge we have $h_n$ choices of available right-half-edges. For the second left-half-edge, $h_n-1$ choices, and so on. It is not hard to see that the number of all such matchings is $h_n!$. Hence,
\begin{align}
    \mathbf{P}(\bcm = G) = \frac{1}{h_n!} \# N(G),
\end{align}
where $N(G)$ is the number of configurations that, after identifying the vertices, result in the multigraph $G$. Note that permuting half-edges incident to vertices will give rise to the same pairs of left- and right-vertices, hence the same multigraph $G$, but yet, when it comes to half-edges, it is a different configuration. The number of such permutations is $\prod_{i \in [n]}d^{(l)}_i! \cdot \prod_{a\in\cup_{k\geq2}[n]_k}d^{(r)}_a!$. However, some of these permutations yield the same half-edge pairings. If two half-edges of the left-vertex $i\in[n]$ are paired to two half-edges of the right-vertex $a\in[n]_k$ and we permute all of them, we will have the same half-edges being matched again. Thus, we divide by $\prod_{i\in[n],a\in[n]_k} x_{ia}!$ to compensate for the 'double-counting' caused by multiple connections.
\end{proof}

Using this result we can prove the main theorem of the section (which is an adaptation of a similar result for $\mathrm{CM}_n(\bm{d})$ - see \textup{\cite[Proposition 7.15]{Hofstad2016}}):

\BCMuniform*

\begin{proof}
Since, by (\ref{BCM_G}), $\mathbf{P}(\bcm = G)$ is the same for every simple graph $G$, also conditional probability\\
$\mathbf{P}(\bcm = G\mid \bcm \hspace{0.1cm} \text{is simple})$ is the same for every simple graph $G$. Hence, for any degree sequence $(\bm{d}^{(l)},\bm{d}^{(r)})=\big((d^{(l)}_i)_{i\in[n]},(d^{(r)}_a)_{a\in\cup_{k\geq2}[n]_k}\big)$ and conditionally on the event \{$\bcm$ is a simple graph\}, $\bcm$ is a uniform simple graph with degree sequence $(\bm{d}^{(l)},\bm{d}^{(r)})=\big((d^{(l)}_i)_{i\in[n]},(d^{(r)}_a)_{a\in\cup_{k\geq2}[n]_k}\big)$.
\end{proof}

\subsection{Relation between \texorpdfstring{$\bgrg$}{} and \texorpdfstring{$\bcm$}{}.}

Thanks to Theorems \ref{our_graph_uniform} and \ref{BCM_uniform} we can show that the static bipartite generalized random graph and the bipartite configuration model under certain conditions yield a certain graph $G$ with the same probability. However, as we mentioned in the proof overview section, if we assume a few regularity conditions on the degree sequences, we can deduce a stronger link determining when certain events happen with high probability for $\bcm$ and $\bgrg$. The necessary regularity conditions are as follows:

\regcond*

If we assume these conditions for $\bcm$, the following is a natural consequence of Theorem \ref{BCM_uniform} (analogously to a similar result for $\mathrm{CM}_n(\bm{d})$ - see \textup{\cite[Corollary 7.17]{Hofstad2016}}):

\begin{corollary} (Uniform graphs with given degree sequence and $\bcm$). \label{BCM_unif_event}
Assume that $(\bm{d}^{(l)},\bm{d}^{(r)})=$\\$\big((d^{(l)}_i)_{i\in[n]},(d^{(r)}_a)_{a\in\cup_{k\geq2}[n]_k}\big)$ satisfies Condition \ref{reg_cond}. Then, an event $\mathcal{E}_n$ occurs with high probability for a uniform simple bipartite random graph with degrees $\bm{d}=(\bm{d}^{(l)},\bm{d}^{(r)})=\big((d^{(l)}_i)_{i\in[n]},(d^{(r)}_a)_{a\in\cup_{k\geq2}[n]_k}\big)$ when it occurs with high probability for $\bcm$.
\end{corollary}

\begin{proof}
Let $\mathrm{UG}_n(\bm{d})$ denote a uniform simple bipartite random graph with degrees $\bm{d}=(\bm{d}^{(l)},\bm{d}^{(r)})$. Let $\mathcal{E}_n$ be a subset of multi-graphs that $\lim_{n \to \infty} \mathbf{P}\big(\bcm \in \mathcal{E}_n^c \big) = 0$. We need to prove that then also $\lim_{n \to \infty} \mathbf{P}\big(\mathrm{UG}_n(\bm{d})\in \mathcal{E}_n^c \big)= 0$. By Theorem \ref{BCM_uniform},
\begin{align}
    \mathbf{P}\big((\mathrm{UG}_n(\bm{d})\in \mathcal{E}_n^c \big) & = \mathbf{P}(\bcm \in \mathcal{E}_n^c\mid  \bcm \hspace{0.1cm} \text{is simple})\\
    & = \frac{\mathbf{P}(\bcm \in \mathcal{E}_n^c, \bcm \hspace{0.1cm} \text{is simple})}{\mathbf{P}(\bcm \hspace{0.1cm} \text{is simple})}\nonumber\\
    &\leq \frac{\mathbf{P}(\bcm \in \mathcal{E}_n^c)}{\mathbf{P}(\bcm \hspace{0.1cm} \text{is simple})}.\nonumber
\end{align}
By assumption we have that $\lim_{n \to \infty} \mathbf{P}(\bcm \in \mathcal{E}_n^c) = 0$. Moreover, by \textup{\cite[Theorem 1.10 (1.45)]{Angel2016}}, for which the conditions are satisfied by Condition \ref{reg_cond}, it follows
\begin{align}\label{BCM_simple_wpp}
    \liminf_{n \to \infty} \mathbf{P}(\bcm \hspace{0.1cm} \text{is simple}) > 0,
\end{align}
so that $\mathbf{P}\big((\mathrm{UG}_n(\bm{d})\in \mathcal{E}_n^c \big) \to 0$.
\end{proof}

Now we can prove the final result on the link between the two models (which is analogous to a similar result for $\mathrm{GRG}_n(\bm{w})$ and $\mathrm{CM}_n(\bm{d})$ - see \textup{\cite[Theorem 7.18]{Hofstad2016}}):

\twomodelsrelation*

\begin{proof}
Equality in (\ref{two_models_distr_equality}) follows from Theorems \ref{our_graph_uniform} and \ref{BCM_uniform} for every simple bipartite graph $G$ with degree sequence $\bm{D}(\bm{d}^{(l)},\bm{d}^{(r)})$. These results imply that $\bgrg $ as well as $\bcm$ are uniform simple random graphs. Further, by (\ref{two_models_distr_equality}) we have that
\begin{align}
    \mathbf{P}(\bgrg \in \mathcal{E}_n\mid \bm{D}=\bm{d}) = \mathbf{P}(\bcm \in \mathcal{E}_n\mid  \bcm \hspace{0.1cm} \text{simple}). \nonumber
\end{align}
We rewrite
\begin{align}
    \mathbf{P}(\bgrg \in \mathcal{E}^c_n) &= \mathbf{E}\big[\mathbf{P}(\bgrg \in \mathcal{E}^c_n\mid D)\big]\\
    &=\mathbf{E}\big[\mathbf{P}(\bcm \in \mathcal{E}^c_n\mid \bcm \hspace{0.1cm} \text{simple})\big]\nonumber\\
    &\leq \mathbf{E}\bigg[\frac{\mathbf{P}(\bcm \in \mathcal{E}^c_n)}{\mathbf{P}(\bcm \hspace{0.1cm} \text{simple})} \wedge 1\bigg],\nonumber
\end{align}
where $\mathbf{E}$ is the expectation w.r.t the degree sequence $D$. We assumed that $\mathbf{P}(\bcm \in \mathcal{E}^c_n) \stackrel{\mathbf{P}}{\longrightarrow}0$. Since $D$ satisfies Condition \ref{reg_cond} in probability, by \textup{\cite[Theorem 1.10 (1.45)]{Angel2016}} it follows
\begin{align}\label{BCM_simple_wpp_2}
    \liminf_{n \to \infty} \mathbf{P}(\bcm \hspace{0.1cm} \text{is simple}) > 0.
\end{align}
Hence, by the dominated convergence theorem, we conclude that $\mathbf{P}(\bgrg \in \mathcal{E}^c_n)\longrightarrow0$.
\end{proof}

\section{Regularity conditions static bipartite graph and consequences} \label{appx_static_graph}

Theorem \ref{two_models_relation} shows that if the degree sequences of $\bcm$ and $\bgrg$ satisfy Condition \ref{reg_cond}(i)-(ii), then it can be deduced that if some event $\mathcal{E}_n$ happens with high probability for $\bcm$, it also happens with high probability for $\bgrg$. The results derived for $\bcm$ in \cite{Hofstad2018} and \cite{Hofstad2022} hold precisely under Conditions \ref{reg_cond}(i)-(ii). Hence, if $\bgrg$ also satisfies these regularity conditions, we can transfer all the results on the local convergence and the giant component from \cite{Hofstad2018} and \cite{Hofstad2022}.

Here, we show that the degree sequence of $\bgrg$ indeed satisfies the required conditions. We will also argue why it is possible to drop Condition \ref{reg_cond}(ii). Before proceeding to the proofs, we state one consequence of Condition \ref{cond_weights} that we frequently make use of in the following sections:

\begin{corollary} \label{asympt_max_weight}
Condition \ref{cond_weights} $(a)-(b)$ implies that $\max_{i \in [n]} w_i = o(n)$ and Condition \ref{cond_weights} $(a)-(c)$ implies that $\max_{i \in [n]} w_i = o(\sqrt{n})$.
\end{corollary}

\subsection{Convergence of left degrees in \texorpdfstring{$\bgrg$}{}}
Throughout this section we assume Condition \ref{cond_weights}(c), however, we later argue it can be lifted. To show the convergence of the left degrees in $\bgrg$, we first need the following auxiliary result:

\begin{theorem} [Poisson approximation of the number of $k$-cliques containing a vertex] \label{Poi_coupl_1}
Let $\mathcal{C}_k(i)$ denote the number of groups of size $k$ containing vertex $i\in[n]$.There exists a coupling $(\hat{\mathcal{C}}_k(i),\hat{Z}_{i,k})$ of $\mathcal{C}_k(i)$ and a Poisson random variable $Z_{i,k}$ with parameter $kp_kw_i$, such that
\begin{align} \label{poisson_approx_final_ineq}
     \mathbf{P}(\mathcal{\hat{C}}_k(i) \neq \hat{Z}_{i,k}) \leq \frac{2(k(k-1)p_k)^2w_i^2}{\ell_n} \Bigg(\frac{k-1}{\ell_n}\Bigg)^{k-2} \Bigg(\frac{\mathbf{E}[W_n^2]}{\mathbf{E}[W_n]}\Bigg)^{k-1}.
\end{align}
\end{theorem}

\begin{proof}
We adapt the proof of a similar result for $\mathrm{GRG}_n(\bm{w})$. (See the proof of \textup{\cite[Theorem 6.7]{Hofstad2016}}.). Note that it holds that
\begin{align} \label{number_k_cliques}
    \mathcal{C}_k(i) = \sum_{a\in [n]_k: a\ni i} \mathds{1}_{\{\text{$a$ is \rm{ON}}\}}
\end{align}
Hence, $\mathcal{C}_k(i)$ is a sum of independent Bernoulli random variables and by \textup{\cite[Theorem 2.10]{Hofstad2018}} we know that there exists a Poisson random variable $\hat{Y}_{i,k}$ with parameter
\begin{align} 
    \lambda_{i,k} = \sum_{a\in[n]_k: a\ni i} \piaon  \label{lambda}
\end{align}
and a random variable $\mathcal{\hat{C}}_k(i)$ with the same distribution as $\mathcal{C}_k(i)$, such that
\begin{align}
    \mathbf{P}(\mathcal{\hat{C}}_k(i) \neq \hat{Y}_{i,k}) \leq \sum_{a\in[n]_k:a\ni i} \big(\piaon\big)^2.
\end{align}
We have
\begin{align} \label{pi_squared}
    \sum_{a\in[n]_k: a\ni i} \big(\piaon\big)^2 &= \sum_{a\in[n]_k: a\ni i} \Bigg( \frac{k!p_kw_i \prod_{j\in a,j\neq i}w_j}{\ell_n^{k-1}+k!p_kw_i \prod_{j\in a,j\neq i}w_j} \Bigg)^2\leq (k!p_k)^2 w_i^2 \sum_{a\in[n]_k: a\ni i} \Bigg( \frac{\prod_{j\in a,j\neq i}w_j}{\ell_n^{k-1}}\Bigg)^2\\
    &\leq (k!p_k)^2 w_i^2 \frac{\big(\sum_{j \in [n]} w_j^2\big)^{k-1}}{(k-1)!(\ell^2_n)^{k-1}}= \frac{(k-1)!(kp_k)^2 w_i^2}{\ell_n^{k-1}} \frac{\big(\mathbf{E}[W_n^2]\big)^{k-1}}{\big(\mathbf{E}[W_n]\big)^{k-1}}. \nonumber
\end{align}
Let $\varepsilon_{i,k} = kp_kw_i-\lambda_{i,k} \geq 0$. Take $\hat{V}_{i,k} \sim Poi(\varepsilon_{i,k})$ and write $\hat{Z}_{i,k} = \hat{Y}_{i,k} + \hat{V}_{i,k}$. By the Markov inequality,
\begin{align}
    \mathbf{P}(\hat{Y}_{i,k} \neq \hat{Z}_{i,k}) = \mathbf{P}(\hat{V}_{i,k} \neq \varnothing) = \mathbf{P}(\hat{V}_{i,k} \geq 1) \leq \mathbf{E}[\hat{V}_{i,k}] = \varepsilon_{i,k}.
\end{align}
We note that
\begin{align} \label{poisson_coupling_better_parameter}
    \varepsilon_{i,k} & = kp_kw_i - \lambda_{i,k} = \frac{k!p_kw_i}{(k-1)!} \cdot 1 - \sum_{j_1<...<j_{k-1}\in[n]} \pi^{\{i,j_1,...,j_{k-1}\}}_{\text{\rm{ON}}}\\
    & = \frac{k!p_kw_i}{(k-1)!} \sum_{j_1,...,j_{k-1}} \frac{w_{j_1}\cdots w_{j_{k-1}}}{\ell_n^{k-1}} - \sum_{j_1<...<j_{k-1}\in[n]} \frac{k!p_kw_i w_{j_1}\cdots w_{j_{k-1}}}{\ell_n^{k-1}+k!p_kw_i w_{j_1}\cdots w_{j_{k-1}}} \nonumber\\
    & = \sum_{j_1<...<j_{k-1}\in[n]} k!p_kw_i w_{j_1}\cdots w_{j_{k-1}} \bigg(\frac{1}{\ell_n^{k-1}} - \frac{1}{\ell_n^{k-1}+k!p_kw_i w_{j_1}\cdots w_{j_{k-1}}} \bigg)\nonumber\\
    & = \sum_{j_1<...<j_{k-1}\in[n]} \frac{(k!p_k)^2w^2_i w^2_{j_1}\cdots w^2_{j_{k-1}}}{\ell_n^{k-1}(\ell_n^{k-1}+k!p_kw_i w_{j_1}\cdots w_{j_{k-1}})} \leq \sum_{j_1<...<j_{k-1}\in[n]} \frac{(k!p_k)^2w^2_i w^2_{j_1}\cdots w^2_{j_{k-1}}}{\ell_n^{2(k-1)}} \nonumber\\
    & \leq \frac{(k-1)!(kp_k)^2w_i^2}{\ell_n^{k-1}} \frac{\big(\sum_{j \in [n]} w_j^2\big)^{k-1}}{\ell_n^{k-1}}. \nonumber
\end{align}
Using the fact that $k!\leq k^k$ this yields
\begin{align}
    \mathbf{P}(\mathcal{\hat{C}}_k(i) \neq \hat{Z}_{i,k}) & \leq \mathbf{P}(\mathcal{\hat{C}}_k(i) \neq \hat{Y}_{i,k}) + \mathbf{P}(\hat{Y}_{i,k} \neq \hat{Z}_{i,k})\\ \nonumber
    & \leq \frac{2(k-1)!(kp_k)^2w_i^2}{\ell_n^{k-1}} \frac{\big(\sum_{j \in [n]} w_j^2\big)^{k-1}}{\ell_n^{k-1}} \leq 2(k(k-1)p_k)^2\Bigg(\frac{k-1}{\ell_n}\Bigg)^{k-2} \frac{w_i^2}{\ell_n} \Bigg(\frac{\mathbf{E}[W_n^2]}{\mathbf{E}[W_n]}\Bigg)^{k-1},
\end{align}
which proves the desired result.
\end{proof}

Thanks to the coupling derived above we conclude that the degree of a uniformly chosen left-vertex converges:

\begin{theorem}[Left-degree in bipartite graph $\bgrg$] \label{conv_unif_left}
Let $D^{(l)}_n$ denote the degree of a uniformly chosen left-vertex in $\bgrg$. Then, as $n\to\infty$,
\begin{align} 
    D^{(l)}_n \stackrel{\text{d}}{\longrightarrow} D^{(l)},
\end{align}
where $D^{(l)}$ is a mixed-Poisson variable with mixing parameter $W\mu$, with $\mu = \sum_k kp_k.$
\end{theorem}

\begin{proof}
We adapt the proof of a similar result for $\mathrm{GRG}_n(\bm{w})$. (See the proof of \textup{\cite[Corollary 6.9]{Hofstad2016}}.). We have that $d^{(l)}_i = \sum_{k=2}^{\infty} \mathcal{C}_k(i)$ for any vertex $i\in[n]$. Fix a sequence $b_n\to\infty$ as $n\to\infty$. Note that,
\begin{align}
    \mathbf{E}\Bigg[\sum_{k>b_n} \mathcal{C}_k(i)\Bigg] \leq  \sum_{k>b_n} kp_k w_i \sum_{j_1,...,j_{k-1}}\frac{w_{j_1}\cdots  w_{j_k}}{\ell_n^{k-1}} = w_i \sum_{k>b_n} kp_k = o(1),
\end{align}
for every $i\in[n]$ and every $b_n\to\infty$ since we assumed $\sum_{k=2}^{\infty}kp_k<\infty$. We conclude that $d^{(l)}_i$ is close to $\sum_{k=2}^{b_n} \mathcal{C}_k(i)$. We know from Theorem \ref{Poi_coupl_1}, the assumption $\sum_{k=2}^{\infty}k^2p_k<\infty$ and the fact that $\max_{i\in[n]}w_i=o(\sqrt{n})$, that for all $k\in[2,b_n]$, $\mathcal{C}_k(i)$ is close in distribution to a Poisson variable with parameter $kp_kw_i$ if we choose $b_n=o(n)$. Hence, for a uniformly chosen vertex $o_n$ from $[n]$, the number of groups of size $k$ containing this vertex, for all $k\in[2,b_n]$ with $b_n$ at most $o(n)$, is close to a Poisson variable with parameter $kp_kw_{o_n}$, where $w_{o_n}$ is the weight of a uniformly chosen vertex. Such a variable follows a mixed-Poisson distribution with mixing distribution $w_{o_n}$, and $w_{o_n}$ is distributed like $W_n$. We know that a mixed-Poisson random variable converges to a limiting mixed-Poisson random variable when the mixing distribution converges. Since we have assumed convergence of $W_n$ to a limiting variable $W$ in Condition \ref{cond_weights}, it follows that for all $k\in[2,b_n]$ with $b_n$ at most $o(n)$, $\mathcal{C}_k(o_n)$ converges to a Poisson random variable with parameter $kp_kW$. Thus, truncated $D^{(l)}_n$ has the same distribution as the sum over $k \in[2,b_n)$ of independent Poisson variables with parameters $kp_kW$, which, as $n\to\infty$, converges to a Poisson variable with parameter $W\mu$. Since the difference between $D^{(l)}_n$ and truncated $D^{(l)}_n$ is small as $n\to\infty$, we conclude $D^{(l)}_n$ also converges to a Poisson variable with parameter $W\mu$.
\end{proof}
Below we use the second-moment method to show that also the expected degree of a uniformly chosen left-vertex converges.

\begin{theorem} \label{1st_mom_unif_bipartite}
Let $D^{(l)}_n$ denote the degree of a uniformly chosen left-vertex in $\bgrg$. Then, as $n\to\infty$,
\begin{align}
    \mathbf{E}[D^{(l)}_n\mid  G_n] \stackrel{\mathbf{P}}{\longrightarrow} \mu\mathbf{E}[W].
\end{align}
\end{theorem}

\begin{proof}
Note that $\mathbf{E}[D^{(l)}_n\mid  G_n] = \frac{1}{n}\sum_{i\in[n]}d^{(l)}_i$ is the same as $\frac{1}{n}\sum_{a\in\cup_{k\geq2}[n]_k}d^{(r)}_a=\frac{1}{n}\sum_{a\in\cup_{k\geq2}[n]_k}|a|\cdot\mathds{1}_{\{\text{$a$ is \rm{ON}}\}}$, which is a sum of independent variables. To avoid more assumptions on the moments of the group-size distribution, we fix a sequence $b_n\to\infty$ as $n\to\infty$ and use truncation with respect to the group size, i.e,
\begin{align}
    \sum_{k=2}^{\infty}\sum_{a\in[n]_k} |a| \mathds{1}_{\{\text{$a$ is \rm{ON}}\}} = \sum_{k=2}^{b_n}\sum_{a\in[n]_k} |a|\mathds{1}_{\{\text{$a$ is \rm{ON}}\}} + \sum_{k=b_n+1}^{\infty}\sum_{a\in[n]_k} |a| \mathds{1}_{\{\text{$a$ is \rm{ON}}\}}.
\end{align}
For further notational convenience, denote $w_a = \prod_{i\in a}w_i$. We compute
\begin{align}
    \mathbf{E}&\big[\frac{1}{n} \sum_{k=b_n+1}^{\infty}\sum_{a\in[n]_k} |a|\mathds{1}_{\{\text{$a$ is \rm{ON}}\}}\big] = \frac{1}{n}\sum_{k>b_n} k \sum_{a\in[n]_k} \frac{k!p_kw_a}{\ell_n^{k-1}+k!p_kw_a}\leq \frac{\ell_n}{n}\sum_{k>b_n}kp_k = o(1),
\end{align}
for every $b_n\to\infty$, since we have assumed $\sum_k kp_k<\infty$. Hence, the contribution from groups larger than $b_n$ will vanish in probability, i.e.,
\begin{align}
    \sum_{k=2}^{\infty}\sum_{a\in[n]_k} |a| \mathds{1}_{\{\text{$a$ is \rm{ON}}\}} = \sum_{k=2}^{b_n}\sum_{a\in[n]_k} |a| \mathds{1}_{\{\text{$a$ is \rm{ON}}\}} + o_{\mathbf{P}}(1).
\end{align}
Note that
\begin{align} \label{pi_a_on_rewrite}
    \frac{k!p_kw_a}{\ell_n^{k-1}+k!p_kw_a} = \frac{k!p_kw_a}{\ell_n^{k-1}} - \frac{(k!p_k)^2(w_a)^2}{\ell_n^{k-1}(\ell_n^{k-1}+k!p_kw_a)}.
\end{align}
Hence,
\begin{align}
    \mathbf{E}[\frac{1}{n}\sum_{i\in[n]}d^{(l)}_i] &\leq \frac{1}{n}\sum_{k=2}^{b_n}k\sum_{a\in[n]_k} \frac{k!p_k w_a}{\ell_n^{k-1}}\leq \frac{\ell_n}{n}\sum_{k=2}^{\infty}kp_k = \mathbf{E}[W_n]\mu \stackrel{n\to\infty}{\longrightarrow}\mathbf{E}[W]\mu.
\end{align}
Further, using the fact that $k!\leq k^k$,
\begin{align} \label{pi_a_on_lower_bound}
    \sum_{a\in[n]_k}&\frac{(k!p_k)^2(w_a)^2}{\ell_n^{k-1}(\ell_n^{k-1}+k!p_kw_a)} \leq \sum_{a\in[n]_k}\frac{(k!p_k)^2(w_a)^2}{(\ell_n^{k-1})^2}\\
    &\leq k!(p_k)^2\frac{\big(\sum_{i\in[n]}w_i^2\big)^k}{\ell_n^{k-2}\ell_n^{k}} = (kp_k)^2 \Bigg(\frac{k}{\ell_n}\Bigg)^{k-2} \Bigg(\frac{\mathbf{E}[W_n^2]}{\mathbf{E}[W_n]}\Bigg)^k, \nonumber
\end{align}
and thus
\begin{align}
    \mathbf{E}[\frac{1}{n}\sum_{i\in[n]}d^{(l)}_i] &\geq \frac{\ell_n}{n}\sum_{k=2}^{b_n}kp_k - \frac{1}{n} \sum_{k=2}^{b_n} k (kp_k)^2 \Bigg(\frac{k}{\ell_n}\Bigg)^{k-2} \Bigg(\frac{\mathbf{E}[W_n^2]}{\mathbf{E}[W_n]}\Bigg)^k\\
    &\geq \frac{\ell_n}{n}\sum_{k=2}^{b_n}kp_k - \frac{b_n}{n} \sum_{k=2}^{b_n} (kp_k)^2 \Bigg(\frac{b_n}{\ell_n}\Bigg)^{k-2} \Bigg(\frac{\mathbf{E}[W_n^2]}{\mathbf{E}[W_n]}\Bigg)^k = \frac{\ell_n}{n}\sum_{k=2}^{b_n}kp_k - o(1) \stackrel{n\to\infty}{\longrightarrow}\mathbf{E}[W]\mu, \nonumber
\end{align}
if we choose $b_n=o(n)$ and since we have assumed $\mu_{(2)} = \sum_kk^2p_k < \infty$. Using the independence and the fact that the variance of an indicator random variable is smaller or equal to its expectation we compute
\begin{align} \label{var_average_left_degree}
    \Var[\frac{1}{n}\sum_{i\in[n]}d^{(l)}_i] \leq \frac{\ell_n}{n^2}\sum_{k=2}^{\infty}k^2p_k = \frac{\mathbf{E}[W_n]}{n} \sum_{k=2}^{\infty}k^2p_k = o(1),
\end{align}
since we have assumed $\sum_{k=2}^{\infty}k^2p_k<\infty$. Taking $n$ large enough so that $\big|\mathbf{E}[\frac{1}{n}\sum_{i\in[n]}d^{(l)}_i] - \mu\mathbf{E}[W]\big|\leq \frac{\varepsilon}{2}$, by Chebyshev's inequality
\begin{align}
    \mathbf{P}\bigg(\bigg|\frac{1}{n}\sum_{i\in[n]}d^{(l)}_i-\mu\mathbf{E}[W]\bigg|>\varepsilon\big) & \leq \mathbf{P}\bigg(\bigg|\frac{1}{n}\sum_{i\in[n]}d^{(l)}_i-\mathbf{E}\bigg[\frac{1}{n}\sum_{i\in[n]}d^{(l)}_i\bigg]\bigg|>\frac{\varepsilon}{2} \bigg)\\
    &\leq \frac{4}{\varepsilon^2}\Var\bigg[\frac{1}{n}\sum_{i\in[n]}d^{(l)}_i\bigg] = o(1). \nonumber
\end{align}
\end{proof}
\begin{remark}
The above shows why Assumptions (\ref{assump_1_mom}) and (\ref{assump_2_mom}) are necessary.
\end{remark}

\subsection{Convergence of right degrees in \texorpdfstring{$\bgrg$}{}} \label{conv_group}
We want to show that the degree of a uniformly chosen right-vertex converges. We first need two auxiliary results:
\begin{theorem} [Convergence of the number of groups of fixed size] \label{conv_A_k}
Denote $A_k = \#\{a: |a|=k\}$. For $k\geq2$,
\begin{align}
    \frac{A_k}{n} \stackrel{\mathbf{P}}{\longrightarrow} p_k\mathbf{E}[W].
\end{align}
\end{theorem}
\begin{proof}
To prove the desired statement we use the second-moment method. We have
\begin{align}
    A_k  = \sum_{a \in[n]_k} \mathds{1}_{ \{\text{group $a$ is \rm{ON}}\}},
\end{align}
and thus,
\begin{align}
    \mathbf{E}[A_k] = \sum_{a \in[n]_k} \piaon = \sum_{j_1<...<j_k \in[n]} \frac{k!p_kw_{j_1}\cdots w_{j_{k}}}{\ell_n^{k-1}+k!p_kw_{j_1}\cdots w_{j_{k}}}.
\end{align}
Using (\ref{pi_a_on_rewrite}) we arrive at
\begin{align} \label{conv_group_equ}
    \mathbf{E}\bigg[\frac{A_k}{n}\bigg] \leq \frac{1}{n \cdot k!} \sum_{j_1,...,j_k \in[n]} \frac{k!p_kw_{j_1}\cdots w_{j_{k}}}{\ell_n^{k-1}} = \frac{p_k\ell_n}{n} \sum_{j_1,...,j_k \in[n]} \frac{w_{j_1}\cdots w_{j_{k}}}{\ell_n^{k}}=p_k\mathbf{E}[W_n] \longrightarrow p_k\mathbf{E}[W], \nonumber
\end{align}
as $n\to\infty$. Since $k$ is fixed, by (\ref{pi_a_on_rewrite}) and (\ref{pi_a_on_lower_bound}),
\begin{align}
    \mathbf{E}\bigg[\frac{A_k}{n}\bigg] \geq p_k\mathbf{E}[W_n] -\frac{(k p_k)^2}{n} \Bigg(\frac{k}{\ell_n}\Bigg)^{k-2} \Bigg(\frac{\mathbf{E}[W_n^2]}{\mathbf{E}[W_n]}\Bigg)^k = p_k\mathbf{E}[W_n] - o(1),
\end{align}
since we have assumed $\mu_{(2)}=\sum_k k^2p_k<\infty$. Therefore,
\begin{align}
    p_k\mathbf{E}[W_n] - o(1) \leq \mathbf{E}\bigg[\frac{A_k}{n}\bigg] \leq p_k\mathbf{E}[W_n].
\end{align}
Moreover, since $A_k$ is a sum of indicator random variables it holds that $\Var[A_k]\leq\mathbf{E}[A_k]$, which yields, for all $k\geq2$,
\begin{align}
    \Var \bigg(\frac{A_k}{n}\bigg) = \frac{1}{n^2} \Var[A_k] \leq  \frac{p_k\mathbf{E}[W_n]}{n} = o(1).
\end{align}
Take $n$ big enough so that $\big|\mathbf{E}\big[\frac{A_k}{n}\big] -p_k\mathbf{E}[W]\big|\leq \frac{\varepsilon}{2}$. Then
\begin{align}
    \mathbf{P}\bigg(\bigg|\frac{A_k}{n}-p_k\mathbf{E}[W]\bigg|>\varepsilon\bigg) \leq \mathbf{P}\bigg(\bigg|\frac{A_k}{n}-\mathbf{E}\big[\frac{A_k}{n}\big]\bigg|>\frac{\varepsilon}{2} \bigg) \leq \frac{4}{\varepsilon^2}\Var\bigg[\frac{A_k}{n}\bigg] = o(1). \nonumber
\end{align}
\end{proof}

\begin{theorem} [Convergence of the groups to vertices ratio] \label{conv_all_gr} Recall $M_n = \#\{a\in[n]_k: \text{$a$ is \rm{ON}}\}=\sum_{k=2}^{\infty}A_k$. As $n\to\infty$,
\begin{align}
    \frac{M_n}{n} \stackrel{\mathbf{P}}{\longrightarrow} \mathbf{E}[W].
\end{align}
\end{theorem}
\begin{proof}
We have
\begin{align}
    M_n = \sum_{k=2}^{\infty}\sum_{a\in[n]_k} \mathds{1}_{\{\text{$a$ is \rm{ON}}\}} = \sum_{k=2}^{\infty} A_k = \sum_{k=2}^{b_n} A_k + \sum_{k>b_n} A_k,
\end{align}
where $b_n$ is a sequence diverging to infinity as $n\to\infty$. Note that, using (\ref{pi_a_on_rewrite}),
\begin{align}
    \mathbf{E}\bigg[\frac{\sum_{k>b_n} A_k}{n}\bigg] \leq \mathbf{E}[W_n]\sum_{k>b_n} p_k = o(1),
\end{align}
for each $b_n\to\infty$. Thus,
\begin{align}
    M_n = \sum_{k=2}^{b_n} A_k + o_{\mathbf{P}}(1).
\end{align}
Again by (\ref{pi_a_on_rewrite}),
\begin{align}
    \mathbf{E}\bigg[\frac{M_n}{n}\bigg] = \sum_{k=2}^{b_n} \mathbf{E}\bigg[\frac{A_k}{n}\bigg] \leq \frac{\ell_n}{n} \sum_{k=2}^{\infty} p_k = \mathbf{E}[W_n] \longrightarrow \mathbf{E}[W],
\end{align}
and on the other hand, by (\ref{pi_a_on_rewrite}) and (\ref{pi_a_on_lower_bound}),
\begin{align}
    \mathbf{E}\bigg[\frac{M_n}{n}\bigg] &\geq \mathbf{E}[W_n]\sum_{k=2}^{b_n} p_k - \frac{1}{n} \sum_{k=2}^{b_n}(kp_k)^2 \Bigg(\frac{k}{\ell_n}\Bigg)^{k-2} \Bigg(\frac{\mathbf{E}[W_n^2]}{\mathbf{E}[W_n]}\Bigg)^k = \mathbf{E}[W_n] \sum_{k=2}^{b_n} p_k - o(1) \stackrel{n\to\infty}{\longrightarrow} \mathbf{E}[W],
\end{align}
if we choose $b_n=o(n)$. Further, since each $A_k$ is a sum of independent indicators,
\begin{align}
    \frac{1}{n^2}\Var(M_n) = \frac{1}{n^2} \sum_{k=2}^{b_n} \Var(A_k) \leq \frac{1}{n^2} \sum_{k=2}^{\infty} \mathbf{E}[A_k] \leq \frac{\ell_n}{n^2} = o(1).
\end{align}
Taking again $n$ big enough so that $\big|\mathbf{E}\big[\frac{M_n}{n}\big] - \mathbf{E}[W]\big|\leq \frac{\varepsilon}{2}$. Then
\begin{align}
    \mathbf{P}\bigg(\bigg|\frac{M_n}{n}-\mathbf{E}[W]\bigg|>\varepsilon\bigg) \leq \mathbf{P}\bigg(\bigg|\frac{M_n}{n}-\mathbf{E}\big[\frac{M_n}{n}\big]\bigg|>\frac{\varepsilon}{2} \bigg) \leq \frac{4}{\varepsilon^2}\Var\bigg(\frac{M_n}{n}\bigg) = o(1). \nonumber
\end{align}
\end{proof}
With the above, we can conclude convergence of the degree of a uniformly chosen right-vertex:

\begin{theorem} [Convergence of the degree of a uniformly chosen group] Recall that we denote the degree of a uniformly chosen group $a \in [n]_{k\geq2}$ by $D_n^{(r)}$. As $n\to\infty$,
\begin{align}\label{conv_unif_a}
    \mathbf{P}(D^{(r)}_n = k\mid  G_n) \stackrel{\mathbf{P}}{\longrightarrow} p_k.
\end{align}
\end{theorem}
\begin{proof}
Note that $\mathbf{P}(D^{(r)}_n = k\mid  G_n) = A_k\slash M_n$. From Theorems \ref{conv_A_k} and \ref{conv_all_gr} we know that $n^{-1}A_k$ and $n^{-1}M_n$ converge in probability, which implies convergence of the joint vector $(n^{-1}A_k,n^{-1}M_n$. Hence, the convergence of the ratio is guaranteed by the continuous mapping theorem:
\begin{align}
    \frac{A_k}{M_n} = \frac{\frac{A_k}{n}}{\frac{M_n}{n}} \stackrel{\mathbf{P}}{\longrightarrow} \frac{p_k\mathbf{E}[W]}{\mathbf{E}[W]} = p_k.
\end{align}
\end{proof}

It also follows easily that the expected degree of a uniformly chosen right-vertex converges.

\begin{corollary} [Convergence of the first moment of the degree of a uniformly chosen group]
It follows from the previous that
\begin{align}\label{conv_1st_mom_a}
    \mathbf{E}[D^{(r)}_n\mid  G_n] \stackrel{\mathbf{P}}{\longrightarrow} \mu.
\end{align}
\end{corollary}

\begin{proof}
Note that 
\begin{align}
    \mathbf{E}[D^{(r)}_n\mid  G_n] = \frac{\sum_{k=2}^{\infty}\sum_{a\in[n]_k}d^{(r)}_a}{M_n} = \frac{\sum_{k=2}^{\infty} kA_k}{M_n}
\end{align}
It is easy to show by the second-moment method and suitable truncation that $\sum_{k=2}^{\infty} kA_k\slash n \stackrel{\mathbf{P}}{\longrightarrow} \mu \mathbf{E}[W]$ and we already showed that $\frac{M_n}{n} \stackrel{\mathbf{P}}{\longrightarrow} \mathbf{E}[W]$. The claim follows again thanks to the joint convergence and continuous mapping theorem.
\end{proof}

\subsection{Convergence of the degree distribution in \texorpdfstring{$\drig$}{}} \label{appx_static_intersection}
\begin{remark}
    Note that the construction of $\drig$ allows for multiple edges, as two vertices $i,j\in[n]$ might meet in more than one group. However, it is easy to show that it will not happen with high probability in a neighborhood of a uniformly chosen vertex and hence is negligible as long as local convergence is concerned.
\end{remark}
\subsubsection{Expected average degree}

\firstmomunifv*

\begin{proof} [Proof of Theorem \ref{1st_mom_unif_v}]
Write $\mathbf{E}[D_n\mid  G_n] = \frac{1}{n}\sum_{i\in[n]}d_i$, where $d_i$ is the degree of vertex $i\in[n]$. Since we want to use the second-moment method and $d_i,d_j$ of some $i,j\in[n]$ are not independent, it is more convenient to express $\mathbf{E}[D_n\mid  G_n]$ in terms of groups, which are independent. Note that $\sum_{i\in[n]}d_i$ is nothing else than twice the number of all edges in $\drig$. As $\drig$ is constructed from $\bgrg$, it is a collection of $k$-cliques, whose ON or OFF status is determined by the ON and OFF processes of groups $a\in\cup_{n\geq2}[n]_k$ present in $\bgrg$. Hence,
\begin{align}
    \sum_{i\in[n]}d_i = \sum_{k=2}^{\infty}\sum_{a\in[n]_k} 2 \cdot \frac{|a|(|a|-1)}{2} \mathds{1}_{\{\text{$a$ is \rm{ON}}\}} = \sum_{k=2}^{\infty}\sum_{a\in[n]_k} |a|(|a|-1) \mathds{1}_{\{\text{$a$ is \rm{ON}}\}},
\end{align}
since the number of edges in a $k$-clique equals $k(k-1)\slash2$. To avoid more assumptions on the moments of the group-size distribution, we once more fix a sequence $b_n\to\infty$ as $n\to\infty$ and use truncation with respect to the group size, i.e,
\begin{align}
    \sum_{k=2}^{\infty}\sum_{a\in[n]_k} |a|(|a|-1) \mathds{1}_{\{\text{$a$ is \rm{ON}}\}} = \sum_{k=2}^{b_n}\sum_{a\in[n]_k} |a|(|a|-1) \mathds{1}_{\{\text{$a$ is \rm{ON}}\}} + \sum_{k>b_n}\sum_{a\in[n]_k} |a|(|a|-1) \mathds{1}_{\{\text{$a$ is \rm{ON}}\}}.
\end{align}
Once again denote $w_a = \prod_{i\in a}w_i$. Using upper bounds derived in previous proofs we compute
\begin{align}
    \mathbf{E}\big[\frac{1}{n} \sum_{k>b_n}\sum_{a\in[n]_k} |a|(|a|-1) \mathds{1}_{\{\text{$a$ is \rm{ON}}\}}\big] &=  \frac{1}{n}\sum_{k>b_n} k(k-1) \sum_{a\in[n]_k} \frac{k!p_kw_a}{\ell_n^{k-1}+k!p_kw_a}\\
    &\leq \frac{\ell_n}{n}\sum_{k>b_n}k(k-1)p_k = o(1), \nonumber
\end{align}
for every $b_n\to\infty$, since we have assumed $\sum_k k^2p_k<\infty$. Hence, the contribution from groups larger than $b_n$ will vanish in probability, i.e.,
\begin{align}
    \sum_{k=2}^{\infty}\sum_{a\in[n]_k} |a|(|a|-1) \mathds{1}_{\{\text{$a$ is \rm{ON}}\}} = \sum_{k=2}^{b_n}\sum_{a\in[n]_k} |a|(|a|-1) \mathds{1}_{\{\text{$a$ is \rm{ON}}\}}+o_{\mathbf{P}}(1).
\end{align}
Thus, again applying previously derived upper bounds,
\begin{align}
    \mathbf{E}[\frac{1}{n}\sum_{i\in[n]}d_i] &= \frac{1}{n} \sum_{k=2}^{b_n}\sum_{a\in[n]_k} |a|(|a|-1) \mathbf{E}\big[\mathds{1}_{\{\text{$a$ is \rm{ON}}\}}\big] +o(1) = \frac{1}{n} \sum_{k=2}^{b_n} k(k-1) \sum_{a\in[n]_k} \piaon+o(1)\\
    &\leq \frac{\ell_n}{n} \sum_{k=2}^{b_n} k(k-1)p_k+o(1)\leq \mathbf{E}[W_n](\mu_{(2)}-\mu) +o(1) \stackrel{n\to\infty}{\longrightarrow} \mathbf{E}W(\mu_{(2)}-\mu). \nonumber
\end{align}
On the other hand, using (\ref{pi_a_on_rewrite}),
\begin{align}
    \mathbf{E}[\frac{1}{n}\sum_{i\in[n]}d_i] &\geq \mathbf{E}[W_n]\sum_{k=2}^{b_n} k(k-1)p_k - \frac{1}{n} \sum_{k=2}^{b_n} k(k-1) (kp_k)^2 \Bigg(\frac{k}{\ell_n}\Bigg)^{k-2} \Bigg(\frac{\mathbf{E}[W_n^2]}{\mathbf{E}[W_n]}\Bigg)^k\\
    &\geq \mathbf{E}[W_n]\sum_{k=2}^{b_n} k(k-1)p_k - \frac{b^2_n}{n} \sum_{k=2}^{b_n} (kp_k)^2 \Bigg(\frac{b_n}{\ell_n}\Bigg)^{k-2} \Bigg(\frac{\mathbf{E}[W_n^2]}{\mathbf{E}[W_n]}\Bigg)^k \nonumber\\
    &=\mathbf{E}[W_n]\sum_{k=2}^{b_n} k(k-1)p_k - o(1) \stackrel{n\to\infty}{\longrightarrow} \mathbf{E}W(\mu_{(2)}-\mu),\nonumber
\end{align}
if we choose $b_n=o(n)$. We now compute the variance:
\begin{align} \label{var_average_degree}
    \Var\bigg(\frac{1}{n}\sum_{i\in[n]}d_i\bigg) &= \frac{1}{n^2} \sum_{k=2}^{b_n}\sum_{a\in[n]_k} |a|^2(|a|-1)^2 \Var\big[\mathds{1}_{\{\text{$a$ is \rm{ON}}\}}\big] \leq \frac{1}{n^2} \sum_{k=2}^{b_n}\sum_{a\in[n]_k} |a|^2(|a|-1)^2 \mathbf{E}\big[\mathds{1}_{\{\text{$a$ is \rm{ON}}\}}\big]\nonumber\\
    &\leq \frac{\ell_n}{n^2}\sum_{k=2}^{b_n} k^2(k-1)^2p_k \leq \frac{\ell_n b_n^2}{n^2}\sum_{k=2}^{b_n} k(k-1)p_k = o(1),
\end{align}
if we choose $b_n=o(\sqrt{n})$. Thus, taking $n$ big enough so that $\big|\mathbf{E}[\frac{1}{n}\sum_{i\in[n]}d_i] - (\mu_{(2)}-\mu)\mathbf{E}[W]\big|\leq \frac{\varepsilon}{2}$, we obtain 
\begin{align*}
    \mathbf{P}\bigg(\bigg|\frac{1}{n}\sum_{i\in[n]}d_i-(\mu_{(2)}-\mu)\mathbf{E}[W]\bigg|>\varepsilon\bigg) \leq \mathbf{P}\bigg(\bigg|\frac{1}{n}\sum_{i\in[n]}d_i-\mathbf{E}\big[\frac{1}{n}\sum_{i\in[n]}d_i\big]\bigg|>\frac{\varepsilon}{2} \bigg) \leq \frac{4}{\varepsilon^2}\Var\bigg(\frac{1}{n}\sum_{i\in[n]}d_i\bigg) = o(1).
\end{align*}
\end{proof}

\subsubsection{Degree sequence}

\degseqstat*

\begin{proof}
We adapt the proof of a similar result for $\mathrm{GRG}_n(\bm{w})$. (See the proof of \textup{\cite[Theorem 6.10]{Hofstad2016}}.) Since $(q_k)_{k\geq0}$ is a probability mass function,
\begin{align}
    \sum_{k\geq0}|Q^{(n)}_k-q_k| =2d_{\textsuperscript{\rm{TV}}}(Q^{(n)},q) \stackrel{\mathbf{P}}{\longrightarrow} 0,
\end{align}
if and only if $\max_{k\geq0}|Q^{(n)}_k-q_k| \stackrel{\mathbf{P}}{\longrightarrow} 0.$ Therefore we have to show that $\mathbf{P}(\max_{k\geq0}|Q^{(n)}_k-q_k| \geq \varepsilon)$ vanishes for every $\varepsilon>0$. Note that,
\begin{align}
    \mathbf{P}(\max_{k\geq0}|Q^{(n)}_k-q_k| \geq \varepsilon) \leq \sum_{k\geq0}\mathbf{P}(|Q^{(n)}_k-q_k| \geq \varepsilon).
\end{align}
We also have that $\mathbf{E}[Q_k^{(n)}]=\mathbf{P}(D_n=k)$ and by the previous theorem we know that $\lim_{n\to\infty} \mathbf{P}(D_n=k) = q_k$. Hence, for $n$ sufficiently large,
\begin{align}
    \max_k |\mathbf{E}[Q^{(n)}_k]-q_k| \leq \frac{\varepsilon}{2}.
\end{align}
Thus, for $n$ sufficiently large,
\begin{align}
    \mathbf{P}(\max_{k\geq0}|Q^{(n)}_k-q_k| \geq \varepsilon) \leq \sum_{k\geq0}\mathbf{P}(|Q^{(n)}_k-\mathbf{E}[Q^{(n)}_k]| \geq \frac{\varepsilon}{2}) \leq \frac{4}{\varepsilon^2}\sum_{k\geq0} \mathrm{Var}(Q^{(n)}_k),
\end{align}
where the last follows from the Chebychev inequality. We have
\begin{align}
    \mathrm{Var}(Q_k^{(n)}) \leq & \frac{1}{n^2} \sum_{i\in[n]}[\mathbf{P}(d_i=k)-\mathbf{P}(d_i=k)^2]\\
    &+ \frac{1}{n^2} \sum_{i,j\in[n],i\neq j}[\mathbf{P}(d_i=d_j=k)-\mathbf{P}(d_i=k)\mathbf{P}(d_j=k)].\nonumber
\end{align}
We have that 
\begin{align}
    \sum_{k\geq0} \frac{1}{n^2} \sum_{i\in[n]}[\mathbf{P}(d_i=k)-\mathbf{P}(d_i=k)^2] \leq \sum_{k\geq0} \frac{1}{n^2} \sum_{i\in[n]}\mathbf{P}(d_i=k) = \frac{1}{n} = o(1).
\end{align}
We want to show that the second term vanishes too. Note that $d_i=\sum_{a:a\ni i}(|a|-1)\mathds{1}_{\{\text{$a$ is \rm{ON}}\}}$. Hence, the correlation between $d_i$ and $d_j$ is due to groups containing both $i$ and $j$. We write
\begin{align}
    d_{i\setminus j} = \sum_{k=2}^{\infty}\sum_{a\in[n]_k:a \ni i, a \niton j} (|a|-1)\mathds{1}_{\{\text{$a$ is \rm{ON}}\}}, 
\end{align}
and we define $d_{j\setminus i}$ analogously. We also define
\begin{align}
    d_{i,j} = \sum_{k=2}^{\infty}\sum_{a\in[n]_k:a \ni i,j} (|a|-1)\mathds{1}_{\{\text{$a$ is \rm{ON}}\}}.
\end{align}
Then $(d_i,d_j)$ has the same law as $(d_{i\setminus j} + d_{i,j},d_{j\setminus i} + d_{i,j})$. Now let us introduce random variable $\hat{\mathds{1}}_{\{\text{$a$ is \rm{ON}}\}}$ such that $\hat{\mathds{1}}_{\{\text{$a$ is \rm{ON}}\}}\stackrel{\text{d}}{=}\mathds{1}_{\{\text{$a$ is \rm{ON}}\}}$ and $\hat{\mathds{1}}_{\{\text{$a$ is \rm{ON}}\}}$ independent of $\big(\mathds{1}_{\{\text{$a$ is \rm{ON}}\}}\big)_{a\in\cup_{k\geq2}[n]_k}$. Thus, 
\begin{align}
    \hat{d}_{i,j} = \sum_{k=2}^{\infty}\sum_{a\in[n]_k:a \ni i,j} (|a|-1)\hat{\mathds{1}}_{\{\text{$a$ is \rm{ON}}\}} \stackrel{\text{d}}{=} d_{i,j},
\end{align}
and then $(d_{i\setminus j} + d_{i,j},d_{j\setminus i} + \hat{d}_{i,j})$ are \emph{independent} random variables with the same marginals as $d_i,d_j$. Hence
\begin{align}
    &\mathbf{P}(d_i=d_j=k) = \mathbf{P}\big((d_{i\setminus j} + d_{i,j},d_{j\setminus i} + d_{i,j})=(k,k)\big),\\
    &\mathbf{P}(d_i=k)\mathbf{P}(d_j=k) = \mathbf{P}\big((d_{i\setminus j} + d_{i,j},d_{j\setminus i} + \hat{d}_{i,j})=(k,k)\big).\nonumber
\end{align}
Therefore,
\begin{align}
    \mathbf{P}&(d_i=d_j=k) - \mathbf{P}(d_i=k)\mathbf{P}(d_j=k)\\
    &= \mathbf{P}\big((d_{i\setminus j} + d_{i,j},d_{j\setminus i} + d_{i,j})=(k,k)\big) - \mathbf{P}\big((d_{i\setminus j} + d_{i,j},d_{j\setminus i} + \hat{d}_{i,j})=(k,k)\big)\nonumber\\
    &\leq \mathbf{P}\big( \big[(d_{i\setminus j} + d_{i,j},d_{j\setminus i} + d_{i,j})=(k,k)\big]\setminus \big[(d_{i\setminus j} + d_{i,j},d_{j\setminus i} + \hat{d}_{i,j})=(k,k) \big] \big)\nonumber\\
    &=\mathbf{P}\big((d_{i\setminus j} + d_{i,j},d_{j\setminus i} + d_{i,j})=(k,k),(d_{i\setminus j} + d_{i,j},d_{j\setminus i} + \hat{d}_{i,j})\neq(k,k)\big). \nonumber
\end{align}
When the above happens, it must be that $d_{i,j} \neq \hat{d}_{i,j}$, so there exists such $a\ni i,j$ that $\mathds{1}_{\{\text{$a$ is \rm{ON}}\}}\neq \hat{\mathds{1}}_{\{\text{$a$ is \rm{ON}}\}}$. If then, for some $a$, $\hat{\mathds{1}}_{\{\text{$a$ is \rm{ON}}\}}=1$, then $\mathds{1}_{\{\text{$a$ is \rm{ON}}\}}=0$ and $d_{i\setminus j} + d_{i,j}=k$; If $\mathds{1}_{\{\text{$a$ is \rm{ON}}\}}=1$, then $\hat{\mathds{1}}_{\{\text{$a$ is \rm{ON}}\}}=0$ and $d_{j\setminus i} + \hat{d}_{i,j}= k-(|a|-1).$ Hence,
\begin{align}
    \mathbf{P}&(d_i=d_j=k) - \mathbf{P}(d_i=k)\mathbf{P}(d_j=k) \leq \sum_{a\ni i,j}\mathbf{P}(\text{$a$ is \rm{ON}})[\mathbf{P}(d_i=k) + \mathbf{P}(d_j = k-|a|+1)].
\end{align}
This yields
\begin{align}
    \sum_{k\geq0}\mathrm{Var}(Q_k^{(n)}) &\leq o(1) + \sum_{k\geq0} \frac{1}{n^2} \sum_{i,j\in[n],i\neq j} \sum_{a\ni i,j}\mathbf{P}(\text{$a$ is \rm{ON}})[\mathbf{P}(d_i=k) + \mathbf{P}(d_j=k-|a|+1)]\\
    &\leq o(1) + \frac{2}{n^2} \sum_{i,j\in[n]} \sum_{a\ni i,j} \frac{|a|!p_{|a|}w_a}{\ell_n^{a-1}},\nonumber
\end{align}
where once again we denote $w_a =\prod_{v\in a} w_v$. Since
\begin{align}
    \sum_{a\ni i,j} \frac{|a|!p_{|a|}w_a}{\ell_n^{a-1}} &= \sum_{l=2}^{\infty} \sum_{v_1<...<v_{l-2}} \frac{l!p_{l}w_iw_jw_a}{\ell_n^{l-1}}\\
    &= \frac{w_iw_j}{\ell_n}\sum_{l=2}^{\infty} l(l-1)p_l \sum_{v_1,...,v_{l-2}} \frac{w_{v_1}...w_{v_l}}{\ell_n^{l-2}} = \frac{w_iw_j}{\ell_n} (\mu_{(2)}-\mu),\nonumber
\end{align}
we obtain 
\begin{align}
    \sum_{k\geq0}\mathrm{Var}(Q_k^{(n)}) &\leq o(1) + \frac{2}{n^2} \sum_{i,j\in[n]} \frac{w_iw_j}{\ell_n} (\mu_{(2)}-\mu)\\
    &= o(1) + (\mu_{(2)}-\mu)\frac{1}{n^2} \sum_{i\in[n]} \frac{w^2_i}{\ell_n} = o(1) + \frac{(\mu_{(2)}-\mu)\mathbf{E}[W_n^2]}{n^2}=o(1).\nonumber
\end{align}
\end{proof}

\begin{remark}[Eliminating conditions on higher moments by weight truncation] \label{appx_eliminating_cond_c}
Note that Theorem \ref{two_models_relation} is only valid for the $\bgrg$ with weights $(w_i)_{i\in[n]}$ satisfying Condition \ref{cond_weights}(a)-(c), as it requires a finite second moment of the degree of a uniformly chosen vertex. However, we argue that the local convergence statement can easily be extended to the $\bgrg$ not satisfying the latter via a \emph{truncation argument}. To do so, we adapt a similar argument for $\mathrm{GRG}_n(\bm{w})$. (See the proof of \textup{\cite[Theorem 4.23]{Hofstad2023}}). Namely, we can truncate the weights of all vertices by some $K>0$, i.e., introduce $\mathrm{BGRG}^{(K)}_n(\bm{w})$ with weights $\big(w_i^{(K)}\big)_{i\in[n]}$ such that
\begin{align}
    w_i^{(K)} = w_i \wedge K.
\end{align}
Note that if $(w_i)_{i\in[n]}$ in $\bgrg$ satisfies Conditions \ref{cond_weights}(a)-(b), then $\big(w_i^{(K)}\big)_{i\in[n]}$ in $\mathrm{BGRG}^{(K)}_n(\bm{w})$ satisfies Conditions \ref{cond_weights}(a)-(c). Hence, the second moments of $D^{(l),(K)}_n$ and $D^{(r),(K)}_n$ are finite and all results on $\bcm$ can be transferred to $\mathrm{BGRG}^{(K)}_n(\bm{w})$ thanks to Theorem \ref{two_models_relation}. This means that $(\mathrm{BGRG}^{(K)}_n(\bm{w}),V_n^{(l)})$ converges locally in probability to some limiting $(G,o)$. We now show that this implies that also $(\bgrg, V_n^{(l)})$ converges locally in probability to $(G,o)$. By local convergence, for any fixed rooted graph $(H_{\star},o')$ and $r\in\mathbf{N}$,
\begin{align}
    \frac{1}{n} \sum_{i\in[n]} \mathds{1}_{\{B_r(G^{(K)}_n,i) \simeq (H_{\star},o')\}} \stackrel{\mathbf{P}}{\longrightarrow} \hspace{0.1cm}& \mathbf{P}(B_r(G,o) \simeq (H_{\star},o')).
\end{align}
We write
\begin{align}
    \frac{1}{n} \sum_{i\in[n]} \mathds{1}_{\{B_r(G_n,i) \simeq (H_{\star},o')\}} &= \frac{1}{n} \sum_{i\in[n]} \mathds{1}_{\{B_r(G_n,i) \simeq (H_{\star},o'), \text{$B_r(G_n,i)$ contains only $j$ such that $w_j\leq K$}\}}\\
    &+ \frac{1}{n} \sum_{i\in[n]} \mathds{1}_{\{B_r(G_n,i) \simeq (H_{\star},o'),\text{$B_r(G_n,i)$ contains $j$ with $w_j>K$}\}}, \nonumber
\end{align}
and note that
\begin{align}
    \lim_{K\to\infty}\mathbf{E}\Bigg[\frac{1}{n} \sum_{i\in[n]} \mathds{1}_{\{B_r(G_n,i) \simeq (H_{\star},o'),\text{$B_r(G_n,i)$ contains $j$ with $w_j>K$}\}}\Bigg]&\\
    \leq \lim_{K\to\infty}\mathbf{P}(\text{$B_r(G_n,V_n^{(l)})$ contains $j:w_j>K$}) = o(1),& \nonumber
\end{align}
by Conditions \ref{cond_weights}(a)-(b). Thus, for $K$ large enough,
\begin{align}
    \frac{1}{n} \sum_{i\in[n]} \mathds{1}_{\{B_r(G_n,i) \simeq (H_{\star},o')\}} &\stackrel{\text{d}}{=} \frac{1}{n} \sum_{i\in[n]} \mathds{1}_{\{B_r(G_n,i) \simeq (H_{\star},o'), \text{$B_r(G_n,i)$ contains only $j:w_j\leq K$}\}} + o_{\mathbf{P}}(1)\\
    &=\frac{1}{n} \sum_{i\in[n]} \mathds{1}_{\{B_r(G^{(K)}_n,i) \simeq (H_{\star},o')\}}+o_{\mathbf{P}}(1) \stackrel{\mathbf{P}}{\longrightarrow} \hspace{0.1cm} \mathbf{P}(B_r(G,o) \simeq (H_{\star},o')).\nonumber
\end{align}
Having shown local convergence of $\bgrg$, we transfer the result on the giant component in a similar manner. Denote the size of the giant component in $\mathrm{BGRG}^{(K)}_n(\bm{w})$ by $|\mathscr{C}^{(K)}_{\max}|$ and the size of the giant component in $\bgrg$ by $|\mathscr{C}_{\max}|$. Transferring results from \cite{Hofstad2022} we obtain, as $n\to\infty$,
\begin{align}
    \frac{|\mathscr{C}^{(K)}_{\max}|}{n} \stackrel{\mathbf{P}}{\longrightarrow} \xi.
\end{align}
Naturally,
\begin{align} \label{giant_truncated_lower_bound}
    \frac{|\mathscr{C}_{\max}|}{n} \geq \frac{|\mathscr{C}^{(K)}_{\max}|}{n} \stackrel{\mathbf{P}}{\longrightarrow} \xi = \mu\big(|\mathscr{C}(o)|=\infty\big).
\end{align}
On the other hand, denote $Z_{\geq k} = \frac{1}{n}\sum_{i\in[n]}\mathds{1}_{\{|\mathscr{C}(i)|\geq k\}}$ and note that on the event $\{Z_{\geq k} \geq 1\}$,
\begin{align}
    \frac{|\mathscr{C}_{\max}|}{n} \leq \frac{Z_{\geq k}}{n}.
\end{align}
By local convergence in probability,
\begin{align}
    \frac{1}{n}Z_{\geq k} = \mathbf{E}\bigg[\mathds{1}_{\big\{|\mathscr{C}(V^{(l)}_n)|\geq k\big\}} \mid  G_n\bigg] \stackrel{\mathbf{P}}{\longrightarrow} \xi_{\geq k} = \mu\big(|\mathscr{C}(o)|\geq k\big),
\end{align}
Note that $\lim_{k\to\infty}\xi_{\geq k} = \xi$. Thus, for every $\varepsilon>0$,
\begin{align}
    \limsup_{n\to\infty} \mathbf{P}\bigg(\frac{|\mathscr{C}_{\max}|}{n} \leq \xi + \varepsilon \bigg) &\leq \lim_{n\to\infty} \mathbf{P}\bigg(\frac{1}{n}Z_{\geq k} \geq \xi+\varepsilon\bigg) \\
    &\leq \lim_{k\to\infty}\lim_{n\to\infty} \mathbf{P}\bigg(\frac{1}{n}Z_{\geq k} \geq \xi_{\geq k} + \varepsilon\bigg) = 0,\nonumber
\end{align}
and the desired statement follows. Hence indeed, the most important results that we treat in this paper, i.e., local convergence and the existence of giant component, are also true for $\bgrg$ only satisfying Conditions \ref{cond_weights}(a)-(b).
\end{remark}

\subsection{Static local convergence of \texorpdfstring{$\bgrg$}{} and \texorpdfstring{$\drig$}{}}  \label{appx_static_sec_loc_conv}

\noindent \textbf{Local convergence of $\bcm$.} Having verified that $\bgrg$ fulfils the necessary regularity conditions we can now conveniently transfer results on the local convergence from \cite{Hofstad2018}. For the comfort of the reader, we quote the statement of the original result on the local convergence of $\bcm$ (\textup{\cite[Theorem 2.14]{Hofstad2018}}), which states that under Condition \ref{reg_cond}, as $n\to\infty$, $(\mathrm{BCM}_n,V^b_n)$ converges locally in probability to $(\mathrm{BP}_{\gamma},0)$.
The limiting object - $(\mathrm{BP}_{\gamma},0)$ - is a mixture of two branching processes with the root $0$. The two processes are needed because of the bipartite structure of $\bcm$ - each of them corresponds to contributions made to the limit by left- and right-vertices respectively. As the structures of local limits of $\bcm$ and $\bgrg$ are very similar, we do not describe the first one in more detail and direct the reader to Section \ref{sec__main_results_stationary} where the latter is explained.\\

\noindent \textbf{Local convergence of $\mathrm{RIGC}$.} In \cite{Hofstad2018} the local convergence of the resulting intersection graph is a consequence of the convergence of the underlying bipartite graph. Thus, the same will take place for our model.\\

Again for the reader's convenience, we first quote the statement of the original result from \cite{Hofstad2018} (see \textup{\cite[Theorem 2.8]{Hofstad2018}}: Under Condition \ref{reg_cond}, as $n\to\infty$, $(\mathrm{RIGC}_n,V^l_n)$ converges locally in probability to $(\mathrm{CP},o)$. As we already mentioned, the convergence of $\mathrm{RIGC}$ follows from the convergence of $\bcm$. Therefore, the limiting object $(\mathrm{CP},o)$ is a community projection of the limiting object $(\mathrm{BP}_{\gamma},0)$ just like $\mathrm{RIGC}_n$ is a community projection of $\bcm$ (see (2.2) in \cite{Hofstad2018}).\\

\noindent \textbf{Static local limit of $\bgrg$ and $\drig$.} We proceed by stating the local limit of $\bgrg$ and $\drig$.

\locconvstatic*

\begin{proof}
Thanks to Theorem \ref{two_models_relation} and the fact that $\bgrg$ fulfils Condition \ref{reg_cond}(i), we can transfer \textup{\cite[Theorem 2.8]{Hofstad2018}} to $\bgrg$, obtaining the above. The limiting left- and right-degrees $D^{(l)}$ and $D^{(r)}$ are the ones derived in the verification of Condition \ref{reg_cond}(i).\\

The result for $\drig$ is equivalent to \textup{\cite[Theorem 2.8]{Hofstad2018}} and it is a consequence of the relationship between the resulting $\drig$ and the underlying $\bgrg$. The convergence of intersection graphs is preserved by the community projection that transforms the underlying bipartite structures into them. Hence, naturally, the local limit of $\drig$ is a community projection (see (\ref{com_proj})) of the local limit of $\bgrg$.
\end{proof}

\subsection{Static giant component} \label{appx_static_sec_giant}

\subsubsection{Giant component in \texorpdfstring{$\bgrg$}{}}

In \cite{Hofstad2022} results on the giant component of $\bcm$ are again shown under Condition \ref{reg_cond}. Hence, thanks to Theorem \ref{two_models_relation} we can transfer them to our situation. We start with the giant component of the underlying bipartite graph. For the reader's convenience, we state the original result from \cite{Hofstad2022} adapted to the notation we use in this paper. 

\begin{restatable}[The largest component of the $\mathrm{BCM}$ \textup{\cite[Theorem 2.11]{Hofstad2022}}]{theorem}{giantbcm}\label{giant_bcm} 
Consider $\mathrm{BCM}_n = \mathrm{BCM}(\bm{d}^{(l)},\bm{d}^{(r)})$ under Condition \ref{reg_cond} and further assume that $\Bar{V}_2+\Bar{A}_2<2$, where $\Bar{V}_k =\frac{1}{n}\#\{i\in[n]:d^{(l)}_i=k\}$ and $\Bar{A}_k =\frac{1}{M_n}\#\{a:|a|=k\}$. Under the supercriticality condition $\mathbf{E}[\Tilde{D}^{(l)}]\mathbf{E}[\Tilde{D}^{(r)}]>1$, we have that $\xi_l>0$, $\eta_l<1$ and $\eta_r=G_{\Tilde{D}^{(l)}}(\eta_l)<1$. Then, as $n\to\infty$,
\begin{align}
    \frac{|\mathscr{C}_{1,b} \cap \mathcal{V}^{(l)}|}{n} \stackrel{\mathbf{P}}{\longrightarrow} & \xi_l,\\
    \frac{|\mathscr{C}_{1,b} \cap \mathcal{V}_k^{(l)}|}{n} \stackrel{\mathbf{P}}{\longrightarrow} & p_k(1-\eta^k_l).
\end{align}
\end{restatable}
Given this, we obtain Theorem \ref{giant_bipartite}.

\begin{proof}[Proof of Theorem \ref{giant_bipartite}]
The result is automatically transferred from Theorem \ref{giant_bcm}. We are allowed to do that because in Theorem \ref{two_models_relation} we linked $\bcm$ satisfying Conditions \ref{reg_cond} to $\bgrg$ satisfying Conditions \ref{reg_cond} and we also showed that $\bgrg$ satisfies these regularity conditions earlier in this section. Note that the condition $\Bar{V}_2+\Bar{A}_2<2$ made in \cite{Hofstad2022} is always satisfied in our model (as $V_0>0$) and does not have to be assumed for Theorem \ref{giant_bipartite} to hold.
\end{proof}

\subsubsection{Giant component in \texorpdfstring{$\drig$}{}}
Once again, we want to transfer results from \cite{Hofstad2022}, namely the result on the giant component in the resulting intersection graph. Again for the comfort of the reader, we quote the original result with notation adapted to our convention:

\begin{restatable}[Size of the largest component \textup{\cite[Theorem 2.6]{Hofstad2022}}]{theorem}{giantRIGC} \label{giant_RIGC}
Consider $\drig$ under Condition \ref{reg_cond}, and further assume that $\Bar{V}_2+\Bar{A}_2<2$, where $\Bar{V}_k =\frac{1}{n}\#\{i\in[n]:d^{(l)}_i=k\}$ and $\Bar{A}_k =\frac{1}{M_n}\#\{a:|a|=k\}$. Then, there exists $\eta_l \in[0,1]$, the smallest
solution of the fixed point equation
\begin{align}
    \eta_l = G_{\Tilde{D}^{(r)}}(G_{\Tilde{D}^{(l)}}(\eta_l)),
\end{align}
and $\xi_l= 1-G_{D^{(l)}}(\eta_l) \in[0,1]$ such that
\begin{align} \label{giant_intersect_conv}
    \frac{|\mathscr{C}_1|}{n} \stackrel{\mathbf{P}}{\longrightarrow} \xi_l.
\end{align}
Furthermore, $\xi_l>0$ exactly when
\begin{align} \label{giant_cond1}
    \mathbf{E}[\Tilde{D}^{(l)}]\mathbf{E}[\Tilde{D}^{(r)}] > 1,
\end{align}
In this case, $\mathscr{C}_1$ is unique.
\end{restatable}

Hence, we obtain that our model has a giant component as well when $\frac{\mathbf{E}[W^2](\mu_{(2)}-\mu)}{\mathbf{E}[W]} > 1$:

\begin{proof}[Proof of Theorem \ref{giant_stationary}]
The convergence in (\ref{giant_intersect_conv_our}) is a consequence of Theorem \ref{giant_bipartite} and it can be transferred because of the graph equivalence from Theorem \ref{two_models_relation}. However, since, as in the case of $\mathrm{RIGC}$, the giant component in the resulting graph will exist only if the giant component in the underlying graph exists this result can be also deduced from Theorem \ref{giant_bipartite}. Note again that the condition $\Bar{V}_2+\Bar{A}_2<2$ made in \cite{Hofstad2022} is always satisfied in our model (as $V_0>0$) and does not have to be assumed for Theorem \ref{giant_stationary} to hold. A nice feature of our model is the fact that we have a more explicit form of the limiting left- and right-degrees than the authors of \cite{Hofstad2018} and \cite{Hofstad2022}. Hence, we can derive the equality in (\ref{giant_cond_our}): By definition of the shift random variable $\Tilde{D}^{(r)}$ (see (\ref{def_shift_variable})),
\begin{align}
    \mathbf{E}[\Tilde{D}^{(r)}] &= \sum_{k=1}k\mathbf{P}(\Tilde{D}^{(r)}=k)=\sum_{k=1}k(k+1)\frac{\mathbf{P}(D^{(r)}=k+1)}{\mathbf{E}[D^{(r)}]} = \frac{1}{\mu} \sum_{k=1}k(k+1)p_{k+1}\\
    &=\frac{1}{\mu} \sum_{l=2}(l-1)lp_l = \frac{\mu_{(2)}-\mu}{\mu},\nonumber
\end{align}
where we have used (\ref{conv_unif_a}) and (\ref{conv_1st_mom_a}) when substituting explicit expressions for the probability mass function and expected value of the random variable $D^{(r)}$.
Similarly,
\begin{align} \label{exp_tilde_d_l}
    \mathbf{E}[\Tilde{D}^{(l)}] &=\sum_{k=1}k(k+1)\frac{\mathbf{P}(D^{(l)}=k+1)}{\mathbf{E}[D^{(l)}]}.
\end{align}
We know by Corollary \ref{conv_unif_left} that $D^{(l)}$ is a mixed-Poisson variable with rate $W\mu$. Hence, we need to condition on the weight variable $W$ to compute its expectation:
\begin{align}
    \mathbf{E}[D^{(l)}] &= \mathbf{E}\big[\mathbf{E}(Poi(w\mu)\mid  W\big] = \mathbf{E}[\mu W]=\mu\mathbf{E}[W].
\end{align}
Plugging it into (\ref{exp_tilde_d_l}) yields
\begin{align}
    \mathbf{E}[\Tilde{D}^{(l)}] &=\frac{1}{\mu\mathbf{E}[W]}\sum_{k=1}k(k+1)\mathbf{P}(D^{(l)}=k+1)\\
    &=\frac{1}{\mu\mathbf{E}[W]}\sum_{l=2}(l-1)l \mathbf{E}\big[\mathbf{E}\big[\mathbf{P}(Poi(w\mu)=l)\mid  W\big]\big].\nonumber
\end{align}
After substituting $\mathbf{P}(Poi(w\mu)=l)=\frac{\big(w\mu\big)^le^{-w\mu}}{l!}$ we obtain
\begin{align}
    \mathbf{E}[\Tilde{D}^{(l)}] &= \frac{1}{\mu\mathbf{E}[W]}\sum_{l=2}(l-1)l \mathbf{E}\big[\mathbf{E}\big[\frac{\big(w\mu\big)^le^{-w\mu}}{l!}\mid  W\big]\big]=\frac{1}{\mu\mathbf{E}[W]}\mathbf{E}\bigg[\mathbf{E}\bigg[\sum_{l=2}\frac{\big(W\mu\big)^l}{(l-2)!}e^{-W\mu}\mid  W\bigg]\bigg]\\
    &= \frac{1}{\mu\mathbf{E}[W]}\mathbf{E}\big[\mathbf{E}\big[e^{-W\mu}\cdot e^{W\mu}(W\mu)^2\mid  W\big]\big] = \frac{\mu^2\mathbf{E}[W^2]}{\mu\mathbf{E}[W]}=\frac{\mu\mathbf{E}[W^2]}{\mathbf{E}[W]}. \nonumber
\end{align}
Therefore, the condition (\ref{giant_cond_our}) becomes
\begin{align}
    \frac{\mathbf{E}[W^2](\mu_{(2)}-\mu)}{\mathbf{E}[W]} >1.
\end{align}
\end{proof}

\begin{remark}[Results for $\bgrgrescaled$] \label{appx_results_union_rescaled}
We claim that $\bgrgrescaled$ fulfils the same conditions that guaranteed convergence of the $\bgrg$, i.e., uniformity and regularity conditions of the degree sequences. Indeed, note that by replacing $f(|a|)=|a|!p_{|a|}$ by $f(|a|)=(1+t)|a|!p_{|a|}$, all the proofs from Appendix \ref{appx_static_graph} follow analogously. Hence, we do not explicitly repeat them here. However, for the comfort of the reader we now state the corresponding regularity conditions as the limiting variables are necessary to understand some of our statements on the union graph:
\begin{enumerate}
    \item Denote the degree of a uniformly chosen left-vertex in $\bgrgrescaled$ by $D^{(l),(t)}_n$. Then, as $n\to\infty$,
\begin{align} 
    D^{(l),(t)}_n \to D^{(l),(t)} \hspace{0.2cm}\text{and}\hspace{0.2cm}\mathbf{E}[D^{(l),(t)}_n\mid  G_n] \stackrel{\mathbf{P}}{\longrightarrow} \mu(1+t)\mathbf{E}[W],
\end{align}
where $D^{(l),(t)}$ is a Poisson variable with parameter $W\mu(1+t)$, with $\mu = \sum_k kp_k$.
    \item Denote the degree of a uniformly chosen right-vertex in $\bgrgrescaled$ by $D_n^{(r),(t)}$ . Then, as $n\to\infty$,
\begin{align}\label{conv_unif_a_union}
    \mathbf{P}(D^{(r),(t)}_n = k\mid  G_n) \stackrel{\mathbf{P}}{\longrightarrow} p_k \hspace{0.2cm}\text{and}\hspace{0.2cm} \mathbf{E}[D^{(r),(t)}_n\mid  G_n] \stackrel{\mathbf{P}}{\longrightarrow} \mu.
\end{align}
\end{enumerate}
Naturally, the above automatically transfers to $\bgrgunion$, as it is asymptotically equivalent to $\bgrgrescaled$.
\end{remark}

\section{Convergence of processes with jumps} \label{appx_jumps_conv}
To prove dynamic results for processes of giant membership and size of the largest group we apply a well-known criterion for convergence of processes in Skorokhod $J_1$ topology, that we now quote with a minor change of replacing $D[0,1]$ with $D[0,t]$:
\begin{restatable}[\textup{\cite[Thm 13.3]{Billingsley2013}}]{theorem}{convSkorokhod} \label{conv_Skorokhod}
Assume a sequence of processes $\big(X_n(s)\big)_{s\in[0,t]}$ and a process $\big(\mathcal{X}(s)\big)_{s\in[0,t]}$ in $D[0,t]$, equipped with the $d_0$ metric, satisfy the following conditions:
\begin{enumerate}[(i)]
    \item For all $\{s_1,...,s_k\}\in[0,t]: \big(X_n(s_1),...,X_n(s_k)\big) \stackrel{d}{\longrightarrow} \big(\mathcal{X}(s_1),...,\mathcal{X}(s_k)\big)$ as $n\to\infty$.
    \item $\mathcal{X}(t) - \mathcal{X}(t-\delta) \stackrel{\mathbf{P}}{\longrightarrow} 0$ as $\delta \to 0$.
    \item For every $\varepsilon,\eta>0$ there exists $n_0\geq1$ and $\delta>0$ such that for all $n\geq n_0$
    \begin{align}
        \mathbf{P}\Bigg( \sup_{s,s_1,s_2: s\in[s_1,s_2],s_2-s_1<\delta} \min \Big(\Big|X_n(s)-X_n(s_1)\Big|,\Big|X_n(s_2)-X_n(s)\Big|\Big) > \varepsilon \Bigg) \leq \eta.
    \end{align}
\end{enumerate}
Then $\big(X_n(s)\big)_{s\in[0,t]} \stackrel{d}{\longrightarrow} \big(\mathcal{X}(s)\big)_{s\in[0,t]}$ as $n\to\infty$ in Skorokhod $J_1$ topology.
\end{restatable}
To extend the notion of local convergence to the dynamic setting we have to take care of the jumps with respect to the local topology, which is less straightforward than the jumps we described before (for instance jumps of indicator processes). However, the fact that the space of rooted graphs equipped with the local topology is Polish allows us to apply the theory of convergence of processes from a compact space to a separable space from \cite{kallenberg2002foundations}. For the comfort of the reader, we now quote the appropriate statements.\\

We start by introducing the notation. Fix two metric spaces $(K,d)$ and $(S, p)$, where $K$ is compact and $S$ is separable and complete, and consider the space $C(K,S)$ of continuous functions from $K$ to $S$, endowed with the uniform metric $\Bar{\rho}(x,y)=\sup_{t\in K} \rho(x_t,y_t)$. The following lemma characterizes conditions necessary for convergence in $C(K,S)$:
\begin{restatable}[\textup{\cite[Lemma 16.2]{kallenberg2002foundations}}]{lemma}{convcompacttoseparable} \label{conv_compact_to_separable}
Let $X,X_1,X_2,...$ be random elements in $C(K,S)$. Then $X_n$ converges weakly to $X$ iff $X_n$ converges to $X$ for all finite-dimensional distributions and $(X_n)$ is relatively compact in distribution.
\end{restatable}
For random elements in sufficiently regular metric spaces, the condition of relative compactness in the above lemma can be replaced by tightness. This key result in the theory of weak convergence is the subject of the following theorem:
\begin{restatable}[\textup{\cite[Theorem 16.3]{kallenberg2002foundations}}]{theorem}{relcompactvstight} \label{rel_compact_vs_tight}
For any sequence of random elements $\xi_1,\xi_2,...$ in a metric space $S$, tightness implies relative compactness in distribution, and the two conditions are equivalent when $S$ is separable and complete.
\end{restatable}

\end{document}